\documentclass[a4paper,12pt]{amsart}

\usepackage[a4paper, margin= 0.9in, twoside=false]{geometry}
\usepackage{fullpage}
\usepackage{mathrsfs}
\usepackage[utf8]{inputenc}
\usepackage[T1]{fontenc}
\usepackage[english]{babel}
\usepackage{amsfonts}
\usepackage{amsthm}
\usepackage{amssymb}
\usepackage{mathcomp}
\usepackage{mathrsfs}
\usepackage[bb=boondox]{mathalfa}
\usepackage{graphicx}
\usepackage{epstopdf}
\usepackage[centertags]{amsmath}
\usepackage[pdftex]{hyperref}
\usepackage{vmargin}
\usepackage{tikz}
\usepackage{tikz-cd}
\usepackage{dsfont}
\usepackage{cases}
\usepackage[capitalize]{cleveref}
\usepackage{mathtools}
\usepackage{shuffle}
\usepackage{adjustbox}
\usepackage[all]{xy}

\makeatletter
\def\section{\@startsection{section}{1}{\z@}%
  {25pt}
  {15pt}
  {\normalfont\scshape\centering}}
\makeatother

\AtBeginDocument{\addtocontents{toc}{\protect\setlength{\parskip}{0pt}}}

\newcommand{\ho}[1]{\mathsf{Ho}(#1)}
\usepackage{ebgaramond}

\allowdisplaybreaks
\definecolor{corail}{rgb}{0.9882,0.4627,0.4157} 
\definecolor{viola}{RGB}{166,146,186}
\definecolor{mocha}{RGB}{164,120,100}
\definecolor{peachfuzz}{rgb}{0.8, 0.5451, 0.3961}
\definecolor{clouddancer}{RGB}{240, 238, 233}
\definecolor{bluefusion}{RGB}{83, 104, 124}

\usepackage{euscript}
\renewcommand{\mathcal}{\EuScript}
\usepackage{enumerate}

\hypersetup{
    colorlinks,
    linkcolor= bluefusion,
    citecolor= bluefusion,
    urlcolor= bluefusion,
}

\title{Point-set models for homotopy coherent coalgebras}
\author{ Dan Petersen, \quad Victor Roca i Lucio, \quad Sinan Yalin}

\address{Dan Petersen, Stockholm University, SE-106 91 Stockholm, Sweden}
\email{\href{mailto:dan.petersen@math.su.se}{dan.petersen@math.su.se}}

\address{Victor Roca i Lucio, Université Paris Cité and Sorbonne Université, CNRS, IMJ-PRG, F-75013 Paris, France}
\email{\href{mailto:rocalucio@imj-prg.fr}{rocalucio@imj-prg.fr}}

\address{Sinan Yalin, Université d'Angers, LAREMA, 49045 Angers, France}
\email{\href{mailto:sinan.yalin@univ-angers.fr}{sinan.yalin@univ-angers.fr}}

\date{\today}
\thanks{The first named author was supported by the National Science Foundation
under Grant No.~DMS-2424441, and a Wallenberg Scholar fellowship. The second named author was supported by the ANR project SHoCos ANR-22-CE40-0008. The third named author acknowledges the support of the ANR projects LieDG ANR-24-CE40-5367 and KAsH ANR-25-CE40-2861.}

\subjclass[2020]{18N70, 18M60, 18N40, 18N60, 55U15, 55U35, 16T15}

\newtheorem{theoremintro}{Theorem}
\renewcommand*{\thetheoremintro}{\Alph{theoremintro}}

\DeclareMathOperator*{\colim}{colim}

\newcommand{\quo}[1]{\text{"}\!#1\!\text{"}\,\,}

\begin{document}
\pagecolor{clouddancer}


    
\newcommand{\baseonecat}{\mathsf{Ch}(\kk)}


\newcommand{\oneop}{\mathcal{P}}
\newcommand{\onecoop}{\mathcal{C}}

    

\newcommand{\onepalg}{\mathcal{P}\text{-}\mathsf{alg}}
\newcommand{\onealg}[1]{#1\text{-}\mathsf{alg}}
\newcommand{\onecoalg}[1]{#1\text{-}\mathsf{coalg}}
\newcommand{\onepcoalg}{\mathcal{P}\text{-}\mathsf{coalg}}

    
\newcommand{\oneCalg}{\mathsf{dg}~\mathcal{C}\text{-}\mathsf{alg}}
\newcommand{\oneCcoalg}{\mathsf{dg}~\mathcal{C}\text{-}\mathsf{coalg}}
    

\newcommand{\basecat}{\textbf{D}(R)}


\newcommand{\iop}{\mathscr{P}}
\newcommand{\icoop}{\mathscr{C}}

    
\newcommand{\ipalg}{\textbf{Alg}_{\mathscr{P}}(\textbf{D}(R))}
\newcommand{\ipcoalg}{\textbf{Coalg}_{\mathscr{P}}(\textbf{D}(R))}

    
\newcommand{\iCalg}{\textbf{Alg}_{\mathscr{C}}(\textbf{D}(R))}
\newcommand{\iCcoalg}{\textbf{Coalg}_{\mathscr{C}}(\textbf{D}(R))}

\theoremstyle{plain}
\newtheorem{theorem}{Theorem}[section]
\newtheorem{theorem1}{Theorem}
\newtheorem{definition1}{Definition}

\newtheorem{lemma}[theorem]{Lemma}
\newtheorem{expectation}[theorem]{Expectation}
\newtheorem{proposition}[theorem]{Proposition}
\newtheorem{assumption}[theorem]{Assumption}
\newtheorem{corollary}[theorem]{Corollary}
\newtheorem{conjecture}[theorem]{Conjecture}

\renewcommand*{\thetheoremintro}{\Alph{theoremintro}}

\theoremstyle{definition}
\newtheorem{definition}[theorem]{Definition}

\theoremstyle{remark}
\newtheorem{notation}[theorem]{\sc Notation}
\newtheorem{remark}[theorem]{\sc Remark}
\newtheorem{example}[theorem]{\sc Example}

\newcommand{\draftnote}[1]{\marginpar{\raggedright\textsf{\hspace{0pt} \tiny #1}}}
\newcommand{\ac}{{\scriptstyle \text{\rm !`}}}
\newcommand{\Ch}{\categ{Ch}}
\newcommand{\dgmod}{\Ch}

\newcommand{\qi}{\xrightarrow{ \,\smash{\raisebox{-0.65ex}{\ensuremath{\scriptstyle\sim}}}\,}}
\newcommand{\lqi}{\xleftarrow{ \,\smash{\raisebox{-0.65ex}{\ensuremath{\scriptstyle\sim}}}\,}}

\newcommand{\Cate}{\categ{Cat}_{\categ E}}
\newcommand{\Catesmall}{\categ{Cat}_{\categ E, \mathrm{small}}}
\newcommand{\Operade}{\Operad_{\categ E}}
\newcommand{\Operadecprime}{\Operad_{\categ E}}
\newcommand{\Operadesmall}{\Operad_{\categ E, \mathrm{small}}}
\newcommand{\eII}{\mathcal{I}}
\newcommand{\ecateg}[1]{\mathcal{#1}}
\newcommand{\cmonlax}{\categ{CMon}_\lax}
\newcommand{\cmonoplax}{\categ{CMon}_\oplax}
\newcommand{\cmonstrong}{\categ{CMon}_\strong}
\newcommand{\cmonstrict}{\categ{CMon}_\strict}
\newcommand{\cat}{\mathrm{cat}}
\newcommand{\lax}{\mathrm{lax}}
\newcommand{\oplax}{\mathrm{oplax}}
\newcommand{\strong}{\mathrm{strong}}
\newcommand{\strict}{\mathrm{strict}}
\newcommand{\pl}{\mathrm{pl}}
\newcommand{\Comonads}{\mathsf{Comonads}}
\newcommand{\Monads}{\mathsf{Monads}}
\newcommand{\Cats}{\mathsf{Cats}}
\newcommand{\Functors}{\mathsf{Functors}}
\newcommand{\Mor}{\mathsf{Mor}}
\newcommand{\Nat}{\mathsf{Nat}}
\newcommand{\sk}{\mathrm{sk}}
\newcommand{\TwoFun}{\text{2-Fun}}
\newcommand{\Ob}{\mathrm{Ob}}
\newcommand{\tr}{\mathrm{tr}}
\newcommand{\catch}{\mathsf{Ch}}
\newcommand{\categ}[1]{\mathsf{#1}}
\newcommand{\set}[1]{\mathrm{#1}}
\newcommand{\catoperad}[1]{\mathsf{#1}}
\newcommand{\operad}[1]{\mathcal{#1}}
\newcommand{\algebra}[1]{\mathrm{#1}}
\newcommand{\coalgebra}[1]{\mathrm{#1}}
\newcommand{\cooperad}[1]{\mathcal{#1}}
\newcommand{\ocooperad}[1]{\overline{\mathcal{#1}}}
\newcommand{\catofmod}[1]{{#1}\mathrm{-}\mathsf{mod}}
\newcommand{\catofcog}[1]{#1\mathrm{-}\mathsf{cog}}
\newcommand{\catcog}[1]{\categ{dg}\text{-}{#1}\text{-}\categ{cog}}
\newcommand{\catalg}[1]{\categ{dg}\text{-}{#1}\text{-}\categ{alg}}
\newcommand{\catofcolcomod}[1]{\mathsf{Col}\mathrm{-}{#1}\mathrm{-}\mathsf{comod}}
\newcommand{\catofcoalgebra}[1]{{#1}\mathrm{-}\mathsf{cog}}
\newcommand{\catofalg}[1]{\operad{#1}\mathrm{-}\mathsf{alg}}
\newcommand{\catofalgebra}[1]{{#1}\mathrm{-}\mathsf{alg}}
\newcommand{\mbs}{\mathsf{S}}
\newcommand{\catocol}[1]{\mathsf{O}_{\set{#1}}}
\newcommand{\catoftrees}{\mathsf{Trees}}
\newcommand{\cattcol}[1]{\catoftrees_{\set{#1}}}
\newcommand{\catcorcol}[1]{\mathsf{Corol}_{\set{#1}}}
\newcommand{\Einfty}{\mathcal{E}_{\infty}}
\newcommand{\nuEinfty}{\mathcal{nuE}_{\infty}}
\newcommand{\Fun}[3]{\mathrm{Fun}^{#1}\left(#2,#3\right)}
\newcommand{\III}{\operad{I}}
\newcommand{\treeoperad}{\mathbb{T}}
\newcommand{\treemodule}{\mathbb{T}}
\newcommand{\core}{\mathrm{Core}}
\newcommand{\forget}{\mathrm{U}}
\newcommand{\treemonad}{\mathbb{O}}
\newcommand{\cogcomonad}[1]{\mathbb{L}^{#1}}
\newcommand{\cofreecog}[1]{\mathrm{L}^{#1}}

\newcommand{\barfunctor}[1]{\mathrm{B}_{#1}}
\newcommand{\baradjoint}[1]{\mathrm{B}^\dag_{#1}}
\newcommand{\cobarfunctor}[1]{\mathrm{C}_{#1}}
\newcommand{\cobaradjoint}[1]{\mathrm{C}^\dag_{#1}}
\newcommand{\Operad}{\mathsf{Operad}}
\newcommand{\coOperad}{\mathsf{coOperad}}

\newcommand{\Aut}[1]{\mathrm{Aut}(#1)}

\newcommand{\verte}[1]{\mathrm{vert}(#1)}
\newcommand{\edge}[1]{\mathrm{edge}(#1)}
\newcommand{\leaves}[1]{\mathrm{leaves}(#1)}
\newcommand{\inner}[1]{\mathrm{inner}(#1)}
\newcommand{\inp}[1]{\mathrm{input}(#1)}

\newcommand{\field}{\mathbb{K}}
\newcommand{\mbk}{\mathbb{K}}
\newcommand{\mbn}{\mathbb{N}}

\newcommand{\id}{\mathrm{Id}}
\newcommand{\ii}{\mathrm{id}}
\newcommand{\unit}{\mathds{1}}

\newcommand{\Lin}{Lin}

\newcommand{\BijC}{\mathsf{Bij}_{C}}

\newcommand{\kk}{\Bbbk}
\newcommand{\PP}{\mathcal{P}}
\newcommand{\C}{\mathcal{C}}
\newcommand{\Sy}{\mathbb{S}}
\newcommand{\Tree}{\mathsf{Tree}}
\newcommand{\treemod}{\mathbb{T}}
\newcommand{\Dend}{\Omega}
\newcommand{\aDend}{\Omega^{\mathsf{act}}}
\newcommand{\cDend}{\Omega^{\mathsf{core}}}
\newcommand{\cDendpart}{\cDend_{\mathsf{part}}}

\newcommand{\build}{\mathrm{Build}}
\newcommand{\col}{\mathrm{col}}

\newcommand{\HOM}{\mathrm{HOM}}
\newcommand{\Hom}[3]{\mathrm{hom}_{#1}\left(#2 , #3 \right)}
\newcommand{\ov}{\overline}
\newcommand{\otimeshadamard}{\otimes_{\mathbb{H}}}

\newcommand{\Aa}{\mathcal{A}}
\newcommand{\BB}{\mathcal{B}}
\newcommand{\CC}{\mathcal{C}}
\newcommand{\DD}{\mathcal{D}}
\newcommand{\EE}{\mathcal{E}}
\newcommand{\FF}{\mathcal{F}}
\newcommand{\II}{\mathbb{1}}
\newcommand{\RR}{\mathcal{R}}
\newcommand{\UU}{\mathcal{U}}
\newcommand{\VV}{\mathcal{V}}
\newcommand{\WW}{\mathcal{W}}
\newcommand{\AAA}{\mathscr{A}}
\newcommand{\BBB}{\mathscr{B}}
\newcommand{\CCC}{\mathscr{C}}
\newcommand{\DDD}{\mathscr{D}}
\newcommand{\EEE}{\mathscr{E}}
\newcommand{\FFF}{\mathscr{F}}

\newcommand{\PPP}{\mathscr{P}}
\newcommand{\QQQ}{\mathscr{Q}}

\newcommand{\QQ}{\mathcal{Q}}

\newcommand{\KKK}{\mathscr{K}}
\newcommand{\KK}{\mathcal{K}}

\newcommand{\ra}{\rightarrow}

\newcommand{\Ai}{\mathcal{A}_{\infty}}
\newcommand{\uAi}{u\mathcal{A}_{\infty}}
\newcommand{\uEinfty}{u\mathcal{E}_{\infty}}
\newcommand{\uAW}{u\mathcal{AW}}

\newcommand{\uAlg}{\mathsf{Alg}}
\newcommand{\nuAlg}{\mathsf{nuAlg}}
\newcommand{\cAlg}{\mathsf{cAlg}}

\newcommand{\ucAlg}{\mathsf{ucAlg}}
\newcommand{\Cog}{\mathsf{Cog}}
\newcommand{\nuCog}{\mathsf{nuCog}}
\newcommand{\uAWcog}{u\mathcal{AW}-\mathsf{cog}}

\newcommand{\uCog}{\mathsf{uCog}}
\newcommand{\cCog}{\mathsf{cCog}}
\newcommand{\ucCog}{\mathsf{ucCog}}
\newcommand{\cNilCog}{\mathsf{cNilCog}}
\newcommand{\ucNilCog}{\mathsf{ucNilCog}}
\newcommand{\NilCog}{\mathsf{NilCog}}

\newcommand{\Cocom}{\mathsf{Cocom}}
\newcommand{\uCocom}{\mathsf{uCocom}}
\newcommand{\NilCocom}{\mathsf{NilCocom}}
\newcommand{\uNilCocom}{\mathsf{uNilCocom}}
\newcommand{\Liealg}{\mathsf{Lie}-\mathsf{alg}}
\newcommand{\cLiealg}{\mathsf{cLie}-\mathsf{alg}}
\newcommand{\Alg}{\mathsf{Alg}}
\newcommand{\Linfty}{\mathcal{L}_{\infty}}
\newcommand{\CMC}{\mathfrak{CMC}}
\newcommand{\Tfree}{\mathbb{T}}

\newcommand{\Hinich}{\mathsf{Hinich} -\mathsf{cog}}

\newcommand{\Ccomod}{\mathscr C -\mathsf{comod}}
\newcommand{\Pmod}{\mathscr P -\mathsf{mod}}

\newcommand{\cCoop}{\mathsf{cCoop}}

\newcommand{\Set}{\mathsf{Set}}
\newcommand{\sSet}{\mathsf{sSet}}
\newcommand{\dgMod}{\mathsf{dgMod}}
\newcommand{\gMod}{\mathsf{gMod}}
\newcommand{\catOrd}{\mathsf{Ord}}
\newcommand{\catBij}{\mathsf{Bij}}
\newcommand{\catSmod}{\mbs\mathsf{mod}}
\newcommand{\EEtw}{\mathcal{E}\text{-}\mathsf{Tw}}
\newcommand{\OpBim}{\mathsf{Op}\text{-}\mathsf{Bim}}

\newcommand{\Palg}{\mathcal{P}-\mathsf{alg}}
\newcommand{\Qalg}{\mathcal{Q}-\mathsf{alg}}
\newcommand{\Pcog}{\mathcal{P}-\mathsf{cog}}
\newcommand{\Qcog}{\mathcal{Q}-\mathsf{cog}}
\newcommand{\Ccog}{\mathcal{C}-\mathsf{cog}}
\newcommand{\Dcog}{\mathcal{D}-\mathsf{cog}}
\newcommand{\uCoCog}{\mathsf{uCoCog}}

\newcommand{\Artinalg}{\mathsf{Artin}-\mathsf{alg}}

\newcommand{\Def}{\mathrm{Def}}
\newcommand{\Bij}{\mathrm{Bij}}
\newcommand{\op}{\mathrm{op}}

\newcommand{\undern}{\underline{n}}
\newcommand{\dginterval}{{N{[1]}}}
\newcommand{\dgsimplex}[1]{{N{[#1]}}}

\newcommand{\cofree}{ T^c}
\newcommand{\Tw}{ Tw}
\newcommand{\End}{\mathcal{E}\mathrm{nd}}
\newcommand{\catEnd}{\mathsf{End}}
\newcommand{\coEnd}{\mathrm{co}\End}
\newcommand{\Mult}{\mathrm{Mult}}
\newcommand{\coMult}{\mathrm{coMult}}

\newcommand{\Lie}{\mathcal{L}\mathr{ie}}
\newcommand{\As}{\mathcal{A}\mathrm{s}}
\newcommand{\uAs}{\mathrm{u}\As}
\newcommand{\coAs}{\mathrm{co}\As}
\newcommand{\Com}{\mathcal{C}\mathrm{om}}
\newcommand{\uCom}{\mathrm{u}\Com}
\newcommand{\Perm}{\catoperad{Perm}}
\newcommand{\uBE}{\mathrm{u}\mathcal{BE}}
\newcommand{\uBEs}{{\uBE}^{\mathrm s}}

\newcommand{\PD}{\mathrm{PD}}
\newcommand{\abs}{\mathrm{abs}}

\newcommand{\itemt}{\item[$\triangleright$]}
\newcommand{\gr}{\mathrm{gr}}

\newcommand{\poubelle}[1]{}

\newcommand{\catdgalg}[1]{\mathsf{dg}~{#1}\text{-}\mathsf{alg}}
\newcommand{\catdgcog}[1]{\mathsf{dg}~{#1}\text{-}\mathsf{cog}}
\newcommand{\Map}[3]{\mathrm{Map}_{#1}\left(#2, #3\right)}
\newcommand{\Ab}{\mathrm{Ab}}
\newcommand{\Res}{\mathrm{Res}}
\newcommand{\FMP}{\mathsf{FMP}}
\newcommand{\comp}{\circ}
\newcommand{\complete}{\mathsf{comp}}
\newcommand{\Psh}{\mathsf{Psh}}
\newcommand{\Cell}{\mathrm{Cell}}
\newcommand{\Art}{\mathrm{Art}}
\newcommand{\CoArt}{\mathsf{CoArt}}
\newcommand{\invqis}{{\left[\mathrm{Q.iso}^{-1}\right]}}
\newcommand{\invw}{{\left[W^{-1}\right]}}
\newcommand{\catcellalg}[1]{\Cell~{#1}\text{-}\mathsf{alg}}
\newcommand{\catartalg}[1]{\Art~{#1}\text{-}\mathsf{alg}}
\newcommand{\catcoartcog}[1]{\CoArt~{#1}\text{-}\mathsf{cog}}

\newcommand{\Victor}[1]{\textcolor{red}{#1}}
\newcommand{\Sinan}[1]{\textcolor{blue}{#1}}
\newcommand{\Dan}[1]{\textcolor{magenta}{#1}}

\begin{abstract}We show a first \emph{rectification} result for homotopy chain coalgebras over a field. On the one hand, we consider the $\infty$-category obtained by localizing  differential graded coalgebras over an operad with respect to quasi-isomorphisms; on the other, we give a general definition of an $\infty$-category of coalgebras over an enriched $\infty$-operad. We show by induction over cell attachments that these two $\infty$-categories are in fact equivalent when the operad is cofibrant. This yields explicit point-set models for $\mathbb{E}_n$-coalgebras and $\mathbb{E}_\infty$-coalgebras in the derived $\infty$-category of chain complexes over a field, and an explicit point-set model for the cellular chains functor with its $\mathbb{E}_\infty$-coalgebra structure. After Bachmann--Burklund, this gives a point-set algebraic model for nilpotent $p$-adic homotopy types. 
\end{abstract}

\maketitle

\tableofcontents

\section{Introduction}

\subsection{Rectification of algebras} Let $\mathbf C$ be an $\infty$-category. By a \emph{point-set model} of $\mathbf C$, we mean a 1-category $\mathsf C$ and a class of weak equivalences $\mathsf W \subset \mathrm{mor}(\mathsf C)$, such that the localization of $\mathsf C$ at $\mathsf W$ presents $\mathbf C$ --- that is, such that there is an equivalence of $\infty$-categories $\mathbf C \simeq \mathsf C[\mathsf W^{-1}]$. 

\medskip

Most $\infty$-categories encountered in day-to-day life are naturally presented to us in terms of point-set models. Having a point-set model can be both conceptually clarifying, and also serve as a useful computational tool: in particular, if $(\mathsf C,\mathsf W)$ comes as part of a model category structure, then one obtains a powerful calculus for e.g.\ computing limits and colimits in $\mathbf C$, mapping spaces, or derived functors. 

\medskip

Many constructions in higher category theory allow one to build new $\infty$-categories out of existing ones: localization, stabilization, sheaf categories, the Lurie tensor product, etc. When the existing $\infty$-categories are presented by point-set models, it is natural to ask for a point-set model for the result of the construction. Here is an important example. Suppose that $(\mathbf C, \otimes)$ is a {monoidal} $\infty$-category. Let $\mathrm{Alg}(\mathbf C)$ denote the $\infty$-category of $\mathbb E_1$-algebras in $\mathbf C$. If $\mathbf C \simeq \mathsf C[\mathsf W^{-1}]$, can we write down a point-set model of $\mathrm{Alg}(\mathbf C)$ in terms of $\mathsf C$ and $\mathsf W$? Under certain hypotheses on $\mathsf C$ and $\mathsf W$, the answer is yes:

\begin{theorem}[Lurie] \label{lurie theorem} Suppose that $(\mathsf C,\mathsf W)$ underlies a combinatorial monoidal model category, which satisfies either: (a) all objects are cofibrant, (b) it is left proper, symmetric monoidal, satisfies the monoid axiom, and is cofibrantly generated by cofibrations between cofibrant objects. Then
	\begin{equation}
	    \label{eq: lurie equivalence}\mathrm{Alg}\big(\mathsf C[\mathsf W^{-1}]\, \big) \simeq \mathrm{Alg}(\mathsf C)~[\mathsf W^{-1}]~,
	\end{equation} 
	where on the left we mean the $\infty$-category of $\mathbb E_1$-algebras in $\mathsf C[\mathsf W^{-1}]~;$ on the right, we mean the 1-category of monoid objects in $\mathsf C$, localized at the class of arrows whose underlying morphism is in the class $\mathsf W$. 	
\end{theorem}

This is a \emph{rectification} theorem: it says that every $\mathbb E_1$-algebra can be rectified to a strictly associative object on the point-set level, and that this rectification is essentially unique. We refer to \cite[Section 4.1.8]{HigherAlgebra} for more details. \cref{lurie theorem} is not the only rectification theorem in the literature: see e.g.\ \cite{haugseng-rectification,HinichRectification,PavlovScholbach18} for extensions to algebras over more general operads and to enriched settings. 

\begin{remark}\label{strategy remark}
    The proofs of all these rectification results are somewhat similar in flavor. One uses the Barr--Beck--Lurie theorem to show that both sides of \eqref{eq: lurie equivalence} are monadic over $\mathbf C$. Then one proves an isomorphism between the two monads: in both cases, the monad is given by the usual formula
\begin{equation}\label{classical free algebra} X \mapsto \coprod_{n \geq 0} X^{\otimes n}.\end{equation}
(On the point-set level this formula is only valid if one restricts $X$ to the full subcategory of cofibrant objects.)
\end{remark}

\subsection{Rectification of coalgebras.} This paper is not about \emph{algebras}, but about \emph{coalgebras}. 
Let us consider again a monoidal $\infty$-category $(\mathbf C,\otimes)$. 

\begin{definition}\label{def:coalgebra}The $\infty$-category of \emph{$\mathbb E_1$-coalgebras} in $\mathbf{C}$ is defined as $\mathrm{Coalg}(\mathbf C) \coloneqq \mathrm{Alg}(\mathbf C^{\mathrm{op}})^{\mathrm{op}}$.
\end{definition} 

Following the above discussion, it is natural to ask: if $\mathbf C \simeq \mathsf C[\mathsf W^{-1}]$, can one write a point-set model of $\mathrm{Coalg}(\mathbf C)$ in terms of $\mathsf C$ and $\mathsf W$, under ``reasonable hypotheses''  on $(\mathsf C,\mathsf W)$? 

\cref{lurie theorem} says that we get a point-set model of $\mathrm{Coalg}(\mathbf C)$ whenever $\mathbf C^{\mathrm{op}}$ satisfies the hypotheses of \cref{lurie theorem}. But these are \emph{not} ``reasonable hypotheses'' to impose. This is because \emph{monoidal model categories} always give rise to \emph{closed monoidal $\infty$-categories}; most monoidal $\infty$-categories encountered in nature are closed, but the opposite of a closed monoidal category is  itself closed only in highly degenerate situations.  Similarly, \emph{combinatorial model categories} always model \emph{presentable $\infty$-categories}; nearly all  $\infty$-categories encountered in nature are presentable, but the opposite of a presentable category is  presentable only in highly degenerate situations.

\begin{remark}\label{cofree are complicated}Let us indicate one reason why the monoidal structure being closed is useful for the strategy described in \cref{strategy remark}. Let $(\mathsf C,\otimes)$ be a monoidal category. The formula \eqref{classical free algebra} for the free algebra in $\mathsf C$ is valid whenever $\mathsf C$ admits countable coproducts, and $-\otimes-$ preserves countable coproducts in both variables. If $\mathsf C$ is closed, then the latter condition is automatic, since $X \otimes -$ is a left adjoint and preserves colimits. By contrast, if $\mathsf C^{\mathrm{op}}$ is closed then $X \otimes -$ is a \emph{right} adjoint. For example, if $\mathsf C = \mathsf{Ab}^{\mathrm{op}}$, then free algebras in $\mathsf C$ are what's classically called \emph{cofree coalgebras}. Cofree coalgebras exist, but their construction is extremely intricate (cf.~\cite{fox,Hazewinkel03}), in particular in comparison to the simple formula \eqref{classical free algebra}.

\end{remark}

Here are some examples indicating the subtleties involved in proving rectification results for coalgebras.

\begin{example}
    Every based suspension is an $\mathbb E_1$-coalgebra in the $\infty$-category of based spaces, with respect to the monoidal structure given by wedge sum. But a coalgebra in the 1-category of topological spaces with respect to wedge sum must be a point. Hence there can be no rectification of coalgebras in this situation. See Remark 2.19. in \cite{Moreno22} for more details.
\end{example}

\begin{example}In a cartesian monoidal category, there exists \emph{exactly one} coalgebra structure on any object: the counit is the unique map to the terminal object, and the comultiplication is necessarily the diagonal. In particular, any coalgebra in spaces is cocommutative. On the point-set level, the latter fact propagates to the stable setting: in any one of the 1-categories of symmetric spectra, orthogonal spectra, $\Gamma$-spaces, $W$-spaces, and Elmendorf--Kriz--Mandell--May's category of $\mathbb S$-modules, all coassociative coalgebras are automatically cocommutative \cite{ps}. This is certainly not the case for coalgebras in the $\infty$-category of spectra, which therefore rarely admit a strict point-set rigidification.
\end{example}

\begin{example}
	Let $\kk$ be a commutative ring. Let $(\mathsf{Ch}_{\geq 0}(k),\otimes)$ be the monoidal category of nonnegatively graded chain complexes over $\kk$, and $(\mathsf{sMod}(\kk),\times)$ the monoidal category of simplicial $\kk$-modules. We let $\mathsf W$ and $\mathsf W'$ denote the classes of quasi-isomorphisms in the respective categories. The Dold--Kan functor 
	$\mathrm{DK} : \mathsf{Ch}_{\geq 0}(\kk) \longrightarrow \mathsf{sMod}(k)$
	is an equivalence of relative categories. Both $\mathrm{DK}$ and its inverse are simultaneously lax and oplax monoidal, and induce an equivalence of monoidal $\infty$-categories $\mathsf{Ch}_{\geq 0}(\kk)~[\mathsf W^{-1}] \simeq \mathsf{sMod}(\kk)~[\mathsf W'^{-1}]$. If $\kk$ is a field, then $\mathrm{Coalg}(\mathsf{Ch}_{\geq 0}(\kk))$ and $\mathrm{Coalg}(\mathsf{sMod}(\kk))$ admit left-transferred model structures from $\mathsf{Ch}_{\geq 0}(\kk)$ and $\mathsf{sMod}(\kk)$. Nevertheless, the  functor 
	\[ \mathrm{Coalg}(\mathsf{Ch}_{\geq 0}(\kk))~[\mathsf W^{-1}] \longrightarrow \mathrm{Coalg}(\mathsf{sMod}(\kk))~[\mathsf W'^{-1}]\]
	induced by $\mathrm{DK}$ is \emph{not} an equivalence of $\infty$-categories, as shown by Sor\'e \cite{sore}. The conclusion is that $\mathrm{Coalg}(\mathsf C)~[\mathsf W^{-1}]$ is in general highly sensitive to 1-categorical structure of $(\mathsf C,\mathsf W)$, and in no way does it only depend on the $\infty$-category $\mathsf C~[\mathsf W^{-1}]$. 
\end{example}

The takeaway from these examples is the following.
\begin{enumerate}
    \item It does not seem reasonable to hope for a rectification of coassociative coalgebras, under any reasonable hypotheses.\footnote{One exception, in the dg setting, is if the coalgebras are sufficiently connective, in which case Koszul duality furnishes an equivalence of $\infty$-categories between the homotopy theories of dg coalgebras and dg algebras, and the latter can be rectified. Compare e.g.~with Quillen's \cite{quillen69} equivalence of homotopy theories between simply connected strict cocommutative dg coalgebras, and connected dg Lie algebras, over a field of characteristic zero.} Maybe the best we should hope for is rectification for \emph{coalgebras over a cofibrant operad}. Topologically, this means considering coalgebras over the Stasheff associahedra; algebraically, this means considering $\mathcal{A}_\infty$-coalgebras. In fact, such a statement was conjectured by Le Grignou and Lejay in \cite[Conjecture 8.3]{linearcoalgebras}.
    
    \item It does not seem reasonable to attempt to prove such a rectification using the strategy of \cref{strategy remark}. Indeed, as indicated in \cref{cofree are complicated}, there is no workable description of the cofree coalgebra over the associative operad, much less the $\mathcal A_\infty$-operad. And the strategy should at some point use cofibrancy. An entirely different argument is needed.
\end{enumerate}

\subsection{Main results} The goal of this paper is to prove such rectification results for coalgebras over cofibrant operads, in the differential graded setting, over a field. In the associative case, our main theorem says the following.

\begin{theoremintro}\label{thm: main theorem associative case}
Let $\kk$ be a field, and let $\mathbf D(\kk) = \mathsf {Ch}(\kk)~[\mathsf{Q.iso}^{-1}]$ be the derived $\infty$-category of chain complexes over $\kk$. The injective model structure on $\mathsf {Ch}(\kk)$ may be left-transferred along the cofree-forgetful adjunction to a model structure on the 1-category $\onecoalg{\mathcal A_\infty}$ of $\mathcal A_\infty$-coalgebras and strict morphisms. This model structure satisfies
    \[ \mathrm{Coalg}(\mathbf D(\kk)) \simeq \onecoalg{\mathcal A_\infty}~[\mathsf{Q.iso}^{-1}].\]
\end{theoremintro}

The $\infty$-category $\mathrm{Coalg}(\mathbf D(\kk))$, and its cocommutative analogue, are fundamental objects of study in homological and homotopical algebra, given e.g.\ the role of dg coalgebras in formal deformation theory, or in construction of chain models of spaces. It is therefore striking that prior to \cref{thm: main theorem associative case}, there was no way of working with the $\infty$-category $\mathrm{Coalg}(\mathbf D(\kk))$ using classical languages of homotopy theory such as model categories, simplicial categories, etc. 

    \begin{remark}The definition of an $\mathcal A_\infty$-coalgebra is less standardized in the literature than the definition of an $\mathcal A_\infty$-algebra. To be clear, for us an $\mathcal A_\infty$-coalgebra is a graded vector space $C$ and a degree $-1$ derivation of the \emph{completed} tensor algebra $\smash{\widehat T}(C[-1])$, which squares to zero. This amounts to a collection of maps 
    \[
    \left\{ \Delta_n^C: C \longrightarrow C^{\otimes n} \right\}_{n \geq 1}~
    \]
    of degree $n-2$, which satisfy the relations imposed by the fact that the derivation that induces them squares to zero: \[
\sum_{r+s+t = n}
(-1)^{\,r + s t}\,
\bigl(
\mathrm{id}^{\otimes r}
\otimes \Delta_s^C
\otimes \mathrm{id}^{\otimes t}
\bigr)
\circ \Delta_{\,r+1+t}^C
\;=\; 0,
\qquad n \ge 1 .
\]
    A \emph{strict} morphism between two such coalgebras (as opposed to an $\infty$-morphism) is a homomorphism of graded vector spaces $f: C \longrightarrow C'$, with the property that the induced map $\smash{\widehat T}(C[-1]) \longrightarrow \smash{\widehat T}(C'[-1])$ is a homomorphism of dg algebras with respect to the differentials given by the respective derivations.\footnote{Beware that, for simplicity, we have only given the definition of a \textit{non-counital}  $\mathcal A_\infty$-coalgebra. Theorem A works for both counital and non-counital $\mathcal A_\infty$-coalgebras.}
    \end{remark}
    
Our actual main theorem is not about $\mathcal A_\infty$-coalgebras, but about coalgebras over a general cofibrant dg operad. On the one hand, these coalgebras over a cofibrant dg operad can be endowed with a left-transferred model structure from chain complexes using the methods of \cite{BHKKRS, HKRS, GKR}. On the other hand, in order to be able to state the general theorem, we need to first define a good notion of coalgebras over enriched $\infty$-operads. 
In this, we will closely follow an approach of Heuts \cite[Appendix A]{heuts2024}, who considered the problem of giving a good $\infty$-categorical definition of not necessarily conilpotent coalgebras over an enriched $\infty$-\emph{cooperad}.  

\medskip

 In the following, we work with the notion of enriched $\infty$-operad defined in \cite[Section 4.1.2]{BrantnerPhD}. See also \cite[Section 5.2.4]{ShiPhD} for a thorough exposition of these ideas. This approach is very similar to the one carried out in \cite{ChuHaugseng20}, and in fact they were very recently proven to be equivalent in \cite{arakawa2026}. Let us also point out a mismatch in terminology in the literature. Classically, one speaks of operads \emph{in} a given category $\mathsf C$. In higher category theory, one speaks of $\infty$-operads \emph{enriched} in $\mathsf C$. The phrase $\infty$-operad, with no modifications, always means $\infty$-operad enriched in the $\infty$-category of spaces $\mathbf {Spc}$.

 \medskip

If $\mathscr P$ is an $\infty$-operad, then we can define $\mathscr P$-algebras in any symmetric monoidal $\infty$-category. But if $\mathscr P$ is an $\infty$-operad enriched in a given $\infty$-category $\mathbf C$, then we can speak of $\mathscr P$-algebras in a symmetric monoidal $\infty$-category $\mathbf D$ only when $\mathbf D$ is copowered over $\mathbf C$. Hence if $\mathscr P$ is an $\infty$-operad enriched in $\mathbf D = \mathbf D(\kk)$, then we can not directly imitate \cref{def:coalgebra} to give a definition of what it means to be a $\mathscr P$-coalgebra in $\mathbf D$, since $\mathbf D^{\mathrm{op}}$ is not copowered over $\mathbf D$. 

\medskip

Let us outline the definition we give in this paper. In what follows, $\mathbf D$ is a compactly rigidly generated stable symmetric monoidal $\infty$-category. There are two ways to associate to a symmetric sequence in $\mathbf D$ an endofunctor of $\mathbf D$: the \emph{Schur functor}
\begin{equation*}
	\label{schurfunctor}\mathbf S\colon\mathbf{sSeq}(\mathbf D) \longrightarrow \mathbf{End}(\mathbf D)
\end{equation*} 
and the \emph{dual Schur functor}
\begin{equation*}
	\label{dualschurfunctor} \widehat{\mathbf S}^ c\colon\mathbf{sSeq}(\mathbf D)^{\mathrm{op}} \longrightarrow \mathbf{End}(\mathbf D).
\end{equation*} 	
The functor $\mathbf S$ is monoidal. Hence it takes an $\mathbf D$-enriched $\infty$-operad $\mathscr P$ (a monoid object in the domain) to a monad on $\mathbf D$ (a monoid object in the target). We may define the category of $\mathscr P$-algebras as the category of algebras for the monad $\mathbf S(\mathscr P)$. On the other hand, the functor $\smash{\widehat{\mathbf S}^c}$ is only lax monoidal, and $\smash{\widehat{\mathbf S}^c}(\mathscr P)$ is only a \emph{lax comonad}. Following Heuts, we circumvent the problem of giving a general definition of an algebra over a lax comonad in an $\infty$-categorical setting, by instead constructing an extension of the dual Schur functor to an endofunctor on pro-objects in $\mathbf D$:
\begin{equation*}
	\label{dualschurfunctor-pro}\widehat{\mathbf{S}}^c_{\mathrm{pro}}\colon \mathbf{sSeq}(\mathbf D)^{\mathrm{op}} \longrightarrow \mathbf{End}(\mathbf{pro}(\mathbf D)).
\end{equation*} 	
The functor $\smash{\widehat{\mathbf{S}}^c_{\mathrm{pro}}}$ is monoidal, so $\smash{\widehat{\mathbf{S}}^c_{\mathrm{pro}}(\mathscr P)}$ is a comonad. Hence the following definition makes sense.

\begin{definition}\label{definition of P-coalg}
Let $\mathscr{P}$ be an enriched $\infty$-operad in $\mathbf{D}$. We define the $\infty$-category of $\mathscr{P}$\textit{-coalgebras} as the pullback 
\[
\begin{tikzcd}[column sep=2.5pc,row sep=2.5pc]
\mathbf{Coalg}_\mathscr{P}(\mathbf{D}) \arrow[r] \arrow[d] \arrow[dr, phantom, "\ulcorner", very near start]
&\mathbf{Coalg}_{\widehat{\mathbf{S}}^c_{\mathrm{pro}}(\mathscr{P})}(\mathbf{pro}(\mathbf{D})) \arrow[d,"U"]\\
\mathbf{D} \arrow[r] 
&\mathbf{pro}(\mathbf{D})
\end{tikzcd}
\]
in the $\infty$-category of $\infty$-categories, where $\mathbf D \longrightarrow \mathbf{pro}(\mathbf{D})$ is the ``constant pro-object'' functor.
\end{definition}

A first sanity check for \cref{definition of P-coalg} is that if each $\mathscr P(n)$ is a dualizable object, so that we can define a ``linear dual'' cooperad $\mathscr P^\vee$, then $\mathscr P$-coalgebras as defined in \cref{definition of P-coalg} are equivalent to $\mathscr P^\vee$-coalgebras as defined in \cite[Appendix A]{heuts2024}.

\medskip

	A second sanity check is the following. There is a unique cocontinuous functor $F : \mathbf{Spc}\longrightarrow \mathbf D$ taking a point to the monoidal unit in $\mathbf D$, and $F$ is symmetric monoidal since the tensor product in $\mathbf D$ preserves colimits. Hence if $\mathscr O$ is an $\infty$-operad in the usual sense, then we get a $\mathbf D$-enriched $\infty$-operad $F(\mathscr O)$. We should check that $\mathscr O$-coalgebras in the sense of Lurie --- that is, $\mathbf{Coalg}_{\mathscr{O}}(\mathbf{D}) \coloneqq \mathbf{Alg}_{\mathscr{O}}(\mathbf{D}^{\mathrm{op}})^{\mathrm{op}}$ --- are equivalent to $F(\mathscr O)$-coalgebras in the sense of \cref{definition of P-coalg}. This is indeed the case.
    
\begin{theoremintro}
Let $\mathscr{O}$ be a (classical) $\infty$-operad. There is an equivalence of $\infty$-categories
\[
\mathbf{Coalg}_{\mathscr{O}}(\mathbf{D}) \simeq \mathbf{Coalg}_{F(\mathscr{O})}(\mathbf{D})~
\]
between Lurie's $\infty$-category of $\mathscr{O}$-coalgebras, and the $\infty$-category of $F(\mathscr{O})$-coalgebras as defined in Definition \ref{definition of P-coalg}.
\end{theoremintro}

Having this reasonable definition of coalgebras, we can now state the main theorem of this paper, which gives explicit point-set models and rectification for differential graded homotopy coherent coalgebras.

\begin{theoremintro}\label{thm: introrectification}
Let $\kk$ be a field, and let $\mathcal{P}$ be a cofibrant dg operad over $\kk$. There is an equivalence of $\infty$-categories
\[
\mathcal{P}\text{-}\mathsf{coalg}~[\mathsf{Q.iso}^{-1}] \simeq \mathbf{Coalg}_{\mathscr{P}}(\mathbf{D}(\kk))
\]
between dg $\mathcal{P}$-coalgebras up to quasi-isomorphism and coalgebras in $\mathbf{D}(\kk)$ over the induced enriched $\infty$-operad $\mathscr{P}$. 
\end{theoremintro}

\subsection{Some ideas of the proof of the main theorem.}
For any dg operad $\mathcal P$, with associated $\mathbf D(\kk)$-enriched $\infty$-operad $\mathscr P$, there is a natural functor 
\[ \onepcoalg\,[\mathsf{Q.iso}^{-1}] \longrightarrow \mathbf{Coalg}_{\mathscr{P}}(\mathbf{D}(\kk)).\]	
We say that $\mathcal P$ is \emph{rectifiable}  if this map is an equivalence. The proof consists of three steps:

	\begin{enumerate}
		\item \label{step1}A free operad $\mathcal P = \mathbb T(M)$, with $M$ a cofibrant dg symmetric sequence, is rectifiable.
		\item \label{step2}Suppose given a pushout diagram of dg operads
        \[
        \begin{tikzcd}[column sep=2.5pc,row sep=2.5pc]
        \mathbb T(M) \arrow[r] \arrow[d] \arrow[dr, phantom,"\ulcorner" rotate=-180, very near end] &
        \mathcal P \arrow[d] \\
        \mathbb T(M') \arrow[r] &
        \mathcal P'.
        \end{tikzcd}
        \]
	If $M \longrightarrow M'$ is a generating cofibration, and $\mathcal P$ is rectifiable, then $\mathcal P'$ is rectifiable. 
		\item \label{step3} A retract of a cofibrant rectifiable operad is rectifiable.
	\end{enumerate}
	This proves the theorem: indeed, cofibrant operads are precisely the retracts of quasi-free operads, and a quasi-free operad is an iterated pushout of the form in (\ref{step2}). In order to implement this proof strategy, we crucially use that the assigment that sends an operad to its category of coalgebras sends colimits to limits (both 1-categorically and $\infty$-categorically). Then, we show that at each step that these limits induce homotopy limits of \textit{quasi-categories}.
	
	\begin{remark}
It is certainly natural to expect that a version of \cref{thm: introrectification} is true also if $\kk$ is an arbitrary ring. Furthermore, there should be a version outside the differential graded setting (e.g.~under hypotheses like those of \cref{lurie theorem}). This is the subject of ongoing work.
	\end{remark}

\subsection{Applications and perspectives}
The following result is a celebrated theorem of Mandell, shown in \cite{Mandell01}:
	
	\begin{theorem}[Mandell]\label{mandell}
		Fix a prime $p$. Consider the contravariant singular cochain functor $X \mapsto C^*(X;\overline{\mathbb{F}}_p)$ from the homotopy category of spaces, to the homotopy category of $\mathbb E_\infty$-algebras over $\overline{\mathbb{F}}_p$. It is fully faithful when restricted to  nilpotent connected $p$-complete spaces of finite $p$-type.
	\end{theorem}
This theorem is a $p$-adic analogue of the theorem over $\mathbb Q$ of  Sullivan \cite{sullivan}, that $X \mapsto A_{\mathit{PL}}(X)$ is fully faithful when restricted to nilpotent connected rational spaces of finite $\mathbb Q$-type. In Sullivan's theorem, the finite type hypothesis is needed only because passing from chains to cochains introduces an ``unnecessary'' dualization: there is no finite type hypothesis in the analogous results of Quillen \cite{quillen69} for simply-connected cocommutative coalgebras. It is thus natural to expect that the same holds in the $p$-adic setting.

\begin{expectation}\label{expectation2}
	Fix a prime $p$. Consider the singular chain functor $X \mapsto C_*(X;\overline{\mathbb{F}}_p)$ from the homotopy category of spaces, to the homotopy category of $\mathbb E_\infty$-coalgebras over $\overline{\mathbb{F}}_p$. It is fully faithful when restricted to  nilpotent connected $p$-complete spaces.
\end{expectation}

Mandell's \cref{mandell} admits an $\infty$-categorical refinement, see \cite{Lurie11Rational}; the statement is identical to \cref{mandell}, but ``homotopy category'' is replaced with ``$\infty$-category''. Bachmann--Burklund \cite{bachmann2024} recently obtained a coalgebraic version, with no finite type hypothesis:

	\begin{theorem}[Bachmann--Burklund]\label{bb} Fix a prime $p$. Consider the singular chain functor $X \mapsto C_*(X;\overline{\mathbb{F}}_p)$ from the $\infty$-category of spaces, to the $\infty$-category $\mathrm{CAlg}(\mathbf D(\overline{\mathbb{F}}_p)^{\mathrm{op}})^{\mathrm{op}}$. It is fully faithful when restricted to  nilpotent connected $p$-complete spaces.
	\end{theorem} 

At first, one may think that \cref{bb} should directly imply Expectation \ref{expectation2}, by taking homotopy categories. But this is not at all immediate, since \emph{a priori} there is no relationship between the $\infty$-category $\mathrm{CAlg}(\mathbf D(\overline{\mathbb{F}}_p)^{\mathrm{op}})^{\mathrm{op}}$ and the $\infty$-category of $\mathbb E_\infty$-coalgebras localized at quasi-isomorphisms (except a natural functor from the latter to the former). Our theorem fills this gap: the $\infty$-categories are indeed equivalent, and \cref{bb} indeed proves Expectation \ref{expectation2}.

\begin{remark}If one tries to prove Expectation \ref{expectation2} directly, at the point-set level, by dualizing Mandell's arguments, then one runs into issues like whether the category of $\mathbb E_\infty$-coalgebras localized at quasi-isomorphisms is comonadic over $\mathbf D(\overline{\mathbb{F}}_p)$ --- a property which is automatic for $\mathrm{CAlg}(\mathbf D(\overline{\mathbb{F}}_p)^{\mathrm{op}})^{\mathrm{op}}$. Thus it seems in a sense that a main obstruction to proving a coalgebraic version of Mandell's theorem was a lack of good point-set models, an obstruction which Burklund--Bachmann circumvented by carrying out the entire argument at the $\infty$-categorical level.  \end{remark}

Our results provide an explicit point-set version of the cellular chains functor $C_*(-;\kk)$ together with its functorial $\mathbb{E}_\infty$-coalgebra structure. For this purpose, we consider the explicit dg $\mathcal{E}$-coalgebra structure of the point-set cellular chain functor $C_*(-;\kk)$ over the Barratt-Eccles dg operad, as constructed in \cite{BergerFresse}, and we pull it back along the cofibrant resolution $\Omega \mathrm{B} \mathcal{E} \qi \mathcal{E}$. 

\begin{theoremintro}\label{thm: presenting chains with the coalgebra structure}
Let $\kk$ be a field. The functor 
\[
C_*(-;\kk): \mathsf{sSet} \longrightarrow \Omega \mathrm{B} \mathcal{E}\text{-}\mathsf{coalg}~, 
\]
is a point-set model for the $\infty$-categorical chains functor together with its $\mathbb{E}_\infty$-coalgebra structure.
\end{theoremintro}

When $\kk$ is a separably closed field of characteristic $p > 0$, this allows us to lift the main result of \cite{bachmann2024} to the point-set level and obtain explicit models for nilpotent $p$-adic homotopy types. The functor $C_*(-;\kk)$ fits in a Quillen adjunction 
\[
\begin{tikzcd}[column sep=5pc,row sep=3pc]
         \mathsf{sSet} \arrow[r, shift left=1.1ex, "C_*(-;\kk)"{name=F}] &\Omega \mathrm{B} \mathcal{E}\text{-}\mathsf{coalg}, \arrow[l, shift left=.75ex, "\overline{\mathcal{R}}"{name=U}]
            \arrow[phantom, from=F, to=U, , "\dashv" rotate=-90]
\end{tikzcd}
\]
between the category of simplicial sets, and the category of dg $\Omega\mathrm{B}\mathcal{E}$-coalgebras endowed with the transferred model structure from chain complexes over $\kk$. The first author used this in \cite{lucio2024higherlietheorypositive} to construct a Quillen adjunction 
\[
\begin{tikzcd}[column sep=5pc,row sep=3pc]
          \mathsf{sSet}_* \arrow[r, shift left=1.1ex, "\mathcal{L}_*"{name=A}]
          &\mathsf{abs}~\mathcal{L}_\infty^\pi\text{-}\mathsf{alg}^{\mathsf{qp}\text{-}\mathsf{comp}} \arrow[l, shift left=.75ex, "\mathcal{R}_*"{name=B}]
           \arrow[phantom, from=A, to=B, , "\dashv" rotate=-90]
\end{tikzcd}
\]
between pointed simplicial sets, and absolute partition $\mathcal{L}_\infty$-algebras which satisfy a separatedness axiom called \textit{qp-completeness} in the terminology of \cite{premierpapier}. As an immediate corollary of Theorem \ref{thm: presenting chains with the coalgebra structure} and \cite[Theorem 1.2]{bachmann2024}, we can remove the \textit{finite type} and the \textit{connected} assumptions on the nilpotent $p$-adic spaces considered in \cite{lucio2024higherlietheorypositive}. 

\begin{corollary}
Let $\kk$ be a separably closed field of characteristic $p > 0$. Let $X$ be a pointed nilpotent simplicial set.
\begin{itemize}
\item[(1)] The derived unit of the adjunction 
\[
\mathbb{R}\eta_X: X \longrightarrow \mathbb{R}\overline{\mathcal{R}}\tilde{C}_*(X;\kk)
\]
is an equivalence in homology with coefficients in $\mathbb{F}_p$, where $\tilde{C}_*(X;\kk)$ denotes the \textit{reduced} chains with its non-counital coalgebra structure over the Barratt-Eccles operad.
\item[(2)] The unit of adjunction
\[
\eta_X: X \qi \mathcal{R}_*\mathcal{L}_*(X) 
\]
is an equivalence in homology with coefficients in $\mathbb{F}_p$.
\end{itemize}
\end{corollary}

One can also remove the \textit{pointed} assumption by working with the counital coalgebra structure over the Barratt-Eccles operad of $C_*(-;\kk)$ and with \textit{curved} absolute partition $\mathcal{L}_\infty$-algebras. Moreover, the above corollary also allows us to remove the \textit{finite type} and the \textit{connected} in \cite[Theorem E]{lucio2024higherlietheorypositive} and in the applications of \cite[Theorem B]{lucio2024higherlietheorypositive} to models of spaces.

\medskip

Finally, let us mention that the main motivation for carrying out this work is the forthcoming paper \cite{poinsetkoszul}, where using point-set models we give a unified framework that intertwines all the different bar-cobar adjunctions in the literature (both 1-categorical and $\infty$-categorical). In particular, it allows us to give point-set models for Lurie's bar-cobar adjunction between augmented $\mathbb{E}_1$-algebras and coaugmented $\mathbb{E}_1$-coalgebras in $\mathbf{D}(\kk)$, as well as Ayala--Francis' generalization of this adjunction to the $\mathbb{E}_n$ case when $\kk$ is of characteristic zero. In both cases, we crucially need to have point-set models for the target $\infty$-categories of homotopy coherent coalgebras in $\mathbf{D}(\kk)$. 

\subsection*{Acknowledgements} We wish to warmly thank Grégory Ginot for several discussions that led to this paper, and in particular for suggesting the proof of Proposition \ref{prop: isofibration}. The first author also wishes to thank Geoffroy Horel, Brice Le Grignou and Damien Lejay for discussions about these and related topics. Finally, we wish to thank Brice Le Grignou for pointing out a mistake in the Lemma 3.9 of the earlier version, see Remark 3.9. 

\subsection*{Conventions}
Regarding notation, we will adopt the following conventions. 


\begin{itemize}
    \item We denote by $\kk$ a field. Our base 1-category will be the category of chain complexes over $\kk$, together with its closed symmetric monoidal structure given by the tensor product $-\otimes-$ of chain complexes and the Koszul sign convention. We will denote the internal hom of this category by $[-,-]$. We adopt the \textit{homological convention}, differentials will be of degree $-1$. We denote this 1-category by $\mathsf{Ch}(\kk)$. 
    
    
    \item In general, dg operads in $\mathsf{Ch}(\kk)$ will be denoted by $\mathcal{P}$. For a dg operad $\mathcal{P}$, we denote the category of dg $\mathcal{P}$-algebras by $\mathcal{P}\text{-}\mathsf{alg}$ and of dg $\mathcal{P}$-coalgebras by $\mathcal{P}\text{-}\mathsf{coalg}$. Similarly for algebras and coalgebras over a dg cooperad.
    
    
    \item Let $\mathsf{C}$ be a category and let $\mathsf{W}$ be a class of arrows in $\mathsf{C}$. We will denote $\mathsf{C}~[\mathsf{W}^{-1}]$ the $\infty$-category obtained by localizing $\mathsf{C}$ at $\mathsf{W}$. When working at the $\infty$-categorical level, limits and colimits should be understood as meaning homotopy limits and colimits. If
\[
\mathsf{F}: \mathsf{C} \longrightarrow \mathsf{D}
\]
is a functor between categories which sends a class of arrows $\mathsf{W}_{\mathsf{C}}$ in $\mathsf{C}$ to a class of arrows $\mathsf{W}_{\mathsf{D}}$ in $\mathsf{D}$, we still denote 
\[
\mathsf{F}: \mathsf{C}~[\mathsf{W}^{-1}_{\mathsf{C}}] \longrightarrow \mathsf{D}~[\mathsf{W}^{-1}_{\mathsf{D}}]
\]
the induced functor at the $\infty$-categorical level. However, we will add $\mathbb{L}\mathsf{F}$ and $\mathbb{R}\mathsf{F}$ to the left (resp. right) derived functors of $\mathsf{F}$ when it is a left (resp. right) Quillen functor which does not preserve weak-equivalences in general. If $\mathsf{C}$ is a 1-category that we consider as an $\infty$-category via the nerve functor $\mathcal{N}$, we will still denote it by $\mathsf{C}$ instead of $\mathcal{N}(\mathsf{C})$.


    \item We denote by $\mathcal{N}^{coh}$ the coherent nerve of a simplicially enriched category.


    \item We denote by $\mathsf{1}$ the 1-category with two objects and one arrow and $\mathsf{1}^{\simeq}$ the relative category with two objects and an equivalence between them. 

    
    \item Operads and algebras will typically be considered in the symmetric monoidal $\infty$-category $\mathbf{D}(\kk) \coloneqq \mathsf{Ch}(\kk)~[\mathsf{Q.iso}^{-1}]$ of chain complexes over $\kk$ localized at  quasi-isomorphisms.
    
    \item In general, an (enriched) $\infty$-operad in $\mathbf{D}(\kk)$ will be denoted by $\mathscr{P}$. Given any enriched $\infty$-operad $\mathscr{P}$, we denote the $\infty$-category of $\mathscr{P}$-algebras in the base $\infty$-category $\mathbf{D}(\kk)$ by $\mathbf{Alg}_{\mathscr{P}}(\mathbf{D}(\kk))$ and the $\infty$-category of $\mathscr{P}$-coalgebras in $\mathbf{D}(\kk)$ by $\mathbf{Coalg}_{\mathscr{P}}(\mathbf{D}(\kk))$. Similarly for algebras and coalgebras over an enriched $\infty$-cooperad.
\end{itemize}

\section{Point-set coalgebras}

The goal of this section is recall the notion of coalgebras over an diffential graded operad. A key feature of this definition is that it does not impose any kind of conilpotency condition on the coalgebras it encodes.

\subsection{Differential graded operads, algebras, and coalgebras}\label{sec:1-operads}
Let $\kk$ be any field. We consider the base 1-category $\mathsf{Ch}(\kk)$ of chain complexes of $\kk$-modules as our base 1-category. Let $\mathsf{Fin}^{\simeq}$ denote the 1-category of finite sets and bijections. We define the category of \textit{dg symmetric sequences} as the category of functors from $\mathsf{Fin}^{\simeq}$ to $\mathsf{Ch}(\kk)$:
\[
\mathsf{sSeq}(\mathsf{Ch}(\kk)) \coloneqq \mathsf{Fun}(\mathsf{Fin}^{\simeq},\mathsf{Ch}(\kk))~. 
\]
For $M$ in $\mathsf{sSeq}(\mathsf{Ch}(\kk))$, we denote by $M(n)$ the evaluation of $M$ at the set $\underline{n}= \{1,\cdots,n\}$. The category of dg symmetric sequences in admits a monoidal structure given by the composition product $\circledcirc$, which for two dg symmetric sequences $M$ and $N$ is  given by
\[
M \circledcirc N(n) \simeq \bigoplus_{k \geq 0} \left(\bigoplus_{\underline{n} = \sqcup_{i=1}^k S_i} M(k) \otimes N(S_1) \otimes \cdots \otimes N(S_k)\right)_{\mathbb{S}_k} .
\]

The unit for the composition is $ I$, given by 
$$
I (n) \coloneqq
\begin{cases}
    0 \text{ if }n \neq 1~,
    \\
    \kk \text{ if }n = 1.
\end{cases}
$$

\begin{definition}[dg operad]
A \textit{dg operad} $\oneop$ is a monoid $(\oneop,\gamma,\eta)$ in the category of dg symmetric sequences with respect to the composition product. 
\end{definition}

\begin{notation}
Let $V$ and $W$ be two chain complexes. We denote by $[V,W]$ the graded module of graded maps between $V$ and $W$, together with the differential given by $\partial(f)= d_W \circ f -(-1)^{|f|}f \circ d_V$. The construction $[-,-]$ defines a canonical self-enrichment of the category of  chain complexes.
\end{notation}

\begin{example}[Endomorphism operad] 
Let $V$ be a chain complex. One can construct the \textit{endomorphism operad} of $V$ by considering the dg symmetric sequence given by 
\[
\mathrm{End}_V(n) \coloneqq [V^{\otimes n}, V]~, 
\] 
with the natural $\mathbb{S}_n$-action and where the composition map is given by the composition of morphisms, see \cite[Chapter 5]{LodayVallette} for more details.
\end{example}

\begin{example}[Coendomorphism operad] 
Let $V$ be a chain complex. One can construct the \textit{coendomorphism operad} of $V$ by considering the dg symmetric sequence given by 
\[
\mathrm{coEnd}_V(n) \coloneqq [V, V^{\otimes n}]~, 
\] 
with the natural $\mathbb{S}_n$-action and where the composition map is given by the composition of morphisms, see \cite[Chapter 5]{LodayVallette} for more details.
\end{example}

\begin{definition}[dg $\oneop$-algebra]
A dg $\oneop$-coalgebra $A$ is a pair $(A,\Gamma_A)$ of a chain complex $A$ together with a morphism of dg operads $\Gamma_A: \oneop \longrightarrow \mathrm{End}_A$. 
\end{definition}

While operads are usually used to encode types of algebras, they can equally well encode types of coalgebras. Unlike coalgebras over a cooperad, these coalgebras typically come without any conilpotency restriction.

\begin{definition}[dg $\oneop$-coalgebra]
A dg $\oneop$-coalgebra $C$ is a pair $(C,\Gamma_C)$ of a chain complex $C$ together with a morphism of dg operads $\Gamma_C: \oneop \longrightarrow \mathrm{coEnd}_C$. 
\end{definition}

\begin{example}[Cocommutative coalgebras]
Let $\mathcal{P} = u\mathcal{C}om$, which is given by $u\mathcal{C}om(n) = \kk$ for all $n \geq 0$, together with the trivial $\mathbb{S}_n$-action and the obvious composition maps. The category of dg $u\mathcal{C}om$-coalgebras is equivalent to the category of \textit{all} counital cocommutative dg coalgebras, with no conilpotency hypotheses. 
\end{example}

\subsection{Defining coalgebras via the dual Schur functor}\label{set:1-dual schur}
Algebras over an operad are algebras over a monad, which is given by the Schur functor associated to the operad. See, for instance, \cite[Chapter 2]{Fresse09}. To any dg symmetric sequence, one can also associate its \textit{dual Schur functor}, which is given by
\[
\begin{tikzcd}[column sep=4pc,row sep=0.5pc]
\widehat{\mathrm{S}}^c : \mathsf{sSeq}(\mathsf{Ch}(\kk))^{\mathsf{op}} \arrow[r]
&\mathsf{End}(\baseonecat) \\
M \arrow[r,mapsto]
&\widehat{\mathrm{S}}^c(M) \coloneqq \displaystyle \prod_{n \geq 0} \left[M(n),(-)^{\otimes n}\right]^{\mathbb{S}_n}~.
\end{tikzcd}
\]
The functor $\widehat{\mathrm{S}}^c(-)$ is only lax monoidal. Hence, for a dg operad $\oneop$, its dual Schur functor $\widehat{\mathrm{S}}^c(\oneop)$ fails to be a comonad. Nevertheless, the definition of a dg $\oneop$-coalgebra can still be rewritten using the dual Schur functor of $\oneop$. 

\begin{lemma}\label{definition: P-coalgebra}
Let $C$ be a chain complex. The data of a dg $\oneop$-coalgebra structure on $C$ is equivalent to the data of a  structural map
\[
\Delta_C: C \longrightarrow \widehat{\mathrm{S}}^c(\oneop)(C)  = \displaystyle \prod_{n \geq 0} \left[\mathcal{P}(n),C^{\otimes n}\right]^{\mathbb{S}_n}~,
\]
such that the following diagram commutes 
\[
\begin{tikzcd}[column sep=4.5pc,row sep=4pc]
C \arrow[r,"\Delta_C"] \arrow[d,"\Delta_C",swap] 
&\widehat{\mathrm{S}}^c(\oneop)(C) \arrow[r,"\widehat{\mathrm{S}}^c(\mathrm{id})(\Delta_C)"]
&\widehat{\mathrm{S}}^c(\oneop) \circ \widehat{\mathrm{S}}^c(\oneop)(C) \arrow[d,"\varphi_{\oneop,\oneop}(C)"] \\
\widehat{\mathrm{S}}^c(\oneop)(C) \arrow[rr,"\widehat{\mathrm{S}}^c(\gamma)(\mathrm{id})"]
&
&\widehat{\mathrm{S}}^c(\oneop \circ \oneop)(C)~,
\end{tikzcd}
\]
where $\varphi$ is the lax monoidal structure of the functor $\widehat{\mathrm{S}}^c$. 
\end{lemma}

\begin{proof}
A map 
\[
\Delta_C: C \longrightarrow \displaystyle \prod_{n \geq 0} \left[\mathcal{P}(n),C^{\otimes n}\right]^{\mathbb{S}_n}~,
\]
is equivalent to a collection of maps $\{\Delta_C^n: C \longrightarrow \left[\mathcal{P}(n),C^{\otimes n}\right]^{\mathbb{S}_n}\}$, which by adjunction is equivalent to a collection of $\mathbb S_n$-equivariant maps $\{\mathcal{P}(n) \longrightarrow \left[C,C^{\otimes n}\right]\}$, i.e.\ a map of dg symmetric sequences $\mathcal{P} \longrightarrow \mathrm{coEnd}_C$. One can check that the map $\Delta_C$ satisfies the compatibility conditions imposed by the above diagram if and only if its associated morphism of dg symmetric sequences $\mathcal{P} \longrightarrow \mathrm{coEnd}_C$ is a morphism of dg operads.
\end{proof}

\begin{remark}
In the terminology of \cite{anelcofree2014}, $\widehat{\mathrm{S}}^c(\oneop)$ is a lax comonad, and dg $\oneop$-coalgebras are equivalent to coalgebras over this lax comonad. 
\end{remark}

\subsection{Comonadicity of coalgebras over an operad}
When one works in chain complexes over a field, the following result show that the category of dg $\mathcal{P}$-coalgebras is indeed comonadic over the base category of chain complexes, for any dg operad $\mathcal{P}$. 

\begin{theorem}[{\cite[Theorem 2.7.11]{anelcofree2014}}]\label{thm: cofree coalgebra over an operad}
Let $\oneop$ be a dg operad. The category of dg $\oneop$-coalgebras is comonadic. In other words, there exists a comonad $(\mathrm{L}(\oneop), \omega, \zeta)$ in the category of dg modules such that the category of $\mathrm{L}(\oneop)$-coalgebras is equivalent to the category of dg $\oneop$-coalgebras.
\end{theorem}

In particular, this entails the existence of a cofree dg $\oneop$-coalgebra. While in the general setting of \cite{anelcofree2014}, the construction of the comonad $\mathrm{L}(\oneop)$ is given by an infinite recursion, the construction of $\mathrm{L}(\oneop)$ in the category of chain complexes over a field stops at the first step. It is given by the following pullback 
\[
\begin{tikzcd}[column sep=3pc,row sep=3pc]
\mathrm{L}(\oneop) \arrow[r,"p_2"] \arrow[d,"p_1",swap,rightarrowtail] \arrow[dr, phantom, "\ulcorner", very near start]
&\widehat{\mathrm{S}}^c(\oneop) \circ \widehat{\mathrm{S}}^c(\oneop) \arrow[d,"\varphi_{\oneop,\oneop}",rightarrowtail] \\
\widehat{\mathrm{S}}^c(\oneop) \arrow[r,"\widehat{\mathrm{S}}^c(\gamma)"]
&\widehat{\mathrm{S}}^c(\oneop \circ \oneop)
\end{tikzcd}
\]
where $\varphi$ is the lax monoidal structure of the functor $\widehat{\mathrm{S}}^c$.

Anel's result holds true actually in more general closed symmetric monoidal categories. Precisely, \cite[Corollary 2.7.12]{anelcofree2014} relies on \cite[Hypothesis 2.7.5]{anelcofree2014} stating that:

\begin{itemize}
    \item the canonical natural transformation $[X,Y]\otimes [X',Y']\rightarrow [X\otimes X',Y\otimes Y']$ is a monomorphism;
    \item the functor $\otimes$ commutes with countable intersections in each variable, where a countable intersection is an $\mathbb{N}$-indexed chain of monomorphisms.
\end{itemize}

Such assumptions are satisfied by cartesian categories (sets, simplicial sets, topoi, compactly generated Hausdorff spaces...) as well as vector spaces and chain complexes over a field for example. 

\subsection{Coadmissible operads}
Let $\oneop$ be a dg operad and let 
\[
\begin{tikzcd}[column sep=5pc,row sep=3pc]
          \onepcoalg \arrow[r, shift left=1.1ex, "U"{name=F}] & \baseonecat , \arrow[l, shift left=.75ex, "\mathrm{L}(\oneop)"{name=U}]
            \arrow[phantom, from=F, to=U, , "\dashv" rotate=-90]
\end{tikzcd}
\]
be the cofree-forgetful adjunction of Theorem \ref{thm: cofree coalgebra over an operad}. Here we consider chain complexes together with the injective model structure, as constructed in \cite[Theorem 2.3.13]{Hovey}. 

\begin{definition}[Coadmissible operad]
A dg operad $\oneop$ is called \textit{coadmissible} if its category of dg $\oneop$-coalgebras admits a combinatorial model structure left-transferred along the cofree-forgetful adjunction, determined by the following classes of maps:
\begin{enumerate}
\item the class of weak-equivalences is given by quasi-isomorphisms;

\item the class of cofibrations is given by degree-wise monomorphisms;

\item the class of fibrations is determined by right-lifting property against acyclic cofibrations.
\end{enumerate}
\end{definition}

Since any chain complex is cofibrant in the injective model structure, it suffices to have a natural cylinder object in the category of dg $\oneop$-coalgebras in order for $\oneop$ to be coadmissible by \cite{BHKKRS, HKRS}, see the particular formulation given in \cite[Appendix B]{premierpapier}. 
 
\begin{remark}
Over a field of characteristic zero, all dg operads $\oneop$ are admissible, meaning that dg $\oneop$-algebras admit a transferred model structure from chain complexes along the free-forgetful adjunction. However, even when $\kk$ is a field of characteristic zero, it is not true that all dg operads are coadmissible. In fact, if $\kk$ is algebraically closed, it is shown in \cite[Proposition 8.10]{linearcoalgebras} that the dg operad $u\mathcal{C}om$ is not coadmissible.
\end{remark}

\begin{example}
Let $\mathcal{E}$ be the Barratt--Eccles operad of \cite{BergerFresse}. Since the interval object $I$ in chain complexes admits a canonical dg $\mathcal{E}$-coalgebra structure and since $\mathcal{E}$ is a Hopf operad, considering the tensor product $I \otimes (-)$ provides a natural cylinder object in dg $\mathcal{E}$-coalgebras. Thus one can left-transfer along the cofree-forgetful adjunction and $\mathcal{E}$ is coadmissible. 
\end{example}

\begin{example}\label{ex:cofibrant is coadmissible}
For any dg operad $\oneop$, there is a canonical map $\oneop \otimes \mathcal{E} \stackrel\sim\longrightarrow \oneop$. If this map admits a section --- in particular, if $\oneop$ is cofibrant --- then $\oneop$ is coadmissible. Indeed, $I \otimes (-)$ is naturally a dg $\mathcal{E} \otimes \oneop$-coalgebra, and by pulling it back along this section, it provides dg $\oneop$-coalgebras with a natural cylinder object. 
\end{example}

\subsection{Cofibrant operads}
We now consider the homotopy theory of dg operads themselves, in particular, the semi-model structure on dg operads constructed by Fresse in \cite[Chapter 12]{Fresse09}. We will say that a dg operad is \textit{cofibrant} if it is cofibrant in this semi-model structure. 

\begin{proposition}\label{prop: cofibrant operads are coadmissible}\leavevmode
\begin{enumerate}
\item Every cofibrant dg operad is coadmissible.
\item Any weak-equivalence of cofibrant operad induces a Quillen equivalence between their categories of coalgebras. 
\end{enumerate}
\end{proposition}

\begin{proof}
The first point follows by the same argument as in \cref{ex:cofibrant is coadmissible}. The second point follows by the same arguments as in \cite[Lemma 33 and Proposition 31]{premierpapier}. 
\end{proof}

A dg operad is called \textit{cell cofibrant} if it is obtained as a colimit of iterated cell attachments. See \cite[Section 12.2.1]{Fresse09} and in particular \cite[Proposition 12.2.3]{Fresse09}. The class of cofibrant dg operads consists precisely of the retracts of cell cofibrant operads. 

\subsection{Dévissage of coalgebras over operads}
The functor which sends a dg operad $\mathcal{P}$ to the category of dg $\mathcal{P}$-coalgebras is a contravariant functor. Using the arguments in \cite{DrummondColeHirshLejay20}, we get that it is a right adjoint functor and therefore sends colimits in dg operads to limits in the category of accessible categories over $\mathsf{Ch}(\kk)$. When $\mathcal{P}$ is a \textit{cell cofibrant} dg operad, this induces a \textit{dévissage} of the category of dg $\mathcal{P}$-coalgebras along the cells that compose $\mathcal{P}$. 

\begin{theorem}\label{thm: 1-categorical adjunction between categories of coalgebras and coendomorphism operads}
There is an adjunction \[
\begin{tikzcd}[column sep=5pc,row sep=3pc]
         \mathsf{Op}(\mathsf{Ch}(\kk))^{\mathsf{op}} \arrow[r, shift left=1.1ex, "\mathsf{Coalg}(-)"{name=F}] &\mathsf{Cat}^{\mathsf{acc}}_{/\mathsf{Ch}(\kk)}, \arrow[l, shift left=.75ex, "\mathrm{coEnd}_{(-)}"{name=U}]
            \arrow[phantom, from=F, to=U, , "\dashv" rotate=90]
\end{tikzcd}
\]between the opposite category of dg operads and the slice over $\mathsf{Ch}(\kk)$ of accessible categories. The right adjoint sends a dg operad $\mathcal{P}$ to the category of dg $\mathcal{P}$-coalgebras and the left adjoint sends an accessible category $F: \mathsf{C} \longrightarrow \mathsf{Ch}(\kk)$ to the coendomorphism operad of the functor $F$. 
\end{theorem}

\begin{proof}
Essentially follows from \cite[Section B.1.3]{DrummondColeHirshLejay20}, using the fact that dg $\mathcal{P}$-coalgebras are accessible by Theorem \ref{thm: cofree coalgebra over an operad}.
\end{proof}

Let $S^{k}(p)$ denote the dg symmetric sequence given by $\kk[\mathbb{S}_p]$ in degree $k \in \mathbb{Z}$ and in arity $p \geq 0$ and by zero elsewhere. Let $D^k(p)$ denote the dg symmetric sequence given by $\kk[\mathbb{S}_p]$ in degrees $k-1$ and $k$, for $k \in \mathbb{Z}$, with the the differential being the identity map, for some arity $p \geq 0$ and by zero elsewhere. Let us denote by $\mathbb{T}(M)$ the free dg operad on a dg symmetric sequence $M$. Then the generating cofibrations of the semi-model structure of dg operads are given by the inclusions $\iota^k(p): \mathbb{T}(S^{k}(p)) \hookrightarrow \mathbb{T}(D^k(p))$ for all $p \geq 0$ and $k \in Z$. Cell dg operads are obtained as colimits of pushouts along these inclusions, and therefore their categories of coalgebras can be reconstructed using these cell attachments. 

\begin{corollary}\label{corollary: devissage along cells of 1-categories of P-coalgebras}\leavevmode
\begin{enumerate}
\item Let $p \geq 0$ and $k \in \mathbb{Z}$, a pushout of dg operads 
\[
\begin{tikzcd}[column sep=3.5pc,row sep=3.5pc]
\mathbb{T}(S^{k-1}(p)) \arrow[r,"\psi"] \arrow[d,"\iota^k(p)",swap]  \arrow[dr, phantom,"\ulcorner" rotate=-180, very near end]
&\oneop_\alpha \arrow[d,"\iota_\alpha"]  \\
\mathbb{T}(D^k(p)) \arrow[r,"\varsigma"]
&\oneop_{\alpha + 1}
\end{tikzcd}
\]
induces a pullback of categories
\[
\begin{tikzcd}[column sep=3.5pc,row sep=3.5pc]
\onecoalg{\oneop_{\alpha + 1}} \arrow[r,"\iota_\alpha^*"] \arrow[d,"\varsigma^*",swap] \arrow[dr, phantom, "\ulcorner", very near start]
&\onecoalg{\oneop_\alpha} \arrow[d,"\psi^*"]  \\
\onecoalg{\mathbb{T}(D^k(p))} \arrow[r,"\iota^k(p)^*"]
&\onecoalg{\mathbb{T}(S^{k-1}(p))}~.
\end{tikzcd}
\]

\item Given a tower of dg operads
\[
\oneop_0 \hookrightarrow \oneop_1 \hookrightarrow \cdots \hookrightarrow \oneop_\alpha \hookrightarrow \cdots \hookrightarrow \colim_{\alpha} \oneop_\alpha \cong \oneop~, 
\]
the category of dg $\oneop$-coalgebras is equivalent to the limit of the tower 
\[
\onecoalg{\oneop_0} \twoheadleftarrow \onecoalg{\oneop_1} \twoheadleftarrow \cdots \twoheadleftarrow \onecoalg{\oneop_\alpha} \twoheadleftarrow  \cdots \twoheadleftarrow \lim_{\alpha} \onecoalg{\oneop_\alpha} \simeq \onepcoalg~. 
\]
\end{enumerate}
\end{corollary}

\begin{proof}
Follows directly from \cref{thm: 1-categorical adjunction between categories of coalgebras and coendomorphism operads}, since the functor that sends operads to their categories of coalgebras sends colimits to strict limits of categories.
\end{proof}

\section{\texorpdfstring{$\infty$}{Infinity}-categorical coalgebras over operads in an enriched setting}

In this section, we start by defining enriched $\infty$-operads as algebras in symmetric sequences, following the approach developed by Brantner in \cite[Section 4.1.2]{BrantnerPhD}. See also \cite[Section 5.2.4]{ShiPhD} for a thorough exposition of these ideas. This approach is equivalent to that of \cite{ChuHaugseng20} by the recent work of \cite{arakawa2026}. The main objective of this section is to define coalgebras over an enriched $\infty$-operad. In order to arrive at this definition, we adapt the ideas of \cite[Appendix A]{heuts2024}. Finally, we compare this definition with the one considered in \cite{ellipticI} and in \cite{peroux} in the non-enriched case, adapting the proof for the algebra case of \cite[Appendix A.2.]{ShiPhD}. 


\subsection{Symmetric sequences and operads}
We work over any compactly rigidly generated stable symmetric monoidal $\infty$-category $\mathbf{D}$. These are presentable stable symmetric monoidal $\infty$-categories, with compact generators which are furthermore dualizable. See \cite{ramzi2024locallyrigidinftycategories} for more details. Examples include $\mathbf{D}(A)$ for any discrete ring or $\mathbb{E}_\infty$-ring spectrum $A$, and quasi-coherent sheaves on a perfect derived stack \cite[Proposition 3.9]{bfn}. Nevertheless, we will eventually only use these definitions in the case of the $\infty$-category $\mathbf{D}(\kk)$ of chain complexes over a field $\kk$ up to quasi-isomorphisms. 

\medskip

As in the 1-categorical case (\cref{sec:1-operads}), we define the $\infty$-category of \textit{symmetric sequences} in $\mathbf{D}$ as the $\infty$-category of functors from $\mathsf{Fin}^{\simeq}$ to $\mathbf{D}$:
\[
\mathbf{sSeq}(\mathbf{D}) \coloneqq \mathbf{Fun}(\mathsf{Fin}^{\simeq},\mathbf{D})~. 
\]The $\infty$-category of symmetric sequences in $\mathbf{D}$ admits a monoidal structure given by the composition product $\circledcirc$, which for two symmetric sequences $M$ and $N$ is pointwise given by the same formula as in the 1-categorical case, 
\[
M \circledcirc N(n) \simeq \bigoplus_{k \geq 0} \left(\bigoplus_{\underline{n} = \sqcup_{i=1}^k S_i} M(k) \otimes N(S_1) \otimes \cdots \otimes N(S_k)\right)_{\mathbb{S}_k} 
\]
as computed in \cite[Appendix A.1.]{ShiPhD}. Using this composition product, one can define operads as monoids in a monoidal $\infty$-category.

\begin{definition}[$\infty$-operad]
An \emph{$\infty$-operad enriched in $\mathbf D$}  is  a monoid object in the $\infty$-category of symmetric sequences in $\mathbf D$, with respect to the composition product. We sometimes also say \emph{$\infty$-operad in $\mathbf D$}, or just \emph{$\infty$-operad}, if $\mathbf D$ is clear from context. 
\end{definition}

The $\infty$-category of operads in $\mathbf{D}$ is therefore given by the $\infty$-category of $\mathbb{E}_1$-algebras in  the monoidal $\infty$-category $(\mathbf{sSeq}(\mathbf{D}),\circledcirc, \mathbb{1})$. 


\begin{remark}[Point-set models for enriched in chain complexes $\infty$-operads]
When $\mathbf D = \mathbf D(\kk)$, this $\infty$-category is presented by the dg operads localized at quasi-isomorphisms. Indeed, there is an equivalence of $\infty$-categories
\[
\mathsf{Op}(\mathsf{Ch}(\kk))~[\mathsf{Q.iso}^{-1}] \simeq \mathbf{Alg}_{\mathbb{E}_1}(\mathbf{sSeq}(\mathbf{D}(\kk)))~. 
\]
This follows from the description of the free $\mathbb{E}_1$-algebra in $\mathbf{sSeq}(\mathbf{D})(\kk)$ given in \cite[Theorem B.2]{pdalgebras}, by applying the Barr--Beck--Lurie theorem of \cite[Theorem 4.7.3.5]{HigherAlgebra} and using the fact that dg symmetric sequences up to quasi-isomorphisms present the $\infty$-category of symmetric sequences in $\mathbf{D}(\kk)$. In particular, any dg operad $\mathcal{P}$ induces an $\infty$-operad in $\mathbf{D}(\kk)$. Our end goal is going to be to compare the $\infty$-category that one obtains by localizing dg $\mathcal{P}$-coalgebras with respect to quasi-isomorphisms and the $\infty$-category of coalgebras over the underlying $\infty$-operad of $\mathcal{P}$. However, in order to address this question, we need to be able to define coalgebras over a general $\infty$-operad. 
\end{remark}

\begin{remark}[About the rectification of $\mathcal{P}$-algebras in chain complexes]
To any $\infty$-operad $\mathscr{P}$, one can associate a monad on $\mathbf{D}(\kk)$ given by its Schur functor 
\[
\mathbf{S}(\mathscr{P}) = \bigoplus_{n \geq 0} \left(\mathscr{P}(n) \otimes (-)^{\otimes n}\right)_{\mathbb{S}_n}.
\]
The $\infty$-category of $\mathscr{P}$-algebras is defined to be the $\infty$-category of algebras over this monad. 

\medskip

Let $\mathcal{P}$ be a $\mathbb{S}$-cofibrant dg operad (meaning its underlying dg symmetric sequence is projective). The  1-categorical free-forgetful adjunction induces an $\infty$-categorical adjunction 
\[
\begin{tikzcd}[column sep=5pc,row sep=3pc]
          \onepalg~[\mathsf{Q.iso}^{-1}] \arrow[r, shift left=1.1ex, "U"{name=F}] & \mathbf{D}(\kk), \arrow[l, shift left=.75ex, "\mathrm{S}(\mathcal{P})(-)"{name=U}]
            \arrow[phantom, from=F, to=U, , "\dashv" rotate=90]
\end{tikzcd}
\]

between dg $\mathcal{P}$-algebras up to quasi-isomorphisms and chain complexes up to quasi-isomorphism. This adjunction is monadic in the $\infty$-categorical sense. Furthermore, the associated monad can be easily identified with the $\infty$-categorical monad that encodes algebras over the $\infty$-operad induced by $\mathcal{P}$. Therefore, by applying \cite[Theorem 4.7.3.5]{HigherAlgebra} we get an equivalence of $\infty$-categories
\[
\onepalg~[\mathsf{Q.iso}^{-1}] \simeq \mathbf{Alg}_{\mathcal{P}}(\mathbf{D}(\kk))~. 
\]
This situation is in sharp contrast with what happens for coalgebras. Since the cofree construction of Theorem \ref{thm: cofree coalgebra over an operad} is extremely hard to compute, it is far from obvious that the 1-categorical cofree-forgetful adjunction induces a comonadic adjunction in the $\infty$-categorical sense. Furthermore, we are not aware of any explicit description of comonads encoding $\infty$-categorical coalgebras, even in basic cases like $\mathbb{E}_\infty$-coalgebras. This is why we cannot use the same arguments to compare the two categories.
\end{remark}

\subsection{Defining coalgebras over an (enriched) $\infty$-operad}To a symmetric sequence $M$, we can associate a \emph{dual Schur functor} $\widehat{\mathbf{S}}^c(M)$ in $\mathbf{End}(\mathbf{D})$,  an $\infty$-categor\-ic\-al analogue of the point-set {dual Schur functor} from \cref{set:1-dual schur}. This construction  defines a functor 
\[
\begin{tikzcd}[column sep=3pc,row sep=0pc]
\widehat{\mathbf{S}}^c(-): \mathbf{sSeq}(\mathbf{D})^{\mathsf{op}} \arrow[r]
&\mathbf{End}(\mathbf{D}) \\
M \arrow[r,mapsto]
&\widehat{\mathbf{S}}^c(M) \coloneqq \displaystyle \prod_{n \geq 0} \left[M(n),(-)^{\otimes n}\right]^{\mathbb{S}_n} 
\end{tikzcd}
\]from the $\infty$-category of symmetric sequences in $\mathbf{D}$ to the $\infty$-category of endofunctors of $\mathbf{D}$. Here $[-,-]$ denotes the self-enrichment of $\mathbf{D}$, adjoint to the tensor product. This endofunctor can be presented by the point-set version of it when we take a $\mathbb{S}$-cofibrant model of the symmetric sequence $M$.

\begin{lemma}\label{lemma: lax monoidal structure}
    There is a natural comparison map 
    \[ \widehat{\mathbf{S}}^c(M) \circ \widehat{\mathbf{S}}^c(N) \longrightarrow \widehat{\mathbf{S}}^c(M \circledcirc N)\]  which endows the functor $\widehat{\mathbf{S}}^c(-)$ with a lax monoidal structure.
\end{lemma}

\begin{proof}
Pointwise, the lax monoidal structure is given as follows. For each $n$, there is a natural morphism
    \begin{gather*}\Big[M(n),\Big( \prod_{k \geq 0} [N(k),(-)^{\otimes k}]^{\mathbb{S}_k}  \Big)^{\otimes n}\Big] \\ \downarrow \\ \Big[M(n), \prod_{k_1,\dots,k_n} [N(k_1),(-)^{\otimes k_1}]^{\mathbb{S}_{k_1}} \otimes \dots \otimes [N(k_n),(-)^{\otimes k_n}]^{\mathbb{S}_{k_n}} \Big] \\ \downarrow \\ \prod_p \Big[M(n), [N^{\otimes n}(p), (-)^{\otimes p}  ]^{\mathbb S_p} \Big] \\ \downarrow~\simeq \\ \prod_p \Big[ M(n) \otimes N^{\otimes n}(p), (-)^{\otimes p} \Big]^{\mathbb S_p}
    \end{gather*} 
    in which the first arrow is induced by the natural transformation
    \[
    \Big(\prod_k A_k\Big)^{\otimes n} \longrightarrow \prod_{k_1,\dots,k_n} A_{k_1} \otimes \dots \otimes A_{k_n}.
    \]  
    These maps are all natural transformations and assemble into a lax monoidal structure.
\end{proof}

This lax monoidal structure implies that it sends enriched $\infty$-cooperads to monads, and thus that it can be used to define algebras over enriched $\infty$-cooperads, as we will do in \cite{poinsetkoszul}. However, as best as we know, there is no available theory of lax comonads in the $\infty$-categorical setting, so we cannot directly work with the image of an enriched $\infty$-operad by the dual Schur functor. 
\medskip

In order to obtain a well-behaved definition of coalgebras over an enriched $\infty$-operad, we adapt the methods of \cite[Appendix A]{heuts2024}. The rough idea is to extend to dual Schur functor to the $\infty$-category of pro-objects in $\mathbf{D}$ to make it monoidal, define $\mathscr{P}$-coalgebras in pro-objects of $\mathbf{D}$ and finally take the pullback of that $\infty$-category along the inclusion of $\mathbf{D}$ into $\mathbf{pro}(\mathbf{D})$.  This definition should be thought as encoding any type of coalgebra (without divided powers) in an enriched setting.

\subsubsection{Compact objects and compact generation of symmetric sequences.}
Recall that $\mathbf{D}$ is assumed to be a compactly rigidly generated symmetric monoidal $\infty$-category. Let us denote by $\mathbf{D}^\omega$ the $\infty$-category of its compact objects. There is an equivalence of $\infty$-categories
\[
\mathbf{Ind}(\mathbf{D}^\omega) \simeq \mathbf{D}~. 
\]
If $A, B$ are two objects in $\mathbf{D}^\omega$, there is a canonical equivalence of functors
\[
[A, -] \otimes [B, -] \simeq [A \otimes B, -]~. 
\]
In other words, one can pull tensor products inside internal homs when the objects are compact. We will crucially use this property to upgrade the dual Schur functor into a strong monoidal functor. Since $\mathbf{D}$ is compactly generated, so is $\mathbf{sSeq}(\mathbf{D})$,  and we have an equivalence of $\infty$-categories
\begin{equation}\label{eq: compact generation of symmetric sequences}
\mathbf{sSeq}(\mathbf{D}) \simeq \mathbf{Ind}(\mathbf{sSeq}(\mathbf{D})^{\omega})~.
\end{equation} The question whether $\mathbf{sSeq}(\mathbf{D})$ is moreover compactly \emph{rigidly} generated depends on the mono\-id\-al structure we choose on the functor category. If we set $\mathbf{sSeq}(\mathbf{D}) = \prod_n \mathbf{Fun}(B\mathbb S_n,\mathbf D)$ and use Day convolution on each factor $\mathbf{Fun}(B\mathbb S_n,\mathbf D)$, then it is compactly rigidly generated. Indeed write $c_!: \mathbf D = \mathbf{Fun}(\ast,\mathbf D) \to \mathbf{Fun}(B\mathbb S_n,\mathbf D)$ for the functor given by left Kan extension: then compact generators of $\mathbf{Fun}(B\mathbb S_n,\mathbf D)$ are given by $c_!(d)$ for $d \in \mathbf D^\omega$, each of which is dualizable with dual $c_!(d^\vee)$. 

\begin{example}
    If $\mathbf D = \mathbf{D}(\kk)$ is the derived category of a ring $\kk$, then 
    \[\mathbf{sSeq}(\mathbf{D})^\omega = \prod_n \mathbf{Perf}\, \kk[\mathbb S_n]. \]
    That is, compact symmetric sequences are those which in each arity $n$ is given by a bounded complex of projective $\kk[\mathbb S_n]$-modules. 
\end{example}

\subsubsection{Extending functors to pro-categories.} The $\infty$-category of pro-objects in $\mathbf{pro}(\mathbf{D})$, denoted by $\mathbf{pro}(\mathbf{D})$, is obtained by freely adjoining cofiltered limits to $\mathbf{D}$. For a definition, see \cite[Section 3.1]{Lurie11Rational}. There is a canonical functor 
\[
c: \mathbf{D} \longrightarrow \mathbf{pro}(\mathbf{D})
\]
which sends objects in $\mathbf{D}$ to constant pro-object. This functor is fully faithful, and it preserves all colimits and finite limits. Furthermore, since $\mathbf{D}$ is symmetric monoidal, there exists a symmetric monoidal structure on $\mathbf{pro}(\mathbf{D})$ such that $c$ becomes a symmetric monoidal functor.

\medskip

Finally, since the tensor product in $\mathbf{D}$ commutes with finite limits, then the tensor product in $\mathbf{pro}(\mathbf{D})$ commutes with arbitrary limits. We refer to \cite[Section 5]{bachmann2024} for more details.

\medskip

Since $\mathbf{D}$ is assumed to be presentable, it admits all limits. Therefore there exists a right adjoint $\mathrm{mat}$ to the functor $c$, called the \textit{materialization} functor
\[
\mathrm{mat}: \mathbf{pro}(\mathbf{D}) \longrightarrow \mathbf{D}
\] 
which is given by taking the limit in $\mathbf{D}$ of pro-objects. 

\medskip
	The universal property of pro-completion is that if $\mathbf E$ is any  $\infty$-category which admits cofiltered limits, then precomposition with $c$ defines an equivalence of $\infty$-categories
	\[ \mathbf{Fun}^{\mathbf{filt}}(\mathbf{pro}(\mathbf D),\mathbf E) \simeq \mathbf{Fun}(\mathbf D,\mathbf E).\]
	Let $F$ be an endofunctor of $\mathbf D$. Under the above equivalence, $c \circ F$ corresponds to a unique endo\-functor $\mathrm p (F)$ of $\mathbf{pro}(\mathbf{D})$ preserving cofiltered limits, which we call the \emph{prolongation} of $F$, following Heuts \cite[Appendix A]{heuts2024}. 

\begin{lemma}
There is an adjunction 
\[
\begin{tikzcd}[column sep=5pc,row sep=3pc]
            \mathbf{End}(\mathbf{D})\arrow[r,"\mathrm{p}"{name=F},shift left=1.1ex ] &\mathbf{End}^{\mathbf{filt}}(\mathbf{pro}(\mathbf{D})) \arrow[l,"\mathrm{mat} \,\circ \,(-)\, \circ\, c"{name=U},shift left=1.1ex ]
            \arrow[phantom, from=F, to=U, , "\dashv" rotate=-90]
\end{tikzcd}
\]

between the $\infty$-category of endofunctors of $\mathbf{pro}(\mathbf{D})$ which preserve cofiltered limits and the $\infty$-category of endofunctors in $\mathbf{D}$. The functor $\mathrm{p}$ is prolongation and its right adjoint is given by pre-composing with $c$ and post-composing with $\mathrm{mat}$. 
\end{lemma} 

\begin{proof}
Let $F$ be an endofunctor of $\mathbf{D}$ and $G$ an endofunctor of $\mathbf{pro}(\mathbf{D})$. Then
\begin{align*}
\mathrm{Map}_{\mathbf{End}^{\mathbf{filt}}(\mathbf{pro}(\mathbf{D}))}(\mathrm{p}(F),G) &\simeq \mathrm{Map}_{\mathbf{Fun}(\mathbf{D},\mathbf{pro}(\mathbf{D}))}(c \circ F ,G \circ c) \\
&\simeq \mathrm{Map}_{\mathbf{End}(\mathbf{D})}(F,\mathrm{mat} \circ G \circ c)~. \qedhere
\end{align*} 
\end{proof}

\subsubsection{The prolongation of the truncated dual Schur functors.} Now we consider the truncated versions of the dual Schur functor, which for a symmetric sequence $M$ are given by 
\[
\widehat{\mathbf{S}}^c_{\leq k}(M)(-) \coloneqq \displaystyle \bigoplus_{n \geq 0}^{k} \left[M(n),(-)^{\otimes n} \right]^{\mathbb{S}_n}~. 
\]
Our goal is to extend them to $\mathbf{pro}(\mathbf{D})$ and take the formal limit over $k \geq 1$. However, we want to do so in a way where the resulting functor will preserve all colimits. We consider the following composition of functors 
\[
\begin{tikzcd}
\left(\mathbf{sSeq}(\mathbf{D})^{\omega}\right)^{\mathsf{op}}\arrow[r,"\widehat{\mathbf{S}}^c_{\leq k}"]
&\mathbf{End}(\mathbf{D}) \arrow[r,"\mathsf{p}"]
&\mathbf{End}^{\mathbf{filt}}(\mathbf{pro}(\mathbf{D}))
\end{tikzcd}
\]
that is, the prologation of the truncated dual Schur functor restricted to the compact generators of symmetric sequences. The $\infty$-category $\mathbf{End}^{\mathbf{filt}}(\mathbf{pro}(\mathbf{D}))$ is stable under limits, therefore the above functor admits an unique extension to all symmetric sequences in $\mathbf{D}$, which we denote by 
\[
\quo{\lim_\alpha} ~\mathrm{p} \widehat{\mathbf{S}}^c_{\leq k}(-): \left(\mathbf{sSeq}(\mathbf{D})\right)^{\mathsf{op}} \longrightarrow \mathbf{End}^{\mathbf{filt}}(\mathbf{pro}(\mathbf{D}))~. 
\]

\subsubsection{Prolongation of the enriched mapping space.} Recall that an object $X$ of a symmetric mono\-id\-al ($\infty$-)category is said to be \emph{exponentiable} if the functor $X \otimes -$ admits a right adjoint. When this is the case, we may denote the adjoint by $[X,-]$. (So a symmetric monoidal category is closed precisely when all objects are exponentiable.) The symmetric monoidal structure on the pro-category $\mathbf{pro}(\mathbf{D})$ is not in general closed. However, the objects of $\mathbf D$ are exponentiable in $\mathbf{pro}(\mathbf{D})$. Heuristically, if $V$ is in $\mathbf D$ and $Y = \quo{\lim_\alpha} Y_\alpha$ is in $\mathbf{pro}(\mathbf D)$, then we may define the exponential $[V,Y]$ as the pro-object $\quo{\lim_\alpha} [V,Y_\alpha]$. 

\begin{lemma}
For any object $V$ in $\mathbf D$, the functor $\mathrm{p}\left([V,-]\right)$ is right adjoint to $c(V) ~ \widehat{\otimes} ~ -$.
\end{lemma}
\begin{proof}
    Let $Y=\quo{\lim_\alpha} Y_\alpha$ and $Z=\quo{\lim_\beta} Z_\beta$ be objects of $\mathbf{pro}(\mathbf D)$. We have
    \begin{align*}
        \mathrm{map}_{\mathbf{pro}(\mathbf{D})}(c(V) \widehat{\otimes} Y, Z) & \simeq \mathrm{map}_{\mathbf{pro}(\mathbf{D})}(\quo{\lim_\alpha} V \otimes Y_\alpha, Z) \\
        & \simeq \colim_\alpha \lim_\beta \mathrm{map}_{\mathbf{D}}( V \otimes Y_\alpha, Z_\beta) \\
        & \simeq \colim_\alpha \lim_\beta \mathrm{map}_{\mathbf{D}}( Y_\alpha, [V,Z_\beta]) \\
        & \simeq  \mathrm{map}_{\mathbf{pro}(\mathbf{D})}(Y, \quo{\lim_\beta} [V,Z_\beta])
    \end{align*}
    and $\quo{\lim_\beta} [V,Z_\beta]$ is $p([V,-])$ evaluated on $Z=\quo{\lim_\beta} Z_\beta$.
\end{proof}

In the following, we will write $\left[c(V),-\right]_{\mathrm{pro}}$ for the hom-object out of $c(V)$ in $\mathbf{pro}(\mathbf{D})$, to avoid confusion with the hom-object $[V,-]$ in $\mathbf D$.

\begin{lemma}
Let $M$ be a symmetric sequence, written as a filtered colimit
\[
M \simeq \colim_{\alpha} M_\alpha 
\]
of compact symmetric sequences $M_\alpha$. There is a natural equivalence of functors
\[
\quo{\lim_\alpha}\mathrm{p}\widehat{\mathbf{S}}^c_{\leq k}(M) \simeq \quo{\lim_\alpha} \bigoplus_{n \geq 0}^{k} \left[c(M_{\alpha})(n),(-)^{\widehat{\otimes} n} \right]^{\mathbb{S}_n}_{\mathrm{pro}},
\]
where $\widehat{\otimes}$ denotes the symmetric monoidal structure of the $\infty$-category $\mathbf{pro}(\mathbf{D})$. 
\end{lemma} 

\begin{proof}
Let us start computing this functor for a given compact symmetric sequence $M_\alpha$. The prolongation functor is strong monoidal, so we only have to compute the prologation of each of the pieces that compose the truncated dual Schur functor. 

The prolongations of the tensor product is given by the tensor product in $\mathbf{pro}(\mathbf{D})$ 
\[
\mathrm{p}(-)^{\otimes n} \simeq (-)^{\widehat{\otimes} n}
\]
since the inclusion $c$ from $\mathbf{D}$ to $\mathbf{pro}(\mathbf{D})$ is strong monoidal. Similarly, for any $V$ in $\mathbf{D}$, we have $ \mathrm{p}[V,-] = \left[c(V),-\right]_{\mathrm{pro}}$  by definition. 
Finally, since $\mathrm{p}$ is a left adjoint, it preserves coproducts, and we have that 
\[
\mathrm{p}\left(\bigoplus_{n \geq 0}^{k} \left[M_\alpha(n),(-)^{\otimes n} \right]^{\mathbb{S}_n}\right) \simeq \bigoplus_{n \geq 0}^{k} \left[c(M_\alpha)(n),(-)^{\widehat{\otimes} n} \right]^{\mathbb{S}_n}_{\mathrm{pro}}~.
\]
We get the general version of this functor as a "formal" filtered limit of its image on compact symmetric sequences. 
\end{proof}

\begin{remark}[Warning]
For a given $Y$, the functor $\left[c(-),Y\right]_{\mathrm{pro}}$ does \emph{not} preserve colimits. In particular, it means that: 
\[
\mathrm{p}\widehat{\mathbf{S}}^c_{\leq k}(M) \not\simeq \quo{\lim_\alpha}\mathrm{p}\widehat{\mathbf{S}}^c_{\leq k}(M_{\alpha}). 
\]
Since we want a functor that preserves all colimits, we first restric it to compact symmetric sequences before taking the prolongation and then extending it freely to all symmetric sequences. We thank Brice Le Grignou for point out this mistake in the previous version of the paper.  
\end{remark}

\subsubsection{A formal tower of truncations.} Taking the formal limit of this family of truncated functors in the $\infty$-category of endofunctors of $\mathbf{pro}(\mathbf{D})$ gives a functor 
\[
\widehat{\mathbf{S}}^c_{\mathrm{pro}}(-) \coloneqq \quo{\lim_k} \quo{\lim_\alpha}\mathrm{p}\widehat{\mathbf{S}}^c_{\leq k}(-): \mathbf{sSeq}(\mathbf{D})^{\mathsf{op}} \longrightarrow \mathbf{End}(\mathbf{pro}(\mathbf{D}))
\]
gives a functor $\widehat{\mathbf{S}}^c_{\mathrm{pro}}(-)$, which we will refer to as the \textit{pro dual Schur functor}. Our goal from now will be to show that it sends colimits to limits and that it is in fact a strong monoidal functor. 

\begin{proposition}\label{prop: pro-Schur functor preserves limits}
The functor 
\[
\widehat{\mathbf{S}}^c_{\mathrm{pro}}(-): \mathbf{sSeq}(\mathbf{D})^{\mathsf{op}} \longrightarrow \mathbf{End}(\mathbf{pro}(\mathbf{D}))
\]
which assigns to a symmetric sequence its pro dual Schur functor preserves all limits. 
\end{proposition}

\begin{proof}
Follows from the fact that $\mathrm{p}\widehat{\mathbf{S}}^c_{\leq k}(-)$ preserves finite limits of compact symmetric sequences, so its unique extension to all symmetric sequences preserves all limits. 
\end{proof}

\begin{example}
Let $\mathbf{D} = \mathbf{D}(\kk)$, where $\kk$ is a field of positive characteristic. Let $M$ be the symmetric sequence given by the trivial representation, that is $M(n) \simeq \kk$ for all $n\geq 0$. Notice that $M$ is \textit{not} a compact symmetric sequence. Its image via the pro dual Schur functor coincides with the description of the cofree cocommutative coalgebra in $\mathbf{pro}(\mathbf{D})$ established in \cite[Lemma 5.1]{bachmann2024}. Indeed, by \cite[Remark 5.2]{bachmann2024} the product in that description has to be understood as the formal limit over the truncations $\quo{\lim_k}$ and the invariants as the formal inverse limit of finite limits which are computed indexwise, coming from the filtration of $B\mathbb{S}_n$ by finite skeleta, and which coincides with the formal limit $\quo{\lim_\alpha}$ in this case as well. 
\end{example}

\subsubsection{A strong monoidal extension of the dual Schur functor.}
We will show that the functor we have constructed is indeed strong monoidal. This, in turn, will allow us to produce a comonad in $\mathbf{pro}(\mathbf{D})$ out of an operad enriched in $\mathbf{D}$ and consider coalgebras over it.

\begin{lemma}\label{lemma: enriched pro-adjunction}
Let $A$ and $B$ be two objects of $\mathbf{D}$. There is a natural weak equivalence of endofunctors of $\mathbf{pro}(\mathbf{D})$
\[
\left[c(A),\left[c(B),-\right]_{\mathrm{pro}}\right]_{\mathrm{pro}} \simeq \left[c(A)~  \widehat{\otimes} ~ c(B),-\right]_{\mathrm{pro}}~. 
\]
\end{lemma}

\begin{proof}
It directly follows from 
\[
\left[c(A),\left[c(B),-\right]_{\mathrm{pro}}\right]_{\mathrm{pro}} \simeq \mathrm{p} \left( \left[A,\left[B,-\right]\right]\right) \simeq \mathrm{p} \left(\left[A \otimes B,-\right]\right) \simeq \left[c(A)~  \widehat{\otimes} ~ c(B),-\right]_{\mathrm{pro}}~.  \qedhere
\]\end{proof}

\begin{theorem}
The functor 
\[
\widehat{\mathbf{S}}^c_{\mathrm{pro}}(-): \mathbf{sSeq}(\mathbf{D})^{\mathsf{op}} \longrightarrow \mathbf{End}(\mathbf{pro}(\mathbf{D}))
\]
is a strong monoidal functor.  
\end{theorem}

\begin{proof}
Let $M$ and $N$ be two symmetric sequences. Let us write them as filtered colimits of compact symmetric sequences 
\[
M \simeq \colim_{\alpha} M_\alpha \quad \text{and} \quad N \simeq \colim_{\beta}~ N_{\beta}. 
\]
Under the equivalence $\mathbf{sSeq}(\mathbf{D}) \simeq \mathbf{Ind}(\mathbf{sSeq}(\mathbf{D})^\omega)$, the composition product $M \circledcirc N$ can be written as
\begin{equation}\label{equation: composition product ind of sym seq}
M \circledcirc N(k) \simeq \bigoplus_{n \geq 0} \colim_{\alpha} \colim_{\beta^{\times n}} \left(\bigoplus_{\underline{k} = \sqcup_{i=1}^n S_i} M_{\alpha}(n) \otimes N_{\beta}(S_1) \otimes \cdots \otimes N_{\beta}(S_n)\right)_{\mathbb{S}_n}
\end{equation}
since tensor products, coinvariants and direct sums commute with filtered colimits. 

\medskip

There is a first isomorphism 
\[
\quo{\lim_{k}} \quo{\lim_{\alpha}}\mathrm{p}\widehat{\mathbf{S}}^c_{\leq k}(M_\alpha) \circ \quo{\lim_{l}} \quo{\lim_{\beta}} \mathrm{p}\widehat{\mathbf{S}}^c_{\leq l}(N_\beta) \simeq \quo{\lim_{k,l}} \quo{\lim_{\alpha}} \mathrm{p}\widehat{\mathbf{S}}^c_{\leq k}(M_\alpha) \circ  \quo{\lim_{\beta}}  \mathrm{p}\widehat{\mathbf{S}}^c_{\leq l}(N_\beta)
\]
which directly follows from the fact that the first functor preserves filtered limits by construction. Let us compute the right hand side term:
\small{
\begin{align*}
&\quo{\lim_{\alpha}} \mathrm{p}\widehat{\mathbf{S}}^c_{\leq k}(M) \circ \quo{\lim_{\beta}} \mathrm{p}\widehat{\mathbf{S}}^c_{\leq l}(N) \simeq \bigoplus_{n \geq 0}^{k} \quo{\lim_{\alpha}}  \left[c(M_\alpha)(n),\left(\bigoplus_{j \geq 0}^{l} \quo{\lim_{\beta}} \left[c(N_\beta)(j),(-)^{\widehat{\otimes} j} \right]^{\mathbb{S}_j}_{\mathrm{pro}} \right)^{\widehat{\otimes} n} \right]^{\mathbb{S}_n}_{\mathrm{pro}} \\
&\simeq \bigoplus_{n \geq 0}^{k}  \quo{\lim_{\alpha}} \quo{\lim_{\beta^{\times n}}} \left[c(M_\alpha)(n),\bigoplus_{(i_1,\cdots,i_n), i_j \leq l} \left[c(N_\beta)(i_1),(-)^{\widehat{\otimes} i_1} \right]^{\mathbb{S}_{i_1}} \widehat{\otimes} \cdots \widehat{\otimes} \left[c(N_\beta)(i_n),(-)^{\widehat{\otimes} i_n} \right]^{\mathbb{S}_{i_n}} \right]^{\mathbb{S}_n}_{\mathrm{pro}} \\
&\simeq \bigoplus_{n \geq 0}^{k} \quo{\lim_{\alpha}} \quo{\lim_{\beta^{\times n}}}\left[\bigoplus_{(i_1,\cdots,i_n), i_j \leq l} c(M_\alpha)(n), \left[c(N_\beta)(i_1)\widehat{\otimes} \cdots \widehat{\otimes} c(N_\beta)(i_n),(-)^{\widehat{\otimes} (i_1 + \cdots + i_n)} \right]^{\mathbb{S}_{i_1} \times \cdots \times \mathbb{S}_{i_n}} \right]^{\mathbb{S}_n}_{\mathrm{pro}} \\
&\simeq \bigoplus_{n \geq 0}^{k}  \quo{\lim_{\alpha}} \quo{\lim_{\beta^{\times n}}} \left[c\left(M_\alpha(n) \otimes_{\mathbb{S}_n} \bigoplus_{(i_1,\cdots,i_n), i_j \leq l} N_\beta(i_1) \otimes \cdots N_\beta(i_n)\right), (-)^{\widehat{\otimes} (i_1 + \cdots + i_n)} \right]^{\mathbb{S}_{i_1 + \cdots + i_n}}_{\mathrm{pro}} 
\end{align*}}

where the third equivalence follows from the fact that $N_\beta(i_j)$ is a compact object for all $1 \leq j \leq n$, and where the fourth equivalence follows from Lemma \ref{lemma: enriched pro-adjunction}. It follows directly from the above that the canonical map
\[
\quo{\lim_{k}} \quo{\lim_{\alpha}}\mathrm{p}\widehat{\mathbf{S}}^c_{\leq k}(M_\alpha) \circ \quo{\lim_{l}} \quo{\lim_{\beta}} \mathrm{p}\widehat{\mathbf{S}}^c_{\leq l}(N_\beta) \longrightarrow \quo{\lim_{r}} \quo{\lim_{\gamma}} \mathrm{p}\widehat{\mathbf{S}}^c_{\leq r}((M  \circledcirc  N)_\gamma)~
\]
is an isomorphism and therefore that $\widehat{\mathbf{S}}^c_{\mathrm{pro}}$ is a strong monoidal functor. 
\end{proof}

\begin{definition}[$\mathscr{P}$-coalgebra]\label{def: infinity categorical P-coalgebras}
Let $\mathscr{P}$ be an enriched $\infty$-operad. The $\infty$-category of $\mathscr{P}$\textit{-coalgebras} is given by the following pullback
\[
\begin{tikzcd}[column sep=3.5pc,row sep=3.5pc]
\mathbf{Coalg}_\mathscr{P}(\mathbf{D}) \arrow[r] \arrow[d] \arrow[dr, phantom, "\ulcorner", very near start]
&\mathbf{Coalg}_{\widehat{\mathbf{S}}^c_{\mathrm{pro}}(\mathscr{P})}(\mathbf{pro}(\mathbf{D})) \arrow[d,"U"]\\
\mathbf{D} \arrow[r,"c"] 
&\mathbf{pro}(\mathbf{D})
\end{tikzcd}
\]
taken in the $\infty$-category of $\infty$-categories. 
\end{definition}

\begin{remark}
Given an object $V$ in $\mathbf{D}$, a $\mathscr{P}$-coalgebra structure on $V$ should be equivalent to a map of operads from $\mathscr{P}$ to a \textit{coendomorphism operad} of $V$. However, as far as we know, no such construction exists yet in the $\infty$-categorical setting with the desired properties.
\end{remark}

\subsection{Dévissage of the $\infty$-category of coalgebras over an operad}
The goal of this section is to show that the functor which assigns to every $\infty$-operad $\mathscr{P}$ its $\infty$-category of $\mathscr{P}$-coalgebras in $\mathbf{D}$ preserves pushouts and sends free $\infty$-operads to $\infty$-categories of coalgebras over endofunctors. 

\medskip

Let us recall the construction of the free $\infty$-operad given by \cite[Appendix B]{pdalgebras}, based on the $1$-categorical construction of \cite{Kelly80}. Let $M$ be a symmetric sequence, the free operad $T(M)$ on $M$ is build inductively as follows. We set 
\[
T^{(0)} \coloneqq \mathbb{1} \quad \text{and} \quad T^{(n)} \coloneqq \mathbb{1} \oplus \left(M \circ T^{(n-1)}(M)\right)~,
\]
together with the maps $i_1: \mathbb{1} \longrightarrow \mathbb{1} \oplus M$ given by the obvious inclusion and $i_n = \mathrm{id}_{\mathbb{1}} \oplus (\mathrm{id}_{M} \circ i_{n-1})$. This gives a sequential diagram of symmetric sequences, and by \cite[Theorem B.2]{pdalgebras} the free operad $T(M)$ exists and its underlying symmetric sequence is given by

\[
T(M) \simeq \colim_n ~ T^{(n)}(M)~. 
\]

Dually, for any functor $F$ in the $\infty$-category of endofunctors of $\mathbf{pro}(\mathbf{D})$, the cofree comonad on $F$ is constructed in \cite[Appendix A.3]{heuts2024} as follows. We set 
\[
C^{(0)} \coloneqq \mathbb{1} \quad \text{and} \quad C^{(n)} \coloneqq \mathbb{1} \times \left(F \circ T^{(n-1)}(F)\right)~,
\]
together with the maps $p_1: \mathbb{1} \times F \longrightarrow \mathbb{1}$ given by the obvious projection and $p_n = \mathrm{id}_{\mathbb{1}} \times (\mathrm{id}_{F} \circ p_{n-1})$. This gives a sequential diagram of endofunctors; the cofree comonad on $F$ thus exists and its endofunctor is given by 

\[
C(M) \simeq \lim_n ~ C^{(n)}(F)~. 
\]

\begin{proposition}\label{prop: sends free operad to cofree monad}
The functor 
\[
\widehat{\mathbf{S}}^c_{\mathrm{pro}}(-): \mathbf{sSeq}(\mathbf{D})^{\mathsf{op}} \longrightarrow \mathbf{End}(\mathbf{pro}(\mathbf{D}))
\]
sends a free operad $T(M)$ on a symmetric sequence $M$ to the cofree comonad on $\widehat{\mathbf{S}}^c_{\mathrm{pro}}(M)$. 
\end{proposition}

\begin{proof}
Follows directly from Proposition \ref{prop: pro-Schur functor preserves limits}.
\end{proof}

\begin{corollary}\label{corollary: coalgebras over a free infinity operad}
Let $M$ be a symmetric sequence. The $\infty$-category of coalgebras in $\mathbf{D}$ over the free operad $T(M)$ is equivalent to the $\infty$-category of coalgebras in $\mathbf{D}$ over the dual Schur endofunctor associated to $M$, given by: 
\[
\widehat{\mathbf{S}}^c(M) = \prod_{n \geq 0} \left[M(n),(-)^{\otimes n}\right]^{\mathbb{S}_n} 
\]
\end{corollary}

\begin{proof}
Combining Proposition \ref{prop: sends free operad to cofree monad} with the dual version of \cite[Remark B.4]{pdalgebras}, there is an equivalence of $\infty$-categories
\[
\mathbf{Coalg}_{\widehat{\mathbf{S}}^c_{\mathrm{pro}}(T(M))}(\mathbf{pro}(\mathbf{D})) \simeq \mathbf{coalg}_{\widehat{\mathbf{S}}^c_{\mathrm{pro}}(M)}(\mathbf{pro}(\mathbf{D}))
\]
between coalgebras over $\widehat{\mathbf{S}}^c_{\mathrm{pro}}(T(M))$ and coalgebras over the endofunctor $\widehat{\mathbf{S}}^c_{\mathrm{pro}}(M)$, since the first is the cofree comonad on the second. Coalgebras over this endofunctor are just objects $V$ in $\mathbf{pro}(\mathbf{D})$ endowed with a map 
\[
V \longrightarrow \quo{\lim_k} \quo{\lim_\alpha} \bigoplus_{n \geq 0}^{k} \left[c(M_\alpha)(n),V^{\widehat{\otimes} n} \right]^{\mathbb{S}_n}_{\mathrm{pro}}~. 
\]
Our goal is now to compute the pullback that appears in Definition \ref{def: infinity categorical P-coalgebras}. Objects in this category correspond to objects $W$ in $\mathbf{D}$ together with a map
\[
cW \longrightarrow  \quo{\lim_k} \quo{\lim_\alpha} \bigoplus_{n \geq 0}^{k} \left[c(M_\alpha)(n),(cW)^{\widehat{\otimes} n} \right]^{\mathbb{S}_n}_{\mathrm{pro}}~,
\]
and the data of such a map, by the adjunction $c \dashv \mathrm{mat}$, is equivalent to the data of a map
\[
W \longrightarrow \mathrm{mat}\left(\quo{\lim_k} \quo{\lim_\alpha} \bigoplus_{n \geq 0}^{k} \left[c(M_\alpha)(n),(cW)^{\widehat{\otimes} n} \right]^{\mathbb{S}_n}_{\mathrm{pro}}\right)~. 
\]
Finally, we can conclude by computing that
\[
\mathrm{mat}\left(\quo{\lim_k} \quo{\lim_\alpha} \bigoplus_{n \geq 0}^{k} \left[c(M_\alpha)(n),(cW)^{\widehat{\otimes} n} \right]^{\mathbb{S}_n}_{\mathrm{pro}}\right) \simeq \prod_{n \geq 0} \left[M(n),W^{\otimes n} \right]^{\mathbb{S}_n} \simeq \widehat{\mathbf{S}}^c(M)(W)~. 
\]
\end{proof}

\begin{proposition}\label{prop: taking infinity categories of coalgebras preserves limits}
The functor 
\[
\begin{tikzcd}[column sep=3.5pc,row sep= 0pc]
\mathbf{Op}(\mathbf{D})^{\mathsf{op}} \arrow[r]
&(\mathbf{Cat}_\infty)_{/\mathbf{D}} \\
\mathscr{P} \arrow[r,mapsto]
&\mathbf{Coalg}_\mathscr{P}(\mathbf{D})
\end{tikzcd}
\]
preserves all limits.
\end{proposition}

\begin{proof}
This proof is dual to point $(c)$ in the proof of \cite[Theorem 5.2]{heuts2024}. The functor \[
\mathbf{coMonad}^{\mathrm{filt}}(\mathbf{pro}(\mathbf{D})) \longrightarrow (\mathbf{Cat}_\infty)_{/\mathbf{pro}(\mathbf{D})}
\]which assigns to a comonad which preserves cofiltered limits on $\mathbf{pro}(\mathbf{D})$ its $\infty$-category of coalgebras in $\mathbf{pro}(\mathbf{D})$ preserves all limits, as explained in the proof of Lemma A.7 in \cite{heuts2024}. We are left to check that the functor \[
\widehat{\mathbf{S}}^c_{\mathrm{pro}}(-): \mathbf{Op}(\mathbf{D})^{\mathsf{op}} \longrightarrow \mathbf{coMonad}^{\mathrm{filt}}(\mathbf{pro}(\mathbf{D}))
\]preserves all limits. Since sifted limits in $\mathbf{Op}(\mathbf{D})^{\mathsf{op}}$ and in $\mathbf{coMonad}^{\mathrm{filt}}(\mathbf{pro}(\mathbf{D}))$ are respectively computed in their ground $\infty$-categories of $\mathbf{sSeq}(\mathbf{D})^{\mathsf{op}}$ and $\mathbf{End}^{\mathrm{filt}}(\mathbf{pro}(\mathbf{D}))$, respectively, the result for sifted limits follows directly from Proposition \ref{prop: pro-Schur functor preserves limits}. Thus we only need to check that it preserves products, which follows from arguments completely analogous to the algebras over an operad case explained in point $(c)$ of the proof of \cite[Theorem 5.2]{heuts2024}. 
\end{proof}

\begin{corollary}
Let $\mathscr{P}$ be an $\infty$-operad. The $\infty$-category of $\mathscr{P}$-coalgebras is comonadic over the base $\infty$-category $\mathbf{D}$. 
\end{corollary}

\begin{proof}
The fully faithful inclusion $c: \mathbf{D} \longrightarrow \mathbf{pro}(\mathbf{D})$ admits a right adjoint, hence it is comonadic, and thus pullback along two comonadic $\infty$-categories remains comonadic.
\end{proof}

\subsection{Comparison with the non-enriched case}
Let $\mathscr{O}$ be a one-coloured $\infty$-operad in the sense of Lurie, as introduced in \cite[Chapter 2]{HigherAlgebra}. They can be thought as enriched $\infty$-operads which are enriched in spaces. 

\begin{definition}[$\mathscr{O}$-coalgebras]\label{Def: O-coalgebras}
Let $\mathbf{C}$ be a symmetric monoidal $\infty$-category. The $\infty$-category of $\mathscr{O}$-coalgebras in $\mathbf{C}$ is defined as:
\[
\mathbf{Coalg}_{\mathscr{O}}(\mathbf{C}) \coloneqq \left(\mathbf{Alg}_{\mathscr{O}}(\mathbf{C}^{\mathsf{op}}))\right)^{\mathsf{op}}~. 
\]
\end{definition}

\begin{remark}
Recall that if $\mathbf{C}$ is a symmetric monoidal $\infty$-category, so is $\mathbf{C}^{\mathsf{op}}$. Hence it makes sense, using Lurie's definition of algebras over an $\infty$-operad, to consider $\mathscr{O}$-algebras in $\mathbf{C}^{\mathsf{op}}$.
\end{remark}

There is a unique functor $F : \mathbf{Spc}\longrightarrow \mathbf D$ which is cocontinuous and takes a point to the monoidal unit in $\mathbf D$; that is, $F(X)=\colim_X \mathbb 1$. Then $F$ is symmetric monoidal, since $F(X) \otimes F(Y) = \colim_X \mathbb 1 \otimes \colim_Y \mathbb 1 = \colim_{X \times Y} \mathbb 1 = F(X\times Y)$, using that $-\otimes-$ preserves colimits in both variables.
As explained in \cite[Appendix A.2]{ShiPhD}, any $\infty$-operad $\mathscr{O}$ in the sense of Lurie induces an $\infty$-operad  in spaces with underlying symmetric sequence $\{\mathscr{O}(r)\}_{r \geq 0}$, and since the functor $F$ is symmetric monoidal, the symmetric sequence $\{F(\mathscr{O}(r))\}_{r \geq 0}$ in $\mathbf{D}$ underlies the enriched $\infty$-operad $F(\mathscr{O})$ in $\mathbf{D}$.

\begin{theorem}\label{thm: comparison with coalgebras in the sense of Lurie}
Let $\mathscr{O}$ be an $\infty$-operad in the sense of Lurie. There is an equivalence of $\infty$-categories
\[
\mathbf{Coalg}_{\mathscr{O}}(\mathbf{D}) \simeq \mathbf{Coalg}_{F(\mathscr{O})}(\mathbf{D})~,
\]where on the one hand we consider Definition \ref{Def: O-coalgebras} over the $\infty$-operad $\mathscr{O}$ and on the other, we consider Definition \ref{def: infinity categorical P-coalgebras} over the enriched $\infty$-operad $F(\mathscr{O})$. 
\end{theorem}

\begin{proof}
It suffices to show that this holds in the case of $\mathbf{pro}(\mathbf{D})$, since then both $\infty$-categories of coalgebras can be obtained as the pullback along the inclusion $\mathbf{D} \longrightarrow \mathbf{pro}(\mathbf{D})$ of the $\infty$-categories of coalgebras in $\mathbf{pro}(\mathbf{D})$. 

\medskip

Let $\mathbf{C}$ be a presentable stable symmetric monoidal $\infty$-category. Let us first recall Shi's result in \cite[Theorem A.2.0.3]{ShiPhD}, which can be reformulated as follows: the $\infty$-category of $\mathscr{O}$-algebras in $\mathbf{C}$ is equivalent to the $\infty$-category of algebras over the following analytic monad:
\[
\bigoplus_{n \geq 0} \left(\colim_{\mathscr{O}(n)} (-)^{\otimes n} \right)_{\mathbb{S}_n}: \mathbf{C} \longrightarrow \mathbf{C}~, 
\]
where one takes the colimit of $(-)^{\otimes r}$ over the space $\mathscr{O}(r)$. 

\medskip

The $\infty$-category $\mathbf{pro}(\mathbf{D})^{\mathsf{op}}$ is presentable, since 
$
\mathbf{pro}(\mathbf{D})^{\mathsf{op}} \simeq \mathbf{ind}(\mathbf{D}^{\mathsf{op}}), 
$ and since the tensor product commutes with all colimits, it is a presentable symmetric monoidal $\infty$-category. Furthermore, it is also stable since $\mathbf{pro}(\mathbf{D})$ is stable by \cite[Lemma 2.5]{Ktheorynonarch} and since the opposite of a stable $\infty$-category is again stable. Hence we can apply the previous result to $\mathbf{pro}(\mathbf{D})^{\mathsf{op}}$ to compute the comonad that gives $\mathscr{O}$-coalgebras in $\mathbf{pro}(\mathbf{D})$, which are coalgebras over the comonad 
\[
\quo{\prod_{n \geq 0}} \left(\quo{\lim_{\mathscr{O}(n)}}(-)^{\widehat{\otimes} n} \right)^{\mathbb{S}_n}: \mathbf{pro}(\mathbf{D}) \longrightarrow \mathbf{pro}(\mathbf{D})~, 
\]
where the infinite product and the limit are taken in $\mathbf{pro}(\mathbf{D})$, which means that they are to be understood as filtered diagrams of finite limits in $\mathbf{D}$. On the other hand, the $\infty$-category of $F(\mathscr{O})$-coalgebras in $\mathbf{pro}(\mathbf{D})$ is given by coalgebras over the comonad
\[
\widehat{\mathbf{S}}^c(F(\mathscr{O})) \simeq \widehat{\mathbf{S}}^c(F(\colim_{\mathscr{O}} \{*\})) \simeq \quo{\lim_{\mathscr{O}}} \widehat{\mathbf{S}}^c(\mathbb{1}) \simeq \quo{\prod_{n \geq 0}} \left(\quo{\lim_{\mathscr{O}(n)}}(-)^{\widehat{\otimes} n} \right)^{\mathbb{S}_n}~. 
\]
Hence the two comonads agree on $\mathbf{pro}(\mathbf{D})$, and the $\infty$-categories of coalgebras are equivalent.  
\end{proof}

\section{Presenting homotopy coalgebras by point-set models}

A major issue in dealing with coalgebras is that the construction of the cofree coalgebra functor is extremely complicated. To bypass this problem, we use the dévissages of the categories of coalgebras that exist both at the 1-categorical and at the $\infty$-categorical level. This allows us to cut the problem into pieces: we first construct point-set models for coalgebras over free operads, and then we glue them along cell attachments to give point-set models for any quasi-free operad. This allows us to give point-set models for $\infty$-categorical coalgebras in terms of coalgebras in terms of coalgebras over any cofibrant operad.

\subsection{Point-set models for coalgebras over an endofunctor}\label{Subsection: point-set models for coalgebras over endofunctors}
Our base 1-category is the category of chain complexes over a field $\kk$ and our base $\infty$-category is the category of chain complexes over this field up to quasi-isomorphism. The later is obtained by localizing the former, thus we have an equivalence of $\infty$-categories 
\[
\mathsf{Ch}(\kk)~[\mathsf{Q.iso}^{-1}] \simeq \mathbf{D}(\kk)~. 
\]

\subsubsection{Coalgebras over an endofunctor} We start by considering an endofunctor 
\[
F: \mathsf{Ch}(\kk) \longrightarrow \mathsf{Ch}(\kk)
\]
on the underlying category of chain complexes. 

\begin{definition}[$F$-coalgebras]\label{definition: coalgebras over an endofunctor}
An dg $F$-coalgebra $V$ is the data $(V,\Delta_V)$ of a chain complex $V$ together with a structural map $\Delta_V: V \longrightarrow F(V)$.
\end{definition}

Morphisms of $F$-coalgebras are chain complex maps which commute with coalgebra structures. This 1-category is relatively simple in general, and in fact, can be written as a simple pullback diagram along the endofunctor $F$. 

\begin{lemma}\label{lemma: 1-categorical coalgebras over an endofunctor as a pullback}
The category of $F$-coalgebras is equivalent to the strict pullback of 1-categories
\[
\begin{tikzcd}[column sep=3.5pc,row sep=3.5pc]
F\text{-}\mathsf{coalg} \arrow[r] \arrow[d] \arrow[dr, phantom, "\ulcorner", very near start]
& \mathsf{Ch}(\kk)^{\mathsf{1}} \arrow[d,"{(\mathrm{ev}_0,\mathrm{ev}_1)}"]  \\
 \mathsf{Ch}(\kk) \arrow[r,"{(\mathrm{id},F)}"]
& \mathsf{Ch}(\kk) \times  \mathsf{Ch}(\kk)
\end{tikzcd}
\]
where $\mathsf{1}$ denotes the category with two objects and one arrow between them, and where the $\mathrm{ev}_0$ (resp. $\mathrm{ev}_1$) is the functor that evaluates at the sources (resp. the target). 
\end{lemma} 

\begin{proof}
It is straightforward to compute that an object in this pullback is precisely a chain complex $V$ together with a map $V \longrightarrow F(V)$, and that morphisms coincide too. 
\end{proof}

\begin{remark}
    If the functor $F$ preserves quasi-isomorphisms, then the pullback of Lemma \ref{lemma: 1-categorical coalgebras over an endofunctor as a pullback} is a pullback in relative categories as well. See Appendix \ref{Appendix A: Relative categories} for more on relative categories. 
\end{remark}

\subsubsection{Coalgebras over endofunctors in the $\infty$-categorical setting} We consider now an endofunctor at the $\infty$-categorical level
\[
\mathrm{T}: \mathbf{D}(\kk) \longrightarrow \mathbf{D}(\kk)
\]
of the derived category of $\kk$. Heuristically, we still want to define coalgebras over $\mathrm{T}$ as objects $V$ in $\mathbf{D}(\kk)$ equipped with a map $V \longrightarrow \mathrm{T}(V)$. However, in order to directly get a homotopy coherent definition, we adopt the following definition. 

\begin{definition}[$\mathrm{T}$-coalgebras]\label{definition: infinity categorical coalgebras over an endofunctor}
The $\infty$-category of $\mathrm{T}$\textit{-coalgebras} in $\mathbf{D}(\kk)$ is defined as the following pullback of $\infty$-categories
\[
\begin{tikzcd}[column sep=3.5pc,row sep=3.5pc]
\mathbf{coalg}_{\mathrm{T}}(\mathbf{D}(\kk)) \arrow[r] \arrow[d] \arrow[dr, phantom, "\ulcorner", very near start]
& \mathbf{D}(\kk)^{\Delta^1} \arrow[d,"{(\mathrm{ev}_0,\mathrm{ev}_1)}"]  \\
 \mathbf{D}(\kk) \arrow[r,"{(\mathrm{id}, \mathrm{T})}"]
& \mathbf{D}(\kk) \times  \mathbf{D}(\kk)~,
\end{tikzcd}
\]
taken in the $\infty$-category of $\infty$-categories. 
\end{definition}

This definition was also considered in \cite[Section 7]{heuts2024}, where he proves that this $\infty$-category is presentable if the endofunctor $\mathrm{T}$ is accessible. This, in turn, implies that the forgetful functor from $\mathbf{coalg}_{\mathrm{T}}(\mathbf{D}(\kk))$ to $\mathbf{D}(\kk)$ admits a right adjoint.

\subsubsection{Some comparison results} In order to show that the localization at quasi-isomorphisms of coalgebras over an endofunctor presents the $\infty$-category of coalgebras over its underlying endofunctor in $\mathbf{D}(\kk)$, we do the following. First, we compare the "strict" point-set model of dg $F$-coalgebras with a "relaxed" version of the definition ($F_\xi$-coalgebras below), and then we show that the image of this "relaxed" version under the subdivivded nerve is a homotopy pullback in complete Segal spaces which present the pullback of $\infty$-categories in Definition \ref{definition: infinity categorical coalgebras over an endofunctor}.

\medskip

\usetikzlibrary{arrows,decorations.pathmorphing,backgrounds,positioning,fit,petri,calc,intersections}

Let $\xi\mathsf{1}$ be the subdivision the relative category $\mathsf{1}$, with two objects, one arrow between them and no weak equivalence. This relative category is also given by $K_{\xi}\Delta[1,0]$, where the functors $K_{\xi}$ is the adjoint functor between bisimplicial sets and relative categories. For more on these notions, see Appendix \ref{Appendix A: Relative categories}. The operation $\xi$ corresponds in fact to a poset subdivision of the poset $\mathsf{1}$, see \cite[Section 4]{BarwickKan1} for more details on this operation. 

\begin{definition}[$F_\xi$-coalgebras]
Let $F$ be an endofunctor of $\mathsf{Ch}(\kk)$ which preserves quasi-isomorphism. The relative category of $F_\xi$-coalgebras is defined as the following pullback 
\[
\begin{tikzcd}[column sep=3pc,row sep=3pc]
F\text{-}\mathsf{coalg}^{\xi} \arrow[r] \arrow[d] \arrow[dr, phantom, "\ulcorner", very near start]
& \mathsf{Ch}(\kk)^{\xi\mathsf{1}} \arrow[d,"{(\mathrm{ev}_0,\mathrm{ev}_1)}"]  \\
 \mathsf{Ch}(\kk) \arrow[r,"{(\mathrm{id},F)}"]
& \mathsf{Ch}(\kk) \times  \mathsf{Ch}(\kk)
\end{tikzcd}
\]
in the category of relative categories. 
\end{definition}

\begin{lemma}\label{lemma: xi-coalgebras}
The canonical projection $\pi: \xi\mathsf{1} \longrightarrow \mathsf{1}$ of relative categories induces a homotopy equivalence of relative categories
\[
F\text{-}\mathsf{coalg}\simeq F\text{-}\mathsf{coalg}^{\xi}~. 
\]
\end{lemma}

\begin{remark}
Here, a homotopy equivalence of relative categories is the data of two functors going in opposite directions whose composites are homotopic to the identity. The notion of homotopy used is the one determined by $(-) \times\mathsf{1}^{\simeq}$, considered in \cite{BarwickKan1}.
\end{remark}

\begin{proof}
The canonical projection $\pi: \xi\mathsf{1} \longrightarrow \mathsf{1}$ is a homotopy equivalence of relative categories with an explicit homotopy inverse $\iota$ given by \cite[Proposition 7.3]{BarwickKan1}. This homotopy equivalence induces for any relative category $\mathsf{C}$, by precomposition in the \emph{internal hom} of relative categories, a homotopy equivalence 
\[
\mathsf{C}^{\check{1}}\simeq \mathsf{C}^{\xi\mathsf{1}}~,
\]
where both arrows commute with both factorizations of the diagonal, fitting in a commutative diagram 
\[
\begin{tikzcd}[column sep=2.5pc,row sep=2.5pc]
	& \mathsf{C} \ar[rd] \ar[ld] & \\
	\mathsf{C}^{\mathsf{1}} \ar[rd, "{(\mathrm{ev}_0,\mathrm{ev}_1)}"'] \ar[rr, "\pi^*"',bend right=10] & & \mathsf{C}^{\xi\mathsf{1}}~. \ar[ld, "{(\mathrm{ev}_0,\mathrm{ev}_1)}"] \ar[ll, "\iota^*"', bend right=10] \\
	& \mathsf{C} \times \mathsf{C} & 
\end{tikzcd}
\]
For more on the internal hom of relative categories, see Subsection \ref{subsection: internal homs}. The map $\iota^*$ and $\pi^*$ fit into the following commutative diagram

\[
\begin{tikzcd}[column sep=2pc,row sep=2.5pc, background color=clouddancer]
F\text{-}\mathsf{coalg}
  \arrow[rr] 
  \arrow[dd] 
  \arrow[dr]
&& F\text{-}\mathsf{coalg}^{\xi}
  \arrow[rr] 
  \arrow[dr] 
  \arrow[dd]
&& \mathsf{dg}\,F\text{-}\mathsf{coalg}
  \arrow[dd] 
  \arrow[dr] \\
& \mathsf{C}^{\mathsf{1}}
  \arrow[rr, "\pi^*", near start, crossing over] 
&& \mathsf{C}^{\xi\mathsf{1}}
  \arrow[rr, "\iota^*", near start, crossing over]
&& \mathsf{C}^{\mathsf{1}}
  \arrow[dd, swap, near start, "{(\mathrm{ev}_0,\mathrm{ev}_1)}"] \\
\mathsf{C}
  \arrow[rr, equals] 
  \arrow[dr, swap, "{(\mathrm{id},F)}"]
&& \mathsf{C}
  \arrow[rr, equals] 
  \arrow[dr, swap, "{(\mathrm{id},F)}"]
&& \mathsf{C}
  \arrow[dr, swap, "{(\mathrm{id},F)}"] \\
& \mathsf{C}\times\mathsf{C}
  \arrow[rr, equals] 
&& \mathsf{C}\times\mathsf{C}
  \arrow[rr, equals] 
&& \mathsf{C}\times\mathsf{C}
\arrow[from=2-2,to=4-2, "{(\mathrm{ev}_0,\mathrm{ev}_1)}", swap, near start, crossing over]
\arrow[from=2-4,to=4-4, near start, "{(\mathrm{ev}_0,\mathrm{ev}_1)}", crossing over]
\end{tikzcd}
\]

hence induce maps $\pi^*:F\text{-}\mathsf{coalg}\longrightarrow F\text{-}\mathsf{coalg}^{\xi}$ and $\iota^*:F\text{-}\mathsf{coalg}^{\xi}\longrightarrow F\text{-}\mathsf{coalg}$ at the level of pullbacks. It remains to check that the two homotopies $\pi^*\circ\iota^*\simeq \mathrm{Id}_{\mathsf{C}^{\mathsf{1}}}$ and $\iota^*\circ\pi^*\simeq \mathrm{Id}_{\mathsf{C}^{\xi\mathsf{1}}}$ lift at the level of pullbacks. For this aim, we first explain why they fit into commutative diagrams of spans as well, and secondly we explain why these induced maps at the level of pullbacks are also a homotopies. The argument is the same for both homotopies, so let us write it down for the first one. We consider a left homotopy
\[
\begin{tikzcd}[column sep=2.5pc,row sep=2.5pc]
\mathsf{C}^{\mathsf{1}} 
  \arrow[d] 
  \arrow[dr, "\pi^*~ \circ ~ \iota^*"] 
& \\
\mathsf{C}^{\mathsf{1}} \times \mathsf{1}^{\simeq} 
  \arrow[r, "\mathrm{H}"'] 
& \mathsf{C}^{\mathsf{1}}~.\\
\mathsf{C}^{\mathsf{1}} 
  \arrow[u] 
  \arrow[ur, "\mathrm{Id}"'] 
&
\end{tikzcd}
\]

All the arrows of this diagram define maps of spans (as explained for $\iota^*\circ\pi^*$ above), so they all induce maps between the corresponding pullbacks. Since $(-) \times\mathsf{1}^{\simeq}$ is the cartesian product, it commutes with limits so the two inclusions $\mathsf{C}^{\mathsf{1}}\hookrightarrow \mathsf{C}^{\mathsf{1}}\times \mathsf{1}^{\simeq}$ induce the canonical inclusions $F\text{-}\mathsf{coalg}\hookrightarrow F\text{-}\mathsf{coalg}\times\mathsf{1}^{\simeq}$ and $\mathrm{H}$ induces a map $F\text{-}\mathsf{coalg}\times\mathsf{1}^{\simeq}\longrightarrow F\text{-}\mathsf{coalg}$ commuting with the inclusions $\iota^*\circ\pi^*$ and $\mathrm{Id}_{F\text{-}\mathsf{coalg}}$. Hence it is a left homotopy between these two maps and we can conclude.
\end{proof}

\begin{corollary}\label{cor:piquasicateq}
The projection $\pi$ induces an equivalence of quasi-categories
\[
i_1^*N_{\xi}\pi^*:i_1^*N_{\xi}F\text{-}\mathsf{coalg}\stackrel{\sim}{\rightarrow}i_1^*N_{\xi}F\text{-}\mathsf{coalg}^{\xi}
\]
with homotopy inverse $i_1^*N_{\xi}\iota^*$.
\end{corollary}

\begin{proof}
The main idea of the proof is to turn the homotopy equivalence of Lemma \ref{lemma: xi-coalgebras} into a weak equivalence. And, in order to do so, we are first going to show that any homotopy equivalence between fibrant relative categories induces a homotopy equivalence between their subdivided nerve, and then conclude using the fact that complete Segal spaces are bifibrant (which is not the case for fibrant relative categories). 

\medskip

Appendix \ref{Appendix A: Relative categories} provides a good functorial path object on fibrant objects in $\mathsf{RelCat}$ defined by $(-)^{\xi\mathsf{1}^{\simeq}}$. So, by abstract non-sense, any left homotopy with respect to $(-) \times \mathsf{1}^{\simeq}$ induces a homotopy with respect to $(-)^{\xi\mathsf{1}^{\simeq}}$. Recall also that $\xi\mathsf{1}^{\simeq}=K_{\xi}\Delta[0,1]$, so it follows that
\[
N_{\xi}\xi\mathsf{1}^{\simeq}=N_{\xi}K_{\xi}\Delta[0,1]\stackrel{\sim}{\leftarrow}\Delta[0,1]
\]
where the weak equivalence is induced by the unit of the adjunction between $K_{\xi}$ and $N_{\xi}$, which is always a Reedy weak equivalence by \cite[Proposition 10.3] {BarwickKan1}. For any pair of relative categories $\mathsf{C},\mathsf{D}$ where $\mathsf{C}$ is fibrant, we can compute that
\begin{eqnarray*}
\mathrm{Hom}_{\ho{\mathsf{bisSet}}}(N_{\xi}\mathsf{D},N_{\xi}(\mathsf{C}^{\xi\mathsf{1}^{\simeq}})) & \cong & \mathrm{Hom}_{\ho{\mathsf{RelCat}}}(\mathsf{D},\mathsf{C}^{\xi\mathsf{1}^{\simeq}})  \\
 & \cong &\mathrm{Hom}_{\ho{\mathsf{RelCat}}} (\mathsf{D}\times\xi\mathsf{1}^{\simeq},\mathsf{C}) \\
 & \cong & \mathrm{Hom}_{\ho{\mathsf{bisSet}}}(N_{\xi}(\mathsf{D}\times\xi\mathsf{1}^{\simeq}),N_{\xi}\mathsf{C}) \\
 & \cong & \mathrm{Hom}_{\ho{\mathsf{bisSet}}}(N_{\xi}\mathsf{D}\times N_{\xi}\xi\mathsf{1}^{\simeq},N_{\xi}\mathsf{C}) \\
 & \cong & \mathrm{Hom}_{\ho{\mathsf{bisSet}}}(N_{\xi}(\mathsf{D}),N_{\xi}\mathsf{C}^{N_{\xi}\xi\mathsf{1}^{\simeq}}) \\
 & \cong & \mathrm{Hom}_{\ho{\mathsf{bisSet}}}(N_{\xi}(\mathsf{D}),N_{\xi}\mathsf{C}^{\Delta[0,1]})~, 
\end{eqnarray*}

in the homotopy category $\ho{\mathsf{bisSet}}$ of bisimplicial sets with respect to Rezk model structure and the homotopy category $\ho{\mathsf{RelCat}}$ of relative categories with respect to the Barwick--Kan model structure. Indeed, since $K_\xi \dashv N_\xi$ is a Quillen adjunction, it induces an adjunction at the level of homotopy categories. Using that $N_\xi$ is essentially surjective (since it is an equivalence), it follows from Yoneda lemma that $N_{\xi}(\mathsf{C}^{\xi\mathsf{1}^{\simeq}}) \simeq N_{\xi}(\mathsf{C})^{\Delta[0,1]}$ in $\ho{\mathsf{bisSet}}$. Consequently, a homotopy equivalence of fibrant relative categories provided induces a homotopy equivalence of complete Segal spaces. Therefore, using Lemma \ref{lemma: xi-coalgebras}, we conclude that $N_{\xi}\pi^*:i_1^*N_{\xi}F\text{-}\mathsf{coalg}$ and $N_{\xi}\pi^*:i_1^*N_{\xi}F\text{-}\mathsf{coalg}^{\xi}$ are homotopy equivalent complete Segal spaces. And since cofibrations in Rezk's model structure are monomorphisms, any complete Segal spaces is bifibrant. Hence, by Whitehead's theorem, any homotopy equivalence between these objects is a  weak equivalence. The final result then follows by applying the right Quillen functor $i_1^*$ to these two fibrant objects. 
\end{proof}

\begin{theorem}\label{thm: rectification of coalgebras over an endofunctor}
Let 
\[
F: \mathsf{Ch}(\kk) \longrightarrow \mathsf{Ch}(\kk)
\]
be an endofunctor of the category of chain complexes which preserves quasi-isomorphisms. There is an equivalence of 
$\infty$-categories
\[
F\text{-}\mathsf{coalg}~[\mathsf{Q.iso}^{-1}] \simeq \mathbf{coalg}_{F}(\mathbf{D}(\kk))
\]
between the $\infty$-category obtained by localizing $F$-coalgebras with respect to quasi-isomorphisms and the $\infty$-category of coalgebras over the endofunctor of $\mathbf{D}(\kk)$ induced by $F$. 
\end{theorem}

\begin{proof}
By Lemma \ref{lemma: 1-categorical coalgebras over an endofunctor as a pullback}, the 1-category of $F$-coalgebras is given by the following pullback 
\[
\begin{tikzcd}[column sep=3.5pc,row sep=3.5pc]
F\text{-}\mathsf{coalg} \arrow[r] \arrow[d] \arrow[dr, phantom, "\ulcorner", very near start]
& \mathsf{Ch}(\kk)^{\mathsf{1}} \arrow[d,"{(\mathrm{ev}_0,\mathrm{ev}_1)}"]  \\
 \mathsf{Ch}(\kk) \arrow[r,"{(\mathrm{id},F)}"]
& \mathsf{Ch}(\kk) \times  \mathsf{Ch}(\kk)~.
\end{tikzcd}
\]
Since $F$ preserves quasi-isomorphisms, by Lemma \ref{lemma: xi-coalgebras} this pullback of relative categories is equivalent to the following pullback 
\[
\begin{tikzcd}[column sep=3.5pc,row sep=3.5pc]
F\text{-}\mathsf{coalg}^{\xi} \arrow[r] \arrow[d] \arrow[dr, phantom, "\ulcorner", very near start]
&\mathsf{Ch}(\kk)^{\xi\mathsf{1}} \arrow[d,"{(\mathrm{ev}_0,\mathrm{ev}_1)}"]  \\
\mathsf{Ch}(\kk) \arrow[r,"{(\mathrm{id},F)}"]
&\mathsf{Ch}(\kk) \times  \mathsf{Ch}(\kk) 
\end{tikzcd}
\]
of relative categories. By Corollary \ref{cor:piquasicateq}, the relative category of $F_\xi$-coalgebras models the same underlying $\infty$-category. Every object in this pullback is fibrant (since they are model categories) and the evaluation map $(\mathrm{ev}_0,\mathrm{ev}_1)$ is a fibration by Proposition \ref{prop: evaluationfibrations}, so it is also a homotopy pullback. 

\medskip

The goal now is to compare this homotopy pullback with a homotopy pullback that we know models the pullback of $\infty$-categories of Definition \ref{definition: infinity categorical coalgebras over an endofunctor}. For that, we can apply the subdivided nerve $N_\xi$ to it and obtain a homotopy pullback of complete Segal spaces
\[
\begin{tikzcd}[column sep=2.5pc,row sep=3.5pc]
N_{\xi}(F\text{-}\mathsf{coalg}^\xi) \arrow[r] \arrow[d] \arrow[dr, phantom, "\ulcorner", very near start]
& N_{\xi}(\mathsf{Ch}(\kk)^{\xi\mathsf{1}}) \arrow[d,"{ N_{\xi}(\mathrm{ev}_0,\mathrm{ev}_1)}"]  \\
 N_{\xi}(\mathsf{Ch}(\kk)) \arrow[r,"{(\mathrm{id},F)}"]
& N_{\xi}(\mathsf{Ch}(\kk)) \times  N_{\xi}(\mathsf{Ch}(\kk))~.
\end{tikzcd}
\]

Let us make the maps involved in this pullback explicit. Since the subdivided nerve preserves products, $N_{\xi}(\mathsf{Ch}(\kk)\times\mathsf{Ch}(\kk)) \cong N_{\xi}(\mathsf{Ch}(\kk)) \times N_{\xi}(\mathsf{Ch}(\kk))$ and $N_{\xi}(\mathrm{id},F)=(\mathrm{id},N_{\xi}F)$. Since $N_\xi$ is right Quillen, this pullback is again a homotopy pullback. The goal now is to compare the span of this pullback 
\[
\begin{tikzcd}[row sep=large, column sep=large]
& N_{\xi}\bigl(\mathsf{Ch}(\kk)^{\xi\mathsf{1}}\bigr) 
    \arrow[d, "{N_{\xi}(\mathrm{ev}_0,\mathrm{ev}_1)}"] \\
N_{\xi}\mathsf{Ch}(\kk) 
    \arrow[r, "{N_{\xi}(\mathrm{id},F)}"'] 
& N_{\xi}\mathsf{Ch}(\kk)\times N_{\xi}\mathsf{Ch}(\kk)
\end{tikzcd}
\]
with the following span 
\[
\begin{tikzcd}[row sep=large, column sep=large]
& (N_{\xi}\mathsf{Ch}(\kk))^{\Delta[1,0]} 
    \arrow[d, "{(\mathrm{ev}_0,\mathrm{ev}_1)}"] \\
N_{\xi}\mathsf{Ch}(\kk) 
    \arrow[r, "{N_{\xi}(\mathrm{id},F)}"'] 
& N_{\xi}\mathsf{Ch}(\kk)\times N_{\xi}\mathsf{Ch}(\kk)
\end{tikzcd}
\]
that models on the nose the $\infty$-categorical pullback of Definition \ref{definition: infinity categorical coalgebras over an endofunctor}. Indeed, the evaluation map
\[
(\mathrm{ev}_0,\mathrm{ev}_1): (N_{\xi}\mathsf{Ch}(\kk))^{\Delta[1,0]}\longrightarrow N_{\xi}\mathsf{Ch}(\kk)\times N_{\xi}\mathsf{Ch}(\kk)
\]
is a fibration by Proposition \ref{prop: evaluationfibrations} and all the objects involved fibrant. Using that 
\[
N_{\xi}\xi\mathsf{1}=N_{\xi}K_{\xi}\Delta[1,0]\stackrel{\sim}{\leftarrow}\Delta[1,0]~, 
\]
we can use the same computation as in the proof of Corollary \ref{cor:piquasicateq} to show that the objects $N_{\xi}(\mathsf{C}^{\xi\mathsf{1}})$ and $N_{\xi}(\mathsf{C})^{\Delta[1,0]}$ are weak equivalent complete Segal spaces. It can be checked that this weak equivalence commutes with the respective evaluation maps and thus defines a weak equivalence of spans. Therefore, their (homotopy) pullbacks are weakly equivalent and we conclude that 
\[
F\text{-}\mathsf{coalg}^{\xi}~[\mathsf{Q.iso}^{-1}] \simeq \mathbf{coalg}_{F}(\mathbf{D}(\kk))~,
\]
and hence that 
\[
F\text{-}\mathsf{coalg}~[\mathsf{Q.iso}^{-1}] \simeq \mathbf{coalg}_{F}(\mathbf{D}(\kk)).
\]
\end{proof}

\subsection{Point-set models for coalgebras over a free operad}

The rectification of coalgebras over endofunctors is the first step towards the general rectification of coalgebras over general cofibrant dg operads. Indeed, since coalgebras over free operads are equivalent to coalgebras over the endofunctor of generators, Theorem \ref{thm: rectification of coalgebras over an endofunctor} directly implies the following result. 

\begin{proposition}\label{prop: rectification of coalgebras over free operads}
Let $M$ be an $\mathbb{S}$-projective dg symmetric sequence and let $\mathcal{P} = \mathbb{T}(M)$ be the free dg operad generated by $M$. There is an equivalence of $\infty$-categories
\[
\mathcal{P}\text{-}\mathsf{coalg}~[\mathsf{Q.iso}^{-1}] \simeq \mathbf{Coalg}_{\mathcal{P}}(\mathbf{D}(\kk))
\]
between dg $\mathcal{P}$-coalgebras up to quasi-isomorphism and coalgebras in $\mathbf{D}(\kk)$ over the induced enriched $\infty$-operad. 
\end{proposition}

\begin{proof}
Since $\mathcal{P}$ is the free dg operad generated by $M$, the category of dg $\mathcal{P}$-coalgebras is equivalent to the category of coalgebras over the dual Schur endofunctor 
\[
\widehat{\mathrm{S}}^c(M)(-) = \prod_{n\geq 0} [M(n),(-)^{\otimes n}]^{\mathbb{S}_n}
\]
associated to the dg symmetric sequence $M$. Furthermore, the functor $\widehat{\mathrm{S}}^c(M)(-)$ preserves quasi-isomorphisms since $M$ is $\mathbb{S}$-projective. Therefore we can apply Theorem \ref{thm: rectification of coalgebras over an endofunctor} to this situation, and get the following series of equivalences of $\infty$-categories
\[
\mathcal{P}\text{-}\mathsf{coalg}~[\mathsf{Q.iso}^{-1}] \simeq \widehat{\mathrm{S}}^c(M)\text{-}\mathsf{coalg}~[\mathsf{Q.iso}^{-1}] \simeq \mathbf{coalg}_{\widehat{\mathbf{S}}^c(M)}(\mathbf{D}(\kk))~. 
\]
Finally, since $\mathcal{P}$ is free on a cofibrant dg symmetric sequence, its enriched $\infty$-operad is also free and we can apply Corollary \ref{corollary: coalgebras over a free infinity operad} and get an equivalence 
\[
\mathbf{coalg}_{\widehat{\mathbf{S}}^c(M)}(\mathbf{D}(\kk)) \simeq \mathbf{Coalg}_{\mathcal{P}}(\mathbf{D}(\kk))
\]
between coalgebras over the $\infty$-categorical dual Schur endofunctor of $M$ and coalgebras over the enriched $\infty$-operad induced by $\mathcal{P}$. 
\end{proof}

\subsection{Gluing $\infty$-categories of coalgebras along cell attachments}
Recall that the model structure on $\mathsf{Ch}(\kk)$ is cofibrantly generated by the chain morphisms $S^{k-1}\hookrightarrow D^k$ and $0\hookrightarrow D^k$, where :
\begin{itemize}
    \item $S^{k-1}$ is defined by $\kk$ in degree $k-1$, and $0$ otherwise, with the trivial differentials;
    \item $D^k$ is defined by $\kk$ in degrees $k$ and $k-1$, and $0$ otherwise, with the non-trivial differential given by $d_n= \mathrm{id}_\kk$.
\end{itemize}

Taking the free dg symmetric sequences on these objects, we get let $S^{k}(p)$ given by $\kk[\mathbb{S}_p]$ in degree $k \in \mathbb{Z}$ and in arity $p \geq 0$ and by zero elsewhere and $D^k(p)$ by $\kk[\mathbb{S}_p]$ in degrees $k-1$ and $k$, for $k \in \mathbb{Z}$, with the the differential being the identity map, for some arity $p \geq 0$ and by zero elsewhere. The generating cofibrations of the semi-model structure of dg operads are given by taking the free dg operads on these dg symmetric sequences.

We now run the induction argument on cell attachments to get a rectification result for dg $\mathcal{P}$-coalgebras, where now $\mathcal{P}$ is any cofibrant cell dg operad. This induction follows from the subsequent key theorem.

\begin{theorem}\label{thm: key theorem}
Let $p \geq 0$ and $k \in \mathbb{Z}$, and let 
\[
\begin{tikzcd}[column sep=3.5pc,row sep=3.5pc]
\mathbb{T}(S^{k-1}(p)) \arrow[r,"\psi"] \arrow[d,"\iota^k(p)",swap]  \arrow[dr, phantom,"\ulcorner" rotate=-180, very near end]
&\oneop_\alpha \arrow[d,"\iota_\alpha"]  \\
\mathbb{T}(D^k(p)) \arrow[r,"\varsigma"]
&\oneop_{\alpha + 1}
\end{tikzcd}
\]
be a pushout diagram of cofibrant dg operads. The induced pullback of \emph{relative} categories
\[
\begin{tikzcd}[column sep=3.5pc,row sep=3.5pc]
\onecoalg{\oneop_{\alpha + 1}} \arrow[r,"\iota_\alpha^*"] \arrow[d,"\varsigma^*",swap] \arrow[dr, phantom, "\ulcorner", very near start]
&\onecoalg{\oneop_\alpha}\arrow[d,"\psi^*"]  \\
\onecoalg{\mathbb{T}(D^k(p))} \arrow[r,"\iota^k(p)^*"]
&\onecoalg{\mathbb{T}(S^{k-1}(p))},
\end{tikzcd}
\]
induces a \emph{homotopy pullback} of quasi-categories once we apply the right Quillen functor $i_1^*N_{\xi}$.
\end{theorem}

\begin{remark}[About the proof strategy]
Theorem \ref{thm: key theorem} would directly follow if 
\[
\iota^k(p)^*: (\onecoalg{\mathbb{T}(D^k(p))},\mathsf{Q.iso}) \longrightarrow (\onecoalg{\mathbb{T}(S^{k-1}(p))},\mathsf{Q.iso})
\]
were a fibration of relative categories, since all the objects in the pullback are fibrant. However, we do not know if this is the case, as there is no available characterization of fibrations in relative categories. So we will show this pullback becomes a homotopy pullback once we apply the functor $i_1^*N_{\xi}$. 
\end{remark}

\subsection{Proof of Theorem \ref{thm: key theorem}}\label{S: the gluing proof}
The pullback square in Theorem \ref{thm: key theorem} is clearly a pullback square of relative categories since the restriction functors involved preserve quasi-isomorphisms. Our main objective is to show that it is sent via the functor $i_1^*N_{\xi}$ to homotopy pullback of quasi-categories. Since all the relative categories are model categories, they are fibrant objects, which are preserved by the functor $i_1^*N_{\xi}$. Hence, it is enough to show the following result.

\begin{theorem}\label{thm: fibrationquasicat}
Let $p \geq 0$ and $k \in \mathbb{Z}$, the restriction functor 
\[
i_1^*N_{\xi}(\iota^k(p)^*):i_1^*N_{\xi}\left(\onecoalg{\mathbb{T}(D^k(p))} \right)\longrightarrow i_1^*N_{\xi} \left(\onecoalg{\mathbb{T}(S^{k-1}(p))}\right)
\]
is a fibration of quasi-categories.
\end{theorem}

The strategy of the proof is to check the conditions of Proposition \ref{prop:fibrationff}. For this, we need a collection of useful lemmas that we packed in the following subsections, each one corresponding to one of the conditions in Proposition \ref{prop:fibrationff}.

\subsubsection{Isofibration property at the homotopy category level} Here we prove that $\ho{\iota^k(p)^*}$ satisfies the first condition of Proposition \ref{prop:fibrationff}.

\begin{proposition}\label{prop: isofibration}
Let $p \geq 0$ and $k \in \mathbb{Z}$, the restriction functor 
\[
\ho{\iota^k(p)^*}: \ho{\onecoalg{\mathbb{T}(D^{k}(p))}} \longrightarrow \ho{\onecoalg{\mathbb{T}(S^{k-1}(p))}}
\]
is an isofibration.
\end{proposition}

\begin{proof}
Let $W$ be a dg $\mathbb{T}(D^k(p))$-coalgebra and $f:\iota^k(p)^*W \longrightarrow V$ be a dg $\mathbb{T}(S^{k-1}(p))$-coalgebra morphism and a quasi-isomorphism. The isofibration property translates into the fact that there exists a dotted arrow in the following  diagram 
\[
\begin{tikzcd}[row sep=3pc, column sep=3pc]
\iota^k(p)^* W  \arrow[d, "\Delta_{\iota^k(p)^*W}"] \arrow[r, "f"] \arrow[dd, bend right=50, "\Delta_W"']
&V \arrow[d, "\Delta_V"'] \arrow[dd, bend left=50, dashed, "\tilde{\Delta}_V"]\\
\left[S^{k-1}, W^{\otimes p}\right] \arrow[r, "f^{\otimes p}\circ (-)"] 
&\left[S^{k-1}, V^{\otimes p}\right] \\
\left[D^k, W^{\otimes p}\right] \arrow[r, "f^{\otimes p}\circ (-)"] \arrow[u, two heads, "(-) ~ \circ ~ \iota^k(p)"'] 
&\left[D^k, V^{\otimes p}\right] \arrow[u, two heads, "(-) ~\circ ~\iota^k(p)"]
\end{tikzcd}
\] 

making it commutative. We factor, in $\mathsf{Ch}(\kk)$ endowed with the projective model structure (which also agrees with the injective model structure) the underlying map $f: W\rightarrow V$ into an acyclic cofibration followed by an acyclic fibration: 
    \[
    W\underset{i}{\stackrel{\sim}\rightarrowtail} C \underset{q}{\stackrel{\sim}\twoheadrightarrow} V~.
    \]
We first prove that $C$ inherits a dg $\mathbb{T}(D^k(p))$-coalgebra structure compatible with the ones of $W$ whose restriction is compatible with the dg $\mathbb{T}(S^{k-1}(p))$-structure of $V$. 

The dotted arrow $\delta^k_C$ in the following commutative diagram
\[
\begin{tikzcd}[column sep=3pc,row sep=3pc]
W \arrow[d,rightarrowtail,"\simeq"] \arrow[r,"\Delta_W"] 
&\left[ D^k, W^{\otimes p}\right] \arrow[r,"f^{\otimes p} \circ (-)"]
&\left[ D^k, V^{\otimes p}\right] \arrow[d,"(-)~ \circ ~ \iota^k(p)",two heads] \\
C \arrow[r,"q"] \arrow[rru,"\delta_C",dashed]
&V \arrow[r,"\Delta_V"]
&\left[ S^{k-1},V^{\otimes p}\right]
\end{tikzcd}
\]
exists since the left vertical arrow is an acyclic cofibration and the right one a fibration. Indeed, the map $(-)~ \circ ~ \iota^k(p)$ is a fibration since the functor $\left[-,V^{\otimes p}\right]$ sends the projective cofibration $\iota^k(p): S^{k-1} \rightarrowtail D^k$ to a projective fibration. We are left to prove that there is a lift in the following commutative diagram, such that the lower triangle also commutes:
\[
\begin{tikzcd}[row sep=large, column sep=large]
W 
  \arrow[rr, "\Delta_W"] 
  \arrow[d, "i"'] 
&& {[D^k, W^{\otimes p}]} 
  \arrow[d, "f^{\otimes p}\circ (-)"] 
  \arrow[rd, "(-)~\circ ~\iota^k(p)", bend left=40] & \\
C 
  \arrow[d, "q"'] 
  \arrow[rr, "\delta_C"] 
&& {[D^k, V^{\otimes p}]} 
  \arrow[d, "(-)~\circ~i_n"] 
& {[S^{k-1}, W^{\otimes p}]} 
  \arrow[ld, "f^{\otimes p}\circ (-)", bend left=40] \\
V 
  \arrow[rr, "\Delta_V"'] 
  \arrow[urr, dashed, "\tilde{\Delta}_V"] 
&& {[S^{k-1}, V^{\otimes p}]}
\end{tikzcd}
\]

hence the restriction of the induced $\mathbb{T}(D^k(p))$-coalgebra structure is the starting $\mathbb{T}(S^{k-1}(p))$-structure on $V$. Since $V$ is cofibrant and $q$ an acyclic fibration, we have a retraction $r: V \longrightarrow C$ of $q$ (i.e. $q \circ r = \mathrm{id}_{V}$) and we can define 
\[
\tilde{\Delta}_V : V \longrightarrow \left[ D^k,V^{\otimes p} \right]
\]
as the composition 
\[ 
V\stackrel{r}\longrightarrow C \stackrel{\delta_C}\longrightarrow  \left[ D^k,V^{\otimes p}\right]~.
\]
Further, since $q$ is a weak-equivalence, so is $r$ and they are inverse of each other in the homotopy category. Which ensures that the remaining triangle is commutative in the homotopy category. Therefore we deduce that the functor $\iota^k(p)^*$ induces an isofibration at the level of homotopy categories.
\end{proof}

\subsubsection{Fullness and faithfulness properties of the restriction $\infty$-functor} We show that the functor $\iota^k(p)$ is fully faithful at the $\infty$-categorical level by explicitly identifying this functor. 

\begin{proposition}
Let $p \geq 0$ and $k \in \mathbb{Z}$. The functor $\iota^k(p)^*$ induces a fully faithful $\infty$-functor $i_1^*N_{\xi}\iota^k(p)^*$.
\end{proposition}

\begin{proof}
By Theorem \ref{thm: rectification of coalgebras over an endofunctor}, we can identify the $\infty$-category of dg $\mathbb{T}(S^{k-1}(p))$-coalgebras with coalgebras over an endofunctor\[
\mathbb{T}(S^{k-1}(p))\text{-}\mathsf{coalg}~[\mathsf{Q.iso}^{-1}] \simeq \mathbf{Coalg}_{\mathbb{T}(S^{k-1}(p))}(\mathbf{D}(\kk)) \simeq \mathbf{coalg}_{\widehat{\mathbf{S}}^c(S^{k-1}(p))}(\mathbf{D}(\kk))~,
\]
where $\widehat{\mathbf{S}}^c(S^{k-1}(p))(X) \simeq [S^{k-1}, X^{\otimes p}]$, where we consider the $\infty$-categorical self-enrichment of $\mathbf{D}(\kk)$. In other words, it is the data of an object $X$ in $\mathbf{D}(\kk)$ together with a degree $-k+1$ map 
\[
\Delta_X: X \longrightarrow X^{\otimes p}~. 
\]
Likewise, by Theorem \ref{thm: rectification of coalgebras over an endofunctor}, we can identify the $\infty$-category of dg $\mathbb{T}(D^{k}(p))$-coalgebras with coalgebras over an endofunctor
\[
\mathbb{T}(D^{k}(p))\text{-}\mathsf{coalg}~[\mathsf{Q.iso}^{-1}] \simeq \mathbf{Coalg}_{\mathbb{T}(D^{k}(p))}(\mathbf{D}(\kk)) \simeq \mathbf{coalg}_{\widehat{\mathbf{S}}^c(D^{k}(p))}(\mathbf{D}(\kk))~,
\]
where the last $\infty$-category is the category of coalgebras over the dual Schur endofunctor associated to $D^{k}(p)$. However, this endofunctor is contractible, therefore the forgetful functor induces an equivalence of $\infty$-categories 
\[
\mathbf{coalg}_{\widehat{\mathbf{S}}^c(D^{k}(p))}(\mathbf{D}(\kk)) \simeq \mathbf{D}(\kk)
\]
since, on the left, objects are chain complexes $V$ equipped with a map $V \longrightarrow 0$, and this map is unique (up to a contractible choice) because $0$ is the terminal object.

The equivalences obtained in Section \ref{Subsection: point-set models for coalgebras over endofunctors} are equivalences of pullbacks, and the pullbacks of the evaluation functors in Definitions \ref{definition: coalgebras over an endofunctor} and \ref{definition: infinity categorical coalgebras over an endofunctor} are the forgetful functors, so the equivalences of $\infty$-categories considered above commute with the forgetful $\infty$-functors. Therefore, the forgetful functor from the $\infty$-category of $\mathbb{T}(D^{k}(p))$-coalgebras to $\mathbf{D}(\kk)$ is also an equivalence of $\infty$-categories.

 Moreover, under these equivalences, the functor $i_1^*N_{\xi}\iota^k(p)^*$ can be identified with the functor that sends an object $X$ in $\mathbf{D}(\kk)$ to the $\widehat{\mathbf{S}}^c(S^{k-1}(p))$-coalgebra given by the zero map 
\[
0: X \longrightarrow X^{\otimes p}~, 
\]
in degree $-k+1$. Let us check that this functor is fully faithful. Let $X,Y$ be two $\widehat{\mathbf{S}}^c(S^{k-1}(p))$-coalgebras. Their mapping space $\mathrm{Map}_{\mathbf{Coalg}_{\widehat{\mathbf{S}}^c(S^{k-1}(p))}}(X,Y)$ is given by the equalizer
\[\operatorname{eq}\left( 
\begin{tikzcd}[column sep = 5pc]
\mathrm{Map}_{\mathbf{D}(\kk)}(X,Y) \arrow[r,shift left=0.5pc,"\Delta_Y \circ (-)"]\arrow[r,shift right=0.5pc,"\widehat{\mathbf{S}}^c(S^{k-1}(p))(-) \circ \Delta_X",swap]
&\mathrm{Map}_{\mathbf{D}(\kk)}(X,\widehat{\mathbf{S}}^c(S^{k-1}(p))(Y))
\end{tikzcd}\right)~.
\]
This description follows directly from Definition \ref{definition: infinity categorical coalgebras over an endofunctor}. Now, if both $X$ and $Y$ are in the image of $i_1^*N_{\xi}\iota^k(p)^*$, then the two maps in the above equalizer are zero and therefore the mapping space between $X$ and $Y$ in $\mathbb{T}(S^{k-1}(p))$-coalgebras is given by their mapping space in $\mathbf{D}(\kk)$.
\end{proof}

\subsection{Point-set models for coalgebras over a quasi-free operad}

To conclude, from Theorem \ref{thm: key theorem} and Corollary \ref{corollary: devissage along cells of 1-categories of P-coalgebras} we deduce the desired rectification result for coalgebras over a cell cofibrant dg operad:
\begin{proposition}\label{prop: rectification for cell cofibrant operads}
Let $\mathcal{P}$ be a cell cofibrant dg operad. There is an equivalence of $\infty$-categories
\[
\mathcal{P}\text{-}\mathsf{coalg}~[\mathsf{Q.iso}^{-1}] \simeq \mathbf{Coalg}_{\mathcal{P}}(\mathbf{D}(\kk))
\]
between dg $\mathcal{P}$-coalgebras up to quasi-isomorphism and coalgebras in $\mathbf{D}(\kk)$ over the induced enriched $\infty$-operad. 
\end{proposition}

\begin{proof}
Since $\mathcal{P}$ is cell cofibrant, it is given as a colimit of cofibrant dg operads along cofibrations
\[
\oneop_0 \hookrightarrow \oneop_1 \hookrightarrow \cdots \hookrightarrow \oneop_\alpha \hookrightarrow \cdots \hookrightarrow \colim_{\alpha} \oneop_\alpha \cong \oneop~, 
\]
where for all $\alpha$, $\oneop_{\alpha +1}$ is obtained from cell attachments onto $\oneop_\alpha$. This gives a limit tower of 1-categories
\[
\onecoalg{\oneop_0}  \twoheadleftarrow \onecoalg{\oneop_1} \twoheadleftarrow \cdots \twoheadleftarrow \onecoalg{\oneop_\alpha}  \twoheadleftarrow  \cdots \lim_{\alpha}\onecoalg{\oneop_\alpha} \cong \onecoalg{\oneop}
\]
by Corollary \ref{corollary: devissage along cells of 1-categories of P-coalgebras}. This limit becomes a homotopy limit of quasi-categories when applying $i_1^*N_{\xi}$: all the objects are fibrant relative categories so they become quasi-categories, and the transition maps are obtained by the pullbacks of Theorem \ref{thm: key theorem} so they all provide fibrations of quasi-categories (using that fibrations are stable under pullbacks).

On the other hand, the colimit 
\[
\oneop_0 \hookrightarrow \oneop_1 \hookrightarrow \cdots \hookrightarrow \oneop_\alpha \hookrightarrow \cdots \hookrightarrow \colim_{\alpha} \oneop_\alpha \cong \oneop~, 
\]
is a homotopy colimit of cofibrant dg operads, so the underlying enriched $\infty$-operad of $\mathcal{P}$ can be written as the colimit of the underlying enriched $\infty$-operads of $\oneop_\alpha$. Therefore, by Proposition \ref{prop: taking infinity categories of coalgebras preserves limits}, we have that 
\[
\lim_{\alpha} \mathbf{Coalg}_{\mathscr{P}_\alpha}(\mathbf{D}(\kk)) \simeq \mathbf{Coalg}_{\mathscr{P}}(\mathbf{D}(\kk))~.
\]
Let us now show by induction that for all $\alpha$, we have an equivalence 
\[
\mathcal{P}\text{-}\mathsf{coalg}~[\mathsf{Q.iso}^{-1}] \simeq \mathbf{Coalg}_{\mathscr{P}_\alpha}(\mathbf{D}(\kk))~.
\]
For $\alpha = 0$, the dg operad $\mathcal{P}_0$ is free, hence the result follows from Proposition \ref{prop: rectification of coalgebras over free operads}. Now, let us assume it is true for some $\alpha$. Then by Theorem \ref{thm: key theorem} together with Proposition \ref{prop: taking infinity categories of coalgebras preserves limits}, the result follows for $\alpha +1$ since both $\infty$-categories can be written as the equivalent (homotopy) pullbacks. Then, the general result follows from the fact that $\infty$-categories can be written as the (homotopy) limit of equivalent towers.
\end{proof}

\subsection{Point-set models for coalgebras over a cofibrant operad}

Finally, we conclude this section by extending the previous results to all cofibrant dg operads. 

\begin{theorem}\label{thm: rectification for cofibrant operads}
Let $\mathcal{P}$ be a cofibrant dg operad. There is an equivalence of $\infty$-categories
\[
\mathcal{P}\text{-}\mathsf{coalg}~[\mathsf{Q.iso}^{-1}] \simeq \mathbf{Coalg}_{\mathcal{P}}(\mathbf{D}(\kk))
\]
between dg $\mathcal{P}$-coalgebras up to quasi-isomorphism and coalgebras in $\mathbf{D}(\kk)$ over the induced enriched $\infty$-operad. 
\end{theorem}

\begin{proof}
The general result follows from two facts. First, any cofibrant dg operad $\mathcal{P}$ admits a quasi-isomorphism to a cell cofibrant dg operad. A particular choice of such is given by the canonical quasi-isomorphism $\Omega\mathrm{B}(\mathcal{P} \otimes \mathcal{E}) \qi \mathcal{P}$, where $\Omega\mathrm{B}$ is the operadic bar-cobar resolution and where $\mathcal{E}$ is the Barratt--Eccles dg operad of \cite{BergerFresse}. Indeed, since $\mathrm{B}(\mathcal{P} \otimes \mathcal{E})$ is quasi-planar in the sense for \cite{premierpapier}, its cobar construction is cell cofibrant, see \cite[Proposition 9 and 11]{premierpapier} for more details. See also \cite{BergerMoerdijk06}. The second fact is that a quasi-isomorphism between cofibrant dg operads induces a Quillen equivalence between their respective categories of coalgebras by \cite[Proposition 31]{premierpapier}, therefore it suffices to apply Proposition \ref{prop: rectification for cell cofibrant operads} to pass from the cell cofibrant to the cofibrant case. 
\end{proof}

\section{Applications: point-set models for non-finite type \texorpdfstring{$p$}{p}-adic homotopy types}

The goal of this section is to apply the rectification result of Theorem \ref{thm: rectification for cofibrant operads} to give explicit point-set models for the $p$-adic homotopy types of nilpotent spaces. This follows on the one hand from the explicit $\mathbb{E}_\infty$-coalgebra structure constructed by Berger and Fresse in \cite{BergerFresse} and, on the other hand, from the intrinsic $\infty$-categorical version of Mandell's theorem proved by Bachmann and Burklund in \cite{bachmann2024}, where they show that the $\infty$-functor of chains with $\bar{\mathbb{F}}_p$ coefficients with its $\mathbb{E}_\infty$-coalgebra structure fully faithfully encodes the $p$-adic homotopy types of nilpotent spaces, without any finite type assumption.

\subsection{Homotopy types as $\mathbb{E}_\infty$-coalgebras}
Let $\kk$ be a separably closed field of characteristic $p >0$. Let $\mathbf{Spc}$ denote the $\infty$-category of spaces. The functor of singular chains defines a functor 
\[
C_*(-;\kk): \mathbf{Spc} \longrightarrow \mathbf{Coalg}_{\mathbb{E}_\infty}(\mathbf{D}(\kk)) 
\]
from the $\infty$-category of space to the $\infty$-category of $\mathbb{E}_\infty$-coalgebras in chain complexes over $\kk$. We say that a space $X$ is \textit{nilpotent} if its fundamental group is a nilpotent group and if it acts nilpotently on all higher homotopy groups, for every possible choice of base point. We consider the Bousfield localization of space with respect to the homology theory $\mathrm{H}_*(-; \mathbb{F}_p)$, and we say that a space $X$ is $p$-complete if it is a local object with respect to this localization. Let us denote by $\mathbf{Spc}_p^{\mathbf{nil}}$ the $\infty$-category of $p$-complete nilpotent spaces. 

\begin{theorem}[{\cite[Theorem 1.2]{bachmann2024}}]\label{thm: bachmann-burklund}
Let $\kk$ be a separably closed field of characteristic $p >0$. The functor of singular chains
\[
C_*(-;\kk): \mathbf{Spc} \longrightarrow \mathbf{Coalg}_{\mathbb{E}_\infty}(\mathbf{D}(\kk)) 
\]
restricted to the $\infty$-category of $p$-complete nilpotent spaces is fully faithful. 
\end{theorem}

\begin{remark}
This theorem is dual to (and generalizes) Mandell's theorem in \cite{Mandell01}, where he shows that the functor of singular cochains is fully faithful on the $\infty$-category of \textit{finite type} $p$-complete nilpotent spaces, that is, $p$-complete nilpotent spaces such that every homology group finitely generated. The idea is that by using chains instead of cochains, one can get rid of the finite type assumption by avoiding an unnecessary dualization. 
\end{remark}

\subsection{Point-set models for $\mathbb{E}_\infty$-coalgebras}
We give a point-set version of Theorem \ref{thm: bachmann-burklund} using the fact that we can give a point-set presentation of the $\infty$-category of $\mathbb{E}_\infty$-coalgebras in chain complexes over $\kk$. 

\medskip

Let $\mathcal{E}$ be the Barratt--Eccles dg operad introduced in \cite{BergerFresse}. Notice that although its underlying dg symmetric sequence is cofibrant, this dg operad is \textit{not} cofibrant as an operad and thus we cannot apply Theorem \ref{thm: rectification for cofibrant operads} to its coalgebras. However, if we consider $\Omega\mathrm{B}\mathcal{E}$, that is, the operad obtained by applying the bar-cobar construction to $\mathcal{E}$, we get a cofibrant dg operad since $\mathrm{B}\mathcal{E}$ is a quasi-planar conilpotent dg cooperad; see \cite[Section 2]{premierpapier} for more details. 

\begin{proposition}
There is an equivalence of $\infty$-categories
\[
\Omega \mathrm{B} \mathcal{E}\text{-}\mathsf{coalg}~[\mathsf{Q.iso}^{-1}] \simeq \mathbf{Coalg}_{\mathbb{E}_\infty}(\mathbf{D}(\kk))~, 
\]
between the $\infty$-category of dg $\Omega\mathrm{B}\mathcal{E}$-coalgebras, localized with respect to quasi-isomorphisms, and the $\infty$-category of $\mathbb{E}_\infty$-coalgebras in chain complexes over $\kk$. 
\end{proposition}

\begin{proof}
Since is a cofibrant dg operad, the $\infty$-category obtained by localizing dg $\Omega \mathrm{B} \mathcal{E}$-coalgebras with respect to quasi-isomorphisms is equivalent to the $\infty$-category of coalgebras over its induced $\infty$-operad, in the sense of Definition \ref{definition: P-coalgebra}. This $\infty$-category is in turn equivalent to the $\infty$-category of $\mathbb{E}_\infty$-coalgebras in the sense of Lurie by Theorem \ref{thm: comparison with coalgebras in the sense of Lurie}, since the induced $\infty$-operad by $\Omega \mathrm{B} \mathcal{E}$ is a model for the (enriched) $\mathbb{E}_\infty$-operad, and thus is an (enriched) $\infty$-operad that ultimately comes from spaces via the construction explained in Theorem \ref{thm: comparison with coalgebras in the sense of Lurie}. 
\end{proof}

Recall that, for any simplicial set $X$, the functor of cellular chains $C_*(X; \kk)$ applied to $X$ admits an explicit dg $\mathcal{E}$-coalgebra structure constructed by Berger and Fresse in \cite{BergerFresse}. Pulling back this functor along the restriction along the canonical map $\Omega \mathrm{B} \mathcal{E} \qi \mathcal{E}$ induced a functor 
\[
C_*(-;\kk): \mathsf{sSet} \longrightarrow \Omega \mathrm{B} \mathcal{E}\text{-}\mathsf{coalg}~, 
\]
which is given by the cellular chains functor endowed with its $\mathcal{E}$-coalgebra structure. 

\begin{theorem}\label{thm: point-set model for the chains functor}
Let $\kk$ be a field. The functor 
\[
C_*(-;\kk): \mathsf{sSet} \longrightarrow \Omega \mathrm{B} \mathcal{E}\text{-}\mathsf{coalg}~, 
\]
is a model for the $\infty$-categorical chains functor. 
\end{theorem}
 
\begin{proof}
When we localize on the left by weak homotopy equivalences and on the right by quasi-isomorphisms, we obtain a functor 
\[
C_*(-;\kk): \mathbf{Spc} \longrightarrow \mathbf{Coalg}_{\mathbb{E}_\infty}(\mathbf{D}(\kk)) 
\]
which, on the underlying $\infty$-category of chain complexes, agrees with the singular chains functor mentioned before. Let us check that the two functors and their $\mathbb{E}_\infty$-coalgebra structures also agree. This follows from the fact that the image of the point $\{*\}$ by the two functors is the unique coalgebra structure on $\kk$, and from the fact that both functors preserve (homotopy) colimits. Since any space is a homotopy colimit of points, the result follows directly. 
\end{proof}

In fact, the functor $C_*(-;\kk)$ admits a right adjoint $\overline{\mathcal{R}}$ and there is a Quillen adjunction 
\[
\begin{tikzcd}[column sep=5pc,row sep=3pc]
         \mathsf{sSet} \arrow[r, shift left=1.1ex, "C_*(-;\kk)"{name=F}] &\Omega \mathrm{B} \mathcal{E}\text{-}\mathsf{coalg}, \arrow[l, shift left=.75ex, "\overline{\mathcal{R}}"{name=U}]
            \arrow[phantom, from=F, to=U, , "\dashv" rotate=-90]
\end{tikzcd}
\]
between the category of simplicial sets and the category of dg $\Omega\mathrm{B}\mathcal{E}$-coalgebras when we endow it with the transferred model structure from chain complexes over $\kk$. 

\begin{corollary}\label{cor: derived unit is an Fp equivalence}
Let $\kk$ be a separably closed field of characteristic $p >0$. Let $X$ be a nilpotent simplicial set. The derived unit of the adjunction 
\[
\eta_X: X \longrightarrow \mathbb{R}\overline{\mathcal{R}}C_*(X;\kk)
\]
is an equivalence in homology with coefficients in $\mathbb{F}_p$. 
\end{corollary}

\begin{proof}
Immediate from Theorems \ref{thm: bachmann-burklund} and \ref{thm: point-set model for the chains functor}. 
\end{proof}

\subsection{Lie-type models}
In \cite{lucio2024higherlietheorypositive}, the first author used the Koszul duality between dg $\Omega\mathrm{B} \mathcal{E}$-coalgebras and absolute partition $\mathcal{L}_\infty$-algebras to obtain models in terms of this later algebraic structure. Roughly speaking, absolute partition $\mathcal{L}_\infty$-algebras can be understood as a particular choice of point-set models for the partition Lie algebras of Brantner and Matthew in \cite{brantnermathew}, but where infinite sums of structural operations are well defined by definition. For an intuition on the notion of absolute algebras, see \cite{absolutealgebras}. The results of \cite{lucio2024higherlietheorypositive}, and in particular Theorem D in \textit{op.cit.}, where obtained by applying linear duality to the results of \cite{Mandell01}. Using the point-set version of the results in \cite{bachmann2024} we have just given, we end this section by removing the finite type assumptions in the results of \cite{lucio2024higherlietheorypositive}. For the rest of this subsection, let $\kk$ be a separably closed field of characteristic $p >0$. 

\medskip

Recall that in \cite{lucio2024higherlietheorypositive}, the first author constructed a Quillen adjunction 
\[
\begin{tikzcd}[column sep=5pc,row sep=3pc]
          \mathsf{sSet}_* \arrow[r, shift left=1.1ex, "\mathcal{L}_*"{name=A}]
          &\mathsf{abs}~\mathcal{L}_\infty^\pi\text{-}\mathsf{alg}^{\mathsf{qp}\text{-}\mathsf{comp}} \arrow[l, shift left=.75ex, "\mathcal{R}_*"{name=B}]
           \arrow[phantom, from=A, to=B, , "\dashv" rotate=-90]
\end{tikzcd}
\]
between pointed simplicial sets and absolute partition $\mathcal{L}_\infty$-algebras which satisfy a separateness axiom called \textit{qp-completeness} in the terminology of \cite{premierpapier}. 

\begin{theorem}
Let $X$ be a pointed nilpotent simplicial set. The unit of adjunction
\[
\eta_X: X \qi \mathcal{R}_*\mathcal{L}_*(X) 
\]
is an equivalence in homology with coefficients in $\mathbb{F}_p$. 
\end{theorem}

\begin{proof}
Follows directly from the fact that the adjunction $\mathcal{L}_* \dashv \mathcal{R}_*$ constructed in \cite[Section 2]{lucio2024higherlietheorypositive} is a model for the derived adjunction of the (pointed) adjunction $\tilde{C}_*(-;\kk) \dashv \overline{\mathcal{R}}_*$, together with Corollary \ref{cor: derived unit is an Fp equivalence}. 
\end{proof}

Using absolute partition $\mathcal{L}_\infty$-algebras, one can also obtain algebraic models for $p$-adic mapping spaces. Note that here we are using an unpointed version of the adjunction $\mathcal{L}_* \dashv \mathcal{R}_*$ involving \textit{curved} absolute partition $\mathcal{L}_\infty$-algebras to get unpointed mapping spaces. See \cite{lucio2024higherlietheorypositive} for more details. 

\begin{theorem}
Let $X$ be a simplicial set and let $Y$ be a nilpotent simplicial set. There is a weak equivalence of simplicial sets 
\[
\mathrm{Map}(X, Y_{\mathbb{F}_p}) \simeq \mathcal{R}\left(\mathrm{hom}(\mathrm{H}_*(X), \mathcal{L}(Y))\right)~,
\]
where $Y_{\mathbb{F}_p}$ denotes the Bousfield-Kan $p$-completion of $Y$. 
\end{theorem}

\begin{proof}
Follows from replacing \cite[Corollary 3.7]{lucio2024higherlietheorypositive} with its stronger version, and using the same arguments as in \cite[Corollary 3.28]{lucio2024higherlietheorypositive}. 
\end{proof}

\begin{remark}[A combinatorial presentation of the homotopy groups of a space]
In \cite[Theorem 3.17]{lucio2024higherlietheorypositive}, the author gave a combinatorial description of the homotopy groups of the $p$-completion of a finite type pointed connected nilpotent simplicial set purely in terms of certain equations in its absolute partition $\mathcal{L}_\infty$-model $\mathcal{L}_*(X)$. This combinatorial description is also valid for any absolute partition $\mathcal{L}_\infty$-algebra, see \cite[Theorem B]{lucio2024higherlietheorypositive}. Therefore we can also remove the finite type hypothesis in the aforementioned theorem. 
\end{remark}

\renewcommand{\thesection}{A}

\section{Appendix. Relative categories and complete Segal spaces}\label{Appendix A: Relative categories}

The goal of this appendix is to review different models for the $\infty$-category of $\infty$-categories, and to recall the links between relative categories and complete Segal spaces. We also construct cylinder objects and path objects in these later two model categories. 

\subsection{Relative categories, complete Segal spaces and the subdivided nerve} We recall the definition a relative category and the existence of a model structure on all relative categories, called the Barwick--Kan model structure, which presents the the $\infty$-category of $\infty$-categories. We refer to \cite{BarwickKan1} for more details. 

\begin{definition}[Relative category]
A \textit{relative category} $(\mathsf{C},\mathsf{W})$ is the data of a category $\mathsf{C}$ equip\-ped with a subcategory $\mathsf{W}$ which contains all the objects of $\mathsf{C}$ and whose arrows are called \textit{weak equivalences}. 
\end{definition}

A morphism between two relative categories $(\mathsf{C},\mathsf{W})$ and $(\mathsf{C}',\mathsf{W}')$ is a functor $F: \mathsf{C} \longrightarrow \mathsf{C}'$ such that $F(\mathsf{W}) \subseteq \mathsf{W}'$. We denote the 1-category of all relative categories by $\mathsf{RelCat}$. 

\medskip

Let $\mathsf{bisSets}$ denote the 1-category of bisimplicial sets. It admits a model structure, called the Rezk model structure \cite{Rezk01}, where every object is cofibrant and the fibrant objects are precisely \textit{complete Segal spaces}. When localized at weak equivalences, it presents the $(\infty,1)$-category of $\infty$-categories. 
\medskip

The idea of Barwick--Kan is to construct an adjunction 
\[
\begin{tikzcd}[column sep=5pc,row sep=3pc]
         \mathsf{bisSet} \arrow[r, shift left=1.1ex, "K_{\xi}"{name=F}] & \mathsf{RelCat} , \arrow[l, shift left=.75ex, "N_{\xi}"{name=U}]
            \arrow[phantom, from=F, to=U, , "\dashv" rotate=-90]
\end{tikzcd}
\]
between bisimplicial sets and relative categories, and to transfer the Rezk model structure along this adjunction, in order to obtain a Quillen equivalent model structure on relative categories. The functor $N_{\xi}$ is called the\textit{ subdivided nerve}, and it is constructed by specifying the following bisimplicial object in relative categories: $\Delta[p,q] \coloneqq \mathsf{p}\times\mathsf{q}^{\simeq}$ for all $p,q \geq 0$. Here $\mathsf{p}$ is the totally ordered set $\{1,\dots, p\}$, considered in the usual way as a category with exactly one arrow $n \longrightarrow m$ if $n \leq m$, and no weak equivalences. Similarly, $\mathsf{q}^{\simeq}$ is the relative category given by the totally ordered set $\{1,\dots, q\}$, except every arrow is a weak equivalence. 

\begin{remark}
    Our notations differ from those in \cite{BarwickKan1}. The relative categories that we write as $\mathsf{p}$ and $\mathsf{p}^\simeq$ are written as $\check{\mathsf{p}}$ and $\hat{\mathsf{p}}$, respectively, in \textit{op.cit.}
\end{remark}

\begin{theorem}[\cite{BarwickKan1}]
The Quillen adjunction 
\[
\begin{tikzcd}[column sep=5pc,row sep=3pc]
         \mathsf{bisSet} \arrow[r, shift left=1.1ex, "K_{\xi}"{name=F}] & \mathsf{RelCat} , \arrow[l, shift left=.75ex, "N_{\xi}"{name=U}]
            \arrow[phantom, from=F, to=U, , "\dashv" rotate=-90]
\end{tikzcd}
\]
is a Quillen equivalence. 
\end{theorem}

In general, cofibrant objects and fibrant objects are quite hard to describe in the model structure of Barwick--Kan. However, relative categories arising from model categories are fibrant.

\begin{theorem}[{\cite[Main Theorem]{Meier16}}]
Let $\mathsf{C}$ be a model category with weak equivalences given by $\mathsf{W}$. Then its underlying relative category $(\mathsf{C},\mathsf{W})$ is a fibrant object in $\mathsf{RelCat}$. 
\end{theorem}

Finally, let us mention that since it is also possible to compare bisimplicial spaces together with Rezk's model structure with simplicial sets with Joyal's model structure using \cite{JoaylTierney07}, we can consider the following composition
\[
\begin{tikzcd}[column sep=5pc,row sep=3pc]
		\mathsf{sSet} 	\arrow[r, shift left=1.1ex, "p_1^*"{name=A}]
         &\mathsf{bisSet} \arrow[r, shift left=1.1ex, "K_{\xi}"{name=F}] \arrow[l, shift left=.75ex, "i_1^*"{name=B}] & \mathsf{RelCat} , \arrow[l, shift left=.75ex, "N_{\xi}"{name=U}]
            \arrow[phantom, from=F, to=U, , "\dashv" rotate=-90] \arrow[phantom, from=A, to=B, , "\dashv" rotate=-90]
\end{tikzcd}
\]
of Quillen equivalences to obtain a quasi-category from a fibrant relative category. Notice, however, that if $\mathsf{C}$ is a model category, the quasi-category $i_1^* N_\xi\mathsf{C}$ is quite different from the quasi-category $\mathcal{N}^{coh}(\mathsf{C}^{cf})$ obtained by applying the coherent nerve to the simplicially enriched subcategory of fibrant-cofibrant objects in $\mathsf{C}$ --- even though both are models for the same underlying $\infty$-category. This distinction will become relevant in Subsection \ref{S: the gluing proof}. 

\subsection{Internal homs in relative categories and bisimplicial sets}\label{subsection: internal homs}
Both the category of bisimplicial sets $\mathsf{bisSets}$ and the category of relative categories $\mathsf{RelCat}$ are cartesian closed symmetric monoidal model categories. Let us  quickly review the construction of their internal homs.

\begin{definition}
The internal hom for the cartesian structure on bisimplicial sets is defined as follows. Given two bisimplicial sets $X$ and $Y$, we denote by $Y^X$ the bisimplicial set whose bisimplices are given by
\[
(Y^X)_{p,q} = \mathrm{Hom}_{\mathsf{bisSets}}(X\times\Delta[p,q],Y)
\]
and where the bisimplicial structure is induced by that of $\Delta[\bullet,\bullet]$.
    \end{definition}

\begin{definition}
    
The internal hom for the cartesian structure on relative categories is defined as follows. Given two relative categories $\mathsf{C}$ and $\mathsf{D}$, we denote by $\mathsf{D}^\mathsf{C}$ the relative category
\begin{itemize}
\item[-] whose objects are relative functors $\mathsf{C} \rightarrow \mathsf{D}$;
\item[-] whose morphisms are relative functors $\mathsf{C}\times\mathsf{1}\rightarrow \mathsf{D}$;
\item[-] whose weak equivalences are the relative functors $\mathsf{C}\times\mathsf{1}^{\simeq}\rightarrow \mathsf{D}$.
\end{itemize}
\end{definition}

\subsection{The lax simplicial model structure on relative categories}
A classical result by Rezk ensures the presence of a simplicial enrichment compatible with the model structure for complete Segal spaces. See \cite{Rezk01} for more details. 

\begin{proposition}
The model category of bisimplicial sets $\mathsf{bisSets}$ forms a simplicial model category in which, given two bisimplicial sets $X$ and $Y$, the simplicial mapping space is defined by the $0$th row of the internal hom. That is,
\[
\mathrm{Map}_{\mathsf{bisSets}}(X,Y)_n = (X^Y)_{n,0} = \mathrm{Hom}_{\mathsf{bisSets}}(X\times\Delta[0,n],Y).
\]
\end{proposition}

The simplicial model structure on relative categories, however, does not seem to appear in the literature, so we provide it here. For this, let us recall first that the cartesian product of  $\mathsf{RelCat}$ is a left Quillen bifunctor adjoint to an internal hom bifunctor that we note $\mathsf{C}^{\mathsf{D}}$, defined in the previous section. 

\medskip

Now, given any closed symmetric monoidal model categories $\mathsf{C}$ and $\mathsf{M}$, a symmetric monoidal left Quillen functor $L:\mathsf{M}\rightarrow\mathsf{C}$ equips $\mathsf{C}$ with the structure of a symmetric monoidal model category tensored and enriched over $\mathsf{M}$, where the enrichment is defined uniquely as a right adjoint to the external tensor product

\[
\begin{tikzcd}[column sep=3pc,row sep=0pc]
\mathsf{C}\times\mathsf{M} \arrow[r]
&\mathsf{C} \\
(X,A) \arrow[r,mapsto]
&X \otimes L(A)~. 
\end{tikzcd}
\]

Indeed, the pushout-product axiom for this external tensor product follows then from the pushout-product axiom of the internal tensor product of $\mathsf{C}$.

\medskip

We are going to apply this to the case where $\mathsf{C}= \mathsf{RelCat}$ and $\mathsf{M}=\mathsf{sSet}$, so let us provide such a left Quillen functor. First, let us recall from \cite{JoaylTierney07} the construction of the box bifunctor 
\[
\square: \mathsf{sSet} \times \mathsf{sSet} \longrightarrow \mathsf{bisSet}
\]
defined as $X \square Y_{m,n} \coloneqq X_m\times Y_n$. By \cite[Proposition 4.6]{JoaylTierney07}, it defines a left Quillen bifunctor
\[
\square :  \mathsf{sSet} \times  \mathsf{sSet} \rightarrow \mathsf{bisSet}
\]
where the category of simplicial sets $\mathsf{sSet}$ is endowed with the Joyal model structure and the category of bisimplicial sets $\mathsf{bisSet}$ is endowed with the Reedy model structure. Let us 
point out that cofibrations in the Joyal model structure (monomorphisms) agree with cofibrations in the Kan--Quillen model structure, and that class of weak equivalences of simplicial sets in the Kan--Quillen model structure is included in the class of weak equivalences in the Joyal model structure. This implies that the box functor also defines a left Quillen bifunctor when $\mathsf{sSet}$ is endowed with the Kan--Quillen model structure. 

\medskip

The second factor projection $p_2:\Delta\times\Delta \longrightarrow\Delta$ induces a functor $p_2^*:\mathsf{sSet} \rightarrow \mathsf{bisSet}$ which also given by $\Delta^0 \square (-)$, so it forms a left Quillen functor. The left Quillen functor functor $L$ is will be given as the composite
\[
L \coloneqq K_{\xi} \circ p_2^*: \mathsf{sSet}\rightarrow \mathsf{RelCat}.
\]

Notice, however, that the functor $L$ defined above is not strong monoidal. 

\begin{lemma}\label{lemma:laxmonL}
The functor $L$ is a symmetric lax monoidal functor.
\end{lemma}
\begin{proof}
The functor $p_2^*$ is strong monoidal by construction: for any simplicial sets $K$ and $L$, we have
\[
p_2^*(K\times L) = \Delta^0\square (K\times L)\cong (\Delta^0\square K) \times (\Delta^0\square L).
\]
Therefore, we have then to understand the behaviour of $K_{\xi}$ with respect to the cartesian product. The functor $L$ preserves colimits because so does $\Delta^0\square (-)$ and $K_{\xi}$. Moreover, every simplicial set is a colimit of standard simplices, and the cartesian product preserves colimits in each variable. Thus, we just need to check what happens with standard simplices, that is, to compare $L(\Delta^p\times\Delta^q)$ with $L(\Delta^p)\times L(\Delta^q)$ for every integers $p$ and $q$.

\medskip

First, we have $p_2^*(\Delta^p\times\Delta^q) = \Delta[0,p]\times\Delta[0,q]$. By construction, the functor $K_{\xi}$ sends any $\Delta[0,p]$ and any cartesian product of these to a maximal relative category, where every arrow is a weak equivalence. Moreover, the maximal functor 
\[
(-)^{\simeq}: \mathsf{Cat} \longrightarrow \mathsf{RelCat}
\]
which send a category $\mathsf{C}$ to the relative category  $(\mathsf{C},\mathsf{C})$ preserves cartesian products (as it is actually right adjoint to the forgetful functor), so we only need to compare the underlying categories of the relative categories $K_{\xi}(\Delta[0,p]\times\Delta[0,q])$ and $K_{\xi}(\Delta[0,p])\times K_{\xi}(\Delta[0,q])$. Recall the commutative square
\[
\begin{tikzcd}[column sep=3pc,row sep=3pc]
\mathsf{sSet} \arrow[r, "K_{\xi}"] \arrow[d, "\mathrm{diag}"'] 
  &\mathsf{RelCat} \arrow[d, "U"] \\
\mathsf{sSet} \arrow[r, "c ~ \circ ~\mathrm{Sd}^2"'] 
  & \mathsf{Cat}
\end{tikzcd}
\]
where $U$ is the forgetful functor, $\mathrm{diag}$ the diagonal of bisimplicial sets, $\mathrm{Sd}^2$ the iterated barycentric subdivision and $c$ the fundamental category. 

\medskip

From this square we can deduce that the underlying categories of
$K_{\xi}(\Delta[0,p]\times\Delta[0,q])$ and $K_{\xi}(\Delta[0,p])\times K_{\xi}(\Delta[0,q])$ are respectively given by $c\mathrm{Sd}^2\mathrm{diag}(\Delta[0,p] \times\Delta[0,q])$ and by the cartesian product $c\mathrm{Sd}^2\mathrm{diag}(\Delta[0,p]) \times c\mathrm{Sd}^2\mathrm{diag}(\Delta[0,q])$.

\medskip

The diagonal functor $\mathrm{diag}$ is clearly strong monoidal. The functor $c\circ \mathrm{Sd}^2$ admits $\mathrm{Ex}^2\circ\mathcal{N}$ as right adjoint, where $\mathcal{N}$ is the nerve and $\mathrm{Ex}^2$ the two-fold iteration of Kan's Ex functor. As a right adjoint, $\mathrm{Ex}^2\circ\mathcal{N}$ preserves cartesian products, so it is strong monoidal since the monoidal structure is the cartesian one. Any left adjoint to a strong monoidal functor is lax monoidal, hence the lax monoidality of $c\circ \mathrm{Sd}^2\circ \mathrm{diag}$.
\end{proof}

\begin{remark}
The lax monoidality can also be seen from the construction of the subdivision functor: on a standard simplex, it is defined as the nerve of the category of non-degenerate simplices. The definition is then extended to any simplicial set $X$ by setting  
\[
\mathrm{Sd}(X) \coloneqq \colim_{\Delta^n \rightarrow X}\mathrm{Sd}(\Delta^n)~. 
\]
\end{remark}

Lax monoidality is sufficient to prove the following result. 

\begin{theorem}\label{thm: lax simplicial tensoring and cotensoring}
The model category $\mathsf{RelCat}$ forms a lax simplicial model category endowed with:

\begin{itemize}
\item[(1)] the tensoring defined by $\mathsf{C} \times L(X)$ for any relative category $\mathsf{C}$ and simplicial set $X$, which preserves colimits in each variable and defines a left Quillen bifunctor;


\item[(2)] cotensoring defined by $\mathsf{C}^{L(X)}$ for any relative category $\mathsf{C}$ and simplicial set $X$;


\item[(3)] simplicial hom spaces defined by
\[
\mathrm{Map}_{\mathsf{RelCat}}(\mathsf{C},\mathsf{D}) \coloneqq \mathrm{Hom}_{\mathsf{RelCat}}(\mathsf{C}\times L(\Delta^{\bullet}),\mathsf{D})\cong \mathrm{Hom}_{\mathsf{RelCat}}(\mathsf{C},\mathsf{D}^{L(\Delta^{\bullet})})
\]
for any relative categories $\mathsf{C}$ and $\mathsf{D}$, where the simplicial structure is induced by the cosimplicial structure of $\Delta^{\bullet}$.
\end{itemize}
\end{theorem}
\begin{proof}
The conditions needed for a lax simplicial model category are all satisfied because of the following properties of the cartesian product and the internal hom: 
\begin{itemize}
\item The pushout-product axiom for the tensoring is an application of the push\-out-product axiom for the cartesian product, since $\times$ is a left Quillen bifunctor and $L$ a left Quillen functor.

\item The dual pushout-product, or pullback-corner axiom for the cotensoring boils down to an application of the pullback-corner axiom for the internal hom, which holds true in any closed symmetric monoidal model category.

\item The adjunction relation between tensoring and cotensoring is a consequence of the adjunction relation between the cartesian product and the internal hom.

\item The existence of the tensoring over $\mathsf{sSet}$ for any model category $\mathsf{C}$ implies automatically the existence of a simplicial hom defined between two objects $A$ and $B$ by setting 
\[
\mathrm{Map}_{\mathsf{C}}(A,B)= \mathrm{Hom}_{\mathsf{C}}(A \otimes \Delta^{\bullet},B),
\]
which is the unique definition forced by the adjunction requirement. This is a consequence of the fact that, for any simplicial set $K$, we have $\mathrm{Hom}_{\mathsf{sSet}}(\Delta^n,K)= K_n$ by the Yoneda lemma. \qedhere
\end{itemize}
\end{proof}
\begin{remark}

Note that $L(\Delta^n)=\xi\mathsf{n}^{\simeq}$, so the theorem above provides functorial simplicial and cosimplicial resolutions in the model category of relative categories $\mathsf{RelCat}$. They are given, for any relative category $\mathsf{C}$, by $\mathsf{C} \times \xi (\bullet)^{\simeq}$ and $\mathsf{C}^{\xi(\bullet)^{\simeq}}$.
\end{remark}

\subsection{Cylinders, path objects and arrow categories in relative categories and complete Segal spaces}
The goal of this final subsection is to construct explicit cylinder and path objects in relative categories and in bisimplicial sets. This will imply, in particular, that certain maps which are relevant in Subsection \ref{Subsection: point-set models for coalgebras over endofunctors} are indeed fibrations. 

\subsubsection{The usual path and cylinder objects} Let us first recall the usual definition of a cylinder or a path object in a general model category. 

\begin{definition}
Let $\mathsf{C}$ be a model category.
\begin{itemize}
\item[(1)] A good cylinder object of $X \in \mathrm{ob}(\mathsf{C})$ is a factorization of the codiagonal map $\mathrm{id} \amalg \mathrm{id}: X \amalg X \longleftarrow X$ as 
\[
X \amalg X\rightarrowtail \mathrm{Cyl}(X)\stackrel{\sim}{\rightarrow} X
\]
where the first map is a cofibration and the second map a weak equivalence.


\item[(2)] A good path object $Y\in \mathrm{ob}(\mathsf{C})$ is a factorization of the diagonal map $(\mathrm{id},\mathrm{id}): Y \longrightarrow Y \times Y$ as 
\[
Y\stackrel{\sim}{\rightarrow}\mathrm{Path}(Y)\twoheadrightarrow Y\times Y~,
\]
where the first map is a weak equivalence and the second one a fibration.
\end{itemize}
\end{definition}

\begin{proposition}\label{prop: goodpath}
Let $\mathsf{C}$ be a lax simplicial model category. The tensoring $(-)\otimes\Delta^1$ and the cotensoring $(-)^{\Delta^1}$ define, respectively, functorial good cylinder objects on cofibrant objects and functorial good path objects on fibrant objects.

\end{proposition}
\begin{proof}
\textit{Tensoring by $\Delta^1$.}
Let $A$ be a cofibrant object in $\mathsf{C}$. The cofaces $d^0,d^1:\Delta^0\rightarrow\Delta^1$ induce two morphisms $\mathrm{id}\otimes d^0,\mathrm{id}\otimes d^1:A \cong A \otimes \Delta^0\rightarrow A\otimes\Delta^1$. Since actually $d^0$ and $d^1$ are the inclusions of horns $\Lambda_0^1\rightarrow\Delta^1$ and $\Lambda_1^1\rightarrow\Delta^1$, they are acyclic cofibrations, so $\mathrm{id}\otimes d^0,\mathrm{id}\otimes d^1$ are acyclic cofibrations as well by the pushout-product axiom satisfied by the tensoring, using that $A$ is cofibrant. Thus, the morphism
\[
\begin{tikzcd}
A\amalg A \cong  A\otimes\Delta^1 \amalg A\otimes\Delta^1 \arrow[rr,"(\mathrm{id} \otimes d^0 \amalg \mathrm{id} \otimes d^1)"]
&
&A\otimes\Delta^1
\end{tikzcd}
\]
equals the morphism $A\otimes\partial\Delta^1\rightarrow A\otimes\Delta^1$. And it is induced by the cofibration $\partial\Delta^1\rightarrow\Delta^1$, so it forms a cofibration as well still by the external pushout-product property.

\medskip

Finally, the codegeneracy $s^0:\Delta^1\rightarrow\Delta^0$ induces $\mathrm{id} \otimes s^0:A\otimes\Delta^1\rightarrow A\otimes\Delta^0 \cong A$ which satisfies that $(\mathrm{id} \otimes s^0)\circ (\mathrm{id} \otimes d^0) = (\mathrm{id}\otimes s^0d^0)= \mathrm{id} = (\mathrm{id} \otimes s^0)\circ (\mathrm{id} \otimes d^1)$, so by the two-out-of-three property of weak equivalences the morphism $\mathrm{id} \otimes s^0$ is a weak equivalence. In conclusion, we do get a factorization of $\mathrm{id} \amalg \mathrm{id}: A\amalg A\rightarrow A$ as
\[
A \amalg A \rightarrowtail A\otimes\Delta^1\stackrel{\sim}{\rightarrow} A
\]
by considering $(\mathrm{id} \otimes d^0 \amalg \mathrm{id} \otimes d^1)$ followed by $\mathrm{id} \otimes s^0$. 

\medskip

\textit{Cotensoring by $\Delta^1$.} The proof is very similar to the previous one. Let $A$ be a fibrant object in $\mathsf{C}$. The acyclic cofibrations $d^0$ and $d^1$ induce two maps $(d^0)^*,(d^1)^*:A^{\Delta^1}\rightarrow A$, which are acyclic fibrations by the pullback-corner axiom for the cotensoring since $A$ is fibrant. Then, the morphism $((d^0)^*,(d^1)^*): A^{\Delta^1}\rightarrow A\times A$ is the cotensoring of $\mathrm{id}$ with the cofibration $\partial\Delta^1\rightarrow \Delta^1$, so it forms a fibration still  by the pullback-corner axiom.

\medskip

Finally $s^0$ induces a morphism $(s^0)^*:A\rightarrow A^{\Delta^1}$ satisfying $(d^0)^*\circ (s^0)^* = (d^1)^*\circ (s^0)^* = \mathrm{id}$, so $(s^0)^*$ is a weak equivalence and we get a factorization
\[
A\stackrel{\sim}{\rightarrow} A^{\Delta^1}\twoheadrightarrow A\times A
\]
as $(s^0)^*$ followed by $((d^0)^*,(d^1)^*)$. 
\end{proof}

\begin{corollary}\label{cor:evaluationfibration}\leavevmode

\begin{itemize}
    \item[(1)] The endofunctor $(-)^{\Delta[0,1]}$ of bisimplicial sets defines a functorial good path object for any fibrant object in the Rezk model structure. 


    \item[(2)]  The endofunctor $(-)^{\xi\mathsf{1}^{\simeq}}$ of relative categories defines a functorial good path object for any fibrant object in the Barwick--Kan model structure.
\end{itemize}


In particular, we get that:

 
\begin{itemize}
\item[(1)] For any fibrant relative category $\mathsf{C}$, the map  $(\mathrm{ev}_0,\mathrm{ev}_1):\mathsf{C}^{\xi\mathsf{1}^{\simeq}}\rightarrow \mathsf{C}\times \mathsf{C}$ is a fibration in $\mathsf{RelCat}$;


\item[(2)] For any complete Segal space $X$, the map  $(\mathrm{ev}_0,\mathrm{ev}_1):X^{\Delta[0,1]}\rightarrow X\times X$ is a fibration in $\mathsf{bisSet}.$
\end{itemize}
\end{corollary}

\begin{remark}\label{rem: quasicatpath}
In the Joyal model category on simplicial sets, there is an explicit path object whose construction is similar to $(-)^{\xi\mathsf{1}^{\simeq}}$ in $\mathsf{RelCat}$. Given a quasi-category $\mathsf{C}$, the path quasi-category $\mathrm{Path}(\mathsf{C})$ is the full sub-quasi-category $\mathrm{Fun}^{iso}(\Delta^1,\mathsf{C})\subset \mathrm{Fun}(\Delta^1,\mathsf{C})$ spanned by the objects corresponding to functors $\Delta^1\rightarrow \mathsf{C}$ which represent an equivalence in $\mathsf{C}$. The restriction maps $(r_0,r_1)$ along $\{0\}\hookrightarrow \Delta^1\hookleftarrow \{1\}$ fit in a factorization of the diagonal
\[
\mathsf{C}\stackrel{\sim}{\hookrightarrow} \textsf{Path}(\mathsf{C})\twoheadrightarrow\mathsf{C}\times\mathsf{C}
\]
where the second map is an isofibration of $\infty$-categories (hence a fibration in the Joyal model structure).
\end{remark}

\subsubsection{Models for the $\infty$-category of morphisms}
Notice that in Corrollary \ref{cor:evaluationfibration}, both path objects are constructed by constructing a model of the $\infty$-category of \textit{equivalences} in a general $\infty$-category $\mathsf{C}$. In this subsection, we construct similar models for the $\infty$-category of arrows in a general $\infty$-category $\mathsf{C}$, which \textit{need not} be equivalences. The end goal is to show a version of Corollary \ref{cor:evaluationfibration} for these models. 

\begin{remark}
Notice that the difference between $\mathrm{Fun}^{iso}(\Delta^1,\mathsf{C})$  and $\mathrm{Fun}(\Delta^1,\mathsf{C})$ in Remark \ref{rem: quasicatpath} (where $\mathsf{C}$ is a quasi-category) parallels the difference between functors $\mathsf{1}^{\simeq}\rightarrow\mathsf{C}$ and $\mathsf{1}\rightarrow\mathsf{C}$ (where $\mathsf{C}$ is a relative category), and also parallels the difference between $\mathrm{Hom}_{\mathsf{bisSet}}(\Delta[0,1],\mathsf{C})$ and $\mathrm{Hom}_{\mathsf{bisSet}}(\Delta[1,0],\mathsf{C})$ (where $\mathsf{C}$ is a complete Segal space). In all of these situations, the first object models equivalences in $\mathsf{C}$ and the second one general arrows in $\mathsf{C}$. 
\end{remark}

Precisely, we would like the evaluation maps of the models of these $\infty$-categories of morphisms to be fibrations in the respective model structures of bisimplicial sets and relative categories. It turns out that the arguments provided above to get Corollary \ref{cor:evaluationfibration} work as well for $(-)^{\Delta[1,0]}$ and $(-)^{\mathsf{1}}$ via a slight modification of the left Quillen functor $L$ which induces the tensoring and the cotensoring in Theorem \ref{thm: lax simplicial tensoring and cotensoring}.

\medskip

We consider the projection $p_1:\Delta\times\Delta\rightarrow\Delta$ on the first factor induces a functor $p_1^*=(-)\square\Delta^0$, which forms another Quillen equivalence between the model categories of quasi-categories and bisimplicial sets, see \cite{JoaylTierney07} for more details. This time, we have $(K_{\xi}\circ p_1^*)=K_{\xi}\Delta[1,0] = \xi\mathsf{1}$. The functor $K_{\xi}\circ p_1^*$ shares with $K_{\xi}\circ p_2^*$ the crucial properties we need:
\begin{itemize}
\item[$\bullet$] it is a left Quillen functor, fitting in a Quillen equivalence between the Joyal model structure on simplicial sets and the Barwick--Kan model structure of relative categories;

\item[$\bullet$] it is lax symmetric monoidal by exactly the same proof as for Lemma \ref{lemma:laxmonL}: the sole modification is that one replaces the sentence "the maximal functor $(-)^{\simeq}: \mathsf{Cat} \longrightarrow \mathsf{RelCat}$ preserves cartesian products" by "the minimal functor $\mathsf{Cat} \longrightarrow \mathsf{RelCat}$, given by sending a category to the relative category with no weak equivalences, preserves cartesian products".
\end{itemize}

From this, the arguments above apply here as well: the functor $(-)\times (K_{\xi}\circ p_1^*)(-)$ is a left Quillen bifunctor satisfying the appropriate adjunction relation with $(-)^{(K_{\xi}\circ p_1^*)(-)}$. This follows from the adjunction relation between tensor product and internal hom in the closed symmetric monoidal category $\mathsf{RelCat}$, and the pushout-product and pullback-corner properties are satisfied by construction as well.

\begin{proposition}\label{prop: evaluationfibrations}\leavevmode
\begin{itemize}
\item[(1)] For any fibrant relative category $\mathsf{C}$, the map  $(\mathrm{ev}_0,\mathrm{ev}_1):\mathsf{C}^{\xi\mathsf{1}}\rightarrow \mathsf{C}\times \mathsf{C}$ is a fibration in $\mathsf{RelCat}$.

\item[(2)] For any complete Segal space $X$, the map  $(\mathrm{ev}_0,\mathrm{ev}_1):X^{\Delta[1,0]}\rightarrow X\times X$ is a fibration in $\mathsf{bisSet}$.
\end{itemize}
\end{proposition}

\begin{proof}
    In both cases, it follows from the the following argument, given in the proof of Proposition \ref{prop: goodpath}. The morphism $((d^0)^*,(d^1)^*): A^{\Delta^1}\rightarrow A\times A$ is the cotensoring of $\mathrm{id}$ with the cofibration $\partial\Delta^1\rightarrow \Delta^1$, so it forms a fibration still  by the pullback-corner axiom.
\end{proof}

\renewcommand{\thesection}{B}

\section{Appendix. Fibrations of quasi-categories}\label{Appendix B: fibrations of quasi-categories}

In this appendix, we give a characterization of fibrations of quasi-categories in the Joyal model structure which will be useful in the proof of Theorem \ref{thm: key theorem}. See Subsection \ref{S: the gluing proof} for more details.

\begin{notation}
If $F$ is a functor  between quasi-categories, we denote by $\ho{F}$ the corresponding functor between homotopy categories, which we implictly view again as quasi-categories via the nerve functor.
\end{notation}

\begin{lemma}\label{pullbackff}
Let $\mathsf{C},\mathsf{D}$ be two quasi-categories. A functor $F:\mathsf{C}\rightarrow\mathsf{D}$ is fully faithful if and only if the square
\[
\xymatrix{
\mathsf{C}\ar[r]^-F\ar[d] & \mathsf{D}\ar[d] \\
\ho{\mathsf{C}}\ar[r]_-{\ho{F}} & \ho{\mathsf{D}}
}
\]
is a pullback of quasi-categories.
\end{lemma}
\begin{proof}
See \cite[Remark 1.2.11.1]{Lurie09}.
\end{proof}

\begin{lemma}\label{lemma:nerveofhF}
Let $F:\mathsf{C}\rightarrow\mathsf{D}$ be a functor between two quasi-categories. If $\ho{F}$ is an isofibration and is full, then $\ho{F}$ is a fibration of quasi-categories.
\end{lemma}
\begin{proof}
A fibration in the Joyal model structure with a quasi-category as the codomain is characterized as an inner fibration whose induced 1-functor between homotopy categories is an isofibration. Since the homotopy category of the nerve of a 1-category is isomorphic to the category itself, the functor $\ho{\ho{F}}\cong \ho F$ is an isofibration by assumption, so we just have to prove that $\ho{F}$ is an inner fibration. For this, let us consider a commutative square of the form
\[
\begin{tikzcd}
\Lambda^n_k 
  \arrow[d, hook] 
  \arrow[r] 
& \ho{\mathsf{C}}
  \arrow[d, "hF"] \\
\Delta^n 
  \arrow[r] 
  \arrow[ur, dashed] 
& \ho{\mathsf{D}}
\end{tikzcd}
\]
We want to show the existence of a lift represented by the dotted arrow. Using the adjunction $\ho{-}: \mathsf{Cat} \rightleftarrows \mathsf{sSet}: \mathcal{N}$ between the nerve functor and the homotopy category functor, it is equivalent to proving the existence of a lift in the commutative square of categories
\[
\begin{tikzcd}
\ho{\Lambda^n_k}
  \arrow[d, hook] 
  \arrow[r] 
&\ho{\mathsf{C}} 
  \arrow[d, "hF"] \\
\ho{\Delta^n} 
  \arrow[r] 
  \arrow[ur, dashed] 
&\ho{\mathsf{D}}
\end{tikzcd}
\]
The horizontal functors $\ho{\Delta^n} \longrightarrow \ho{\mathsf{D}}$ and $\ho{\Lambda^n_k} \longrightarrow \ho{\mathsf{C}}$ are diagrams and the left vertical map is just an inclusion of categories, so the existence of the dotted lift amounts to determine whether an $\ho{\Delta^n}$-shaped diagram in $\ho{\mathsf{D}}$ whose $\ho{\Lambda^n_k}$-shaped subdiagram is in the image of $\ho{F}$ is itself in the image of $\ho{F}$. This holds true because the categories $\ho{\Lambda^n_k}$ and $\ho{\Delta^n}$ have the same objects and the functor $\ho{F}$ is full.
\end{proof}

\begin{proposition}\label{prop:fibrationff}
Let $F:\mathsf{C}\rightarrow\mathsf{D}$ be a functor between quasi-categories satisfying the following properties:
\begin{itemize}
\item[(i)] $\ho{F}$ is an isofibration;

\item[(ii)] $F$ is fully faithful.

\end{itemize}
Then the functor $F$ is a fibration of quasi-categories.
\end{proposition}
\begin{proof}
Let $F$ be a functor satisfying these assumptions. By Lemma \ref{lemma:nerveofhF}, assumptions (i) and (ii) imply that $\ho{F}$ is a fibration of quasi-categories. By Lemma \ref{pullbackff}, $F$ is the pullback of $\ho{F}$, and fibrations are stable under pullbacks, so $F$ is also a fibration. 
\end{proof}

\bibliographystyle{alpha}
\bibliography{bibax}
\end{document}